%% file: hedm.tex
\documentclass{amsart}

\usepackage{amssymb}
\usepackage{graphicx}
\usepackage{subfigure}
\usepackage{subfloat}
\usepackage[abs]{overpic}
\usepackage{framed}
\usepackage{mathrsfs}
\usepackage{enumerate}
\usepackage[backref]{hyperref}
\hypersetup{
  colorlinks   = true,    
  urlcolor     = blue,    
  linkcolor    = blue,    
  citecolor    = red,     
}

\newtheorem{theorem}{Theorem}[section]
\newtheorem{lemma}[theorem]{Lemma}
\newtheorem{corollary}[theorem]{Corollary}
\newtheorem{proposition}[theorem]{Proposition}

\theoremstyle{definition}
\newtheorem{definition}[theorem]{Definition}

\newtheorem{assumption}[theorem]{Assumption}

\theoremstyle{remark}
\newtheorem{remark}[theorem]{Remark}

\numberwithin{equation}{section}
\allowdisplaybreaks
\DeclareMathOperator*{\argmin}{arg\,min}

\begin{document}

\title{Hypoelliptic Diffusion Maps I: Tangent Bundles}



\author{Tingran Gao}
\address{Department of Mathematics, Duke University, Durham, NC 27708-0320}
\email{trgao10@math.duke.edu}


\date{\today}


\begin{abstract}
We introduce the concept of \emph{Hypoelliptic Diffusion Maps} (HDM), a framework generalizing \emph{Diffusion Maps} in the context of manifold learning and dimensionality reduction. Standard non-linear dimensionality reduction methods (e.g., LLE, ISOMAP, Laplacian Eigenmaps, Diffusion Maps) focus on mining massive data sets using weighted affinity graphs; Orientable Diffusion Maps and Vector Diffusion Maps enrich these graphs by attaching to each node also some local geometry. HDM likewise considers a scenario where each node possesses additional structure, which is now itself of interest to investigate. Virtually, HDM augments the original data set with attached structures, and provides tools for studying and organizing the augmented ensemble. The goal is to obtain information on individual structures attached to the nodes and on the relationship between structures attached to nearby nodes, so as to study the underlying manifold from which the nodes are sampled. In this paper, we analyze HDM on tangent bundles, revealing its intimate connection with sub-Riemannian geometry and a family of hypoelliptic differential operators. In a later paper, we shall consider more general fibre bundles.
\end{abstract}

\maketitle

\tableofcontents

\section{Introduction}
\label{sec:introduction}
\input{intro.tex}

\section{Motivating The Fibre Bundle Assumption}
\label{sec:fibre-bundle-assumpt}
\input{fbassum.tex}

\section{Hypoelliptic Diffusion Maps: The Formulation}
\label{sec:hypo-diff-formulation}
\input{formulation.tex}

\section{HDM on Tangent and Unit Tangent Bundles}
\label{sec:hypo-diff-maps-tangent}
\input{hdm_tangent.tex}

\section{Numerical Experiments and the Riemannian Adiabatic Limits}
\label{sec:numer-exper}
\input{numerical.tex}

\section{Discussion and Future Work}
\label{sec:disc-future-work}
\input{discussion.tex}

\appendix
\section{The Geometry of Tangent Bundles}
\label{app:geom-tang-bundl}
\input{tangent_bundle.tex}

\section{Proofs of Theorem~\ref{thm:main_tm}, ~\ref{thm:main_utm}, ~\ref{thm:utm_finite_sampling_noiseless}, and~\ref{thm:utm_finite_sampling_noise}}
\label{app:proof-main-theorem}
\input{proof.tex}

\bibliographystyle{plain}
\bibliography{hedm}

\end{document}

%% file: intro.tex
Acquiring complex, massive, and often high-dimensional data sets has become a common practice in many fields of natural and social sciences; while inspiring and stimulating, these data sets can be challenging to analyze or understand efficiently. To gain insight despite the volume and dimension of the data, methods from a wide range of science fields have been brought into the picture, rooted in statistical inference, machine learning, signal processing, to mention just a few.

Among the exploding research interests and directions in data science, the relation between the graph Laplacian~\cite{Chung1997} and the manifold Laplacian~\cite{Rosenberg1997Laplacian} has emerged as a useful guiding principle. Specifically, the field of \emph{non-linear dimensionality reduction} has witnessed the emergence of a variety of Laplacian-based techniques, such as Locally Linear Embedding (LLE)~\cite{LLE2000}, ISOMAP~\cite{ISOMAP2000}, Hessian Eigenmaps~\cite{HessianLLE2003}, Local Tangent Space Alignment (LTSA)~\cite{LTSA2005}, Diffusion Maps~\cite{CoifmanLafon2006}, Orientable Diffusion Maps (ODM)~\cite{SingerWu2011ODM}, Vector Diffusion Maps (VDM)~\cite{SingerWu2012VDM}, and Schr\"odinger Eigenmaps~\cite{SSSE2014}. The general practice of these methods is to treat each object in the data set (these objects could be images, texts, shapes, etc.) as an abstract node or vertex, and form a similarity graph by connecting each pair of similar nodes with an edge, weighted by their similarity score. Built with varying flexibility, these methods provide valuable tools for organizing complex networks and data sets by ``learning'' the global geometry from the local connectivity of weighted graphs.

The Diffusion Map (DM) framework~\cite{CoifmanLafon2006,LafonThesis2004,CoifmanLafonLMNWZ2005PNAS1,CoifmanLafonLMNWZ2005PNAS2,CoifmanMaggioni2006,SingerWu2011ODM,SingerWu2012VDM} proposes a probabilistic interpretation for graph-Laplacian-based dimensionality reduction algorithms. Under the assumption that the discrete graph is appropriately sampled from a smooth manifold, it assigns transition probabilities from a vertex to each of its neighbors (vertices connected to it) according to the edge weights, thus defining a graph random walk the continuous limit of which is a diffusion process~\cite{WatanabeIkeda1981SDE,Durrett1996SC} over the underlying manifold. The eigenvalues and eigenvectors of the graph Laplacian, which converge to those of the manifold Laplacian under appropriate assumptions~\cite{BelkinNiyogi2005,BelkinNiyogi2007}, then reveal intrinsic information about the smooth manifold. More precisely, ~\cite{BBG1994} proves that these eigenvectors embed the manifold into an infinite dimensional $l^2$ space, in such a way that the \emph{diffusion distance}~\cite{CoifmanLafon2006} (rather than the geodesic distance) is preserved. Appropriate truncation of these sequences leads to an embedding of the smooth manifold into a finite dimensional Euclidean space, with small metric distortion.

Under the manifold assumption, ~\cite{SingerWu2011ODM,SingerWu2012VDM} recently observed that estimating random walks and diffusion processes on structures associated with the original manifold (as opposed to estimates of diffusion on the manifold itself) are able to handle a wider range of tasks, or obtain improved precision or robustness for tasks considered earlier. For instance, ~\cite{SingerWu2011ODM} constructed a random walk on the \emph{orientation bundle}~\cite[\S I.7]{BottTu1982} associated with the manifold, and translated the detection of orientability into an eigenvector problem, the solution of which reveals the existence of a global section on the orientation bundle; ~\cite{SingerWu2012VDM} introduced a random walk on the \emph{tangent bundle} associated with the manifold, and proposed an algorithm that embeds the manifold into an $l^2$ space using eigen-vector-fields instead of eigenvectors (and thus the name Vector Diffusion Maps (VDM)). In~\cite{Wu2013VDMEmbedding} the VDM approach is used, analogously to ~\cite{BBG1994}, to embed the manifold into a finite dimensional Euclidean space. Although the VDM embedding does not reduce the dimensionality as much as standard diffusion embedding methods, it benefits from improved robustness to noise, as illustrated by the analysis of some notoriously noisy data sets~\cite{KarouiWu2013,KarouiWu2014}.

Both ~\cite{SingerWu2011ODM} and ~\cite{SingerWu2012VDM} incorporate additional structures into the graph Laplacian framework: in ~\cite{SingerWu2012VDM} this is an extra orthogonal transformation (estimated from local tangent planes) attached to each weighted edge in the graph; in ~\cite{SingerWu2011ODM} the edge weights are overwritten with signs determined by this orthogonal transformation. These methods are successful because they incorporate more local geometry in the path to dimensionality reduction, by estimating tangent planes. In fact, the advantage of utilizing local geometric information from the tangent bundle had been noticed earlier: Fig.~\ref{fig:diff_map_tangent} shows a simple example, borrowed from~\cite[\S 2.6.1]{LafonThesis2004}, where the original data set (shown in Fig.~\ref{fig:diff_map_tangent}(a)) is a Descartes Folium with self-intersection at the origin, parametrized by
\begin{equation*}
  x \left( \theta \right) = \frac{3\tan\theta}{1+\tan^3\theta}, \quad y \left( \theta \right) = \frac{3\tan^2\theta}{1+\tan^3\theta},\quad \theta\in \left[ -\frac{\pi}{2},\frac{\pi}{2} \right].
\end{equation*}
This curve is the projection onto a plane of a helix in $\mathbb{R}^3$. A standard isotropic random walker on the planar curve would get lost at the intersection, even when sober, as shown in Fig.~\ref{fig:diff_map_tangent}(b), where the embedding completely mixes blue and red tails beyond the crossing point. In contrast, incorporating tangent information into local similarity scores yields a much more clear embedding back to $\mathbb{R}^3$ (see Fig.~\ref{fig:diff_map_tangent}(c)), which \emph{blows up} (in the sense of complex algebraic geometry~\cite[pp.182]{GriffithsHarris2011}) the self-intersecting curve at its singularity and unraveled its hidden geometry. Specifically, the similarity measure used in the modified diffusion map between any pair of points $\left( x \left( \theta_1 \right),y \left( \theta_1 \right) \right)$ and $\left( x \left( \theta_2 \right),y \left( \theta_2 \right) \right)$ on the curve is
\begin{equation*}
  \begin{aligned}
    d\left( \left(x \left( \theta_1 \right),y \left( \theta_1 \right)\right), \left(x \left( \theta_2 \right),y \left( \theta_2 \right) \right)\right)^2&=\left\|\left(x \left( \theta_1 \right), y \left( \theta_1 \right)\right)-\left(x \left( \theta_2 \right), y \left( \theta_2 \right)\right)\right\|_2^2\\
    &+\mu\left\|\frac{\left( x' \left( \theta_1 \right),y' \left( \theta_1 \right) \right)}{\left\| \left( x' \left( \theta_1 \right),y' \left( \theta_1 \right) \right)  \right\|_2}-\frac{\left( x' \left( \theta_2 \right),y' \left( \theta_2 \right) \right)}{\left\| \left( x' \left( \theta_2 \right),y' \left( \theta_2 \right) \right) \right\|_2} \right\|_2^2,
  \end{aligned}
\end{equation*}
where $\mu>0$ is a parameter that balances the two contributions to the dissimilarity score in consideration. (Two distinct tangent vectors exist at the self-intersection, but they each belong to a distinct point in the parametrization.)
\begin{figure}[htps]
  \centering
  \includegraphics[width=1.0\textwidth]{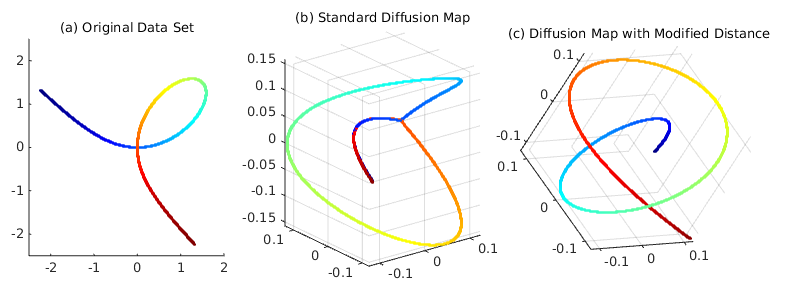}
  \caption{A Diffusion Map Incorporating Local Geometric Information}
  \label{fig:diff_map_tangent}
\end{figure}
It is possible to use the methodology of ODM and VDM to tackle similar problems in much broader contexts, where the local geometric information can be of a different type than information about tangent planes. Indeed, for many data sets, a single data point has abundant structural details; typically graph-Laplacian-based methods begin by ``abstracting away'' these details, encoding only pairwise similarites. In some circumstances, the hidden details (pixels in an image, vertices/faces on a triangular mesh, key words and transition sentences in a text, etc.) may themselves be of interest. For example, in the geometry processing problem of analyzing large collections of 3D shapes, it is desirable to enable user exploration of shape variations across the collection. In this case, abstracting each single shape as a graph node completely ignores the spatial configuration of an individual shape. On the other hand, even when sticking to pairwise similarity scores significantly simplifies the data manipulation, the best way to score similarity is not always clear. In practice, the similarity measure is often dictated by practical heuristics, which may be misguided for incompletely understood data. In addition, there are situations for which it can be proved that no finite-dimensional representation will do justice to the data. (For instance, in topological data analysis of shapes and surfaces, the only known sufficient statistics (other than the data set itself) is the set of all persistent diagrams taken from all directions~\cite{TMB2014}.)

In this paper, we propose the \emph{Hypoelliptic Diffusion Map} (HDM), a new graph-Laplacian-based framework for analyzing complex data sets. This method focuses on data sets in which pairwise similarity between data points is not sufficiently informative, but each single data point carries sophisticated \emph{individual structure}. In practice, this type of data set often arises when the data acquired is too noisy, has huge degrees of freedom, or contains un-ordered features (as opposed to sequential data). An example that has all these characteristics is, e.g., a data set in which each data point is a two-dimensional surface in $\mathbb{R}^3$, represented either by a triangular mesh or a collection of persistent diagrams. In many cases, computing pairwise similarity within such data sets requires minimizing some functional over the space of admissible pairwise correspondences, and the similarity score between two surfaces is achieved by a certain optimal correspondence map between the surfaces. It is conceivable that the optimal correspondence contains substantial information, missing from the condensed similarity score. The HDM framework is our first attempt at mining this hidden information from correspondences.

Like ODM and VDM, HDM generalizes the DM framework, but it takes an essentially different path. We are most interested in the scenario in which the individual structures themselves are also manifolds. In order to take them into consideration, we first augment the manifold underlying DM, denoted as $M$, with extra dimensions. To each point $x$ on $M$, this augmentation attaches the individual manifold at $x$, denoted as $F_x$; we assure that around each $x\in M$ there exists an open neighborhood $U$ such that on $U$ the augmented structure ``looks like'' a product of $U$ with a ``universal template'' manifold $F$. Intuitively, $M$ plays the role of a ``parametrization'' for all the $F_x$. Of course, the existence of such a universal template makes sense only if the $F_x, x\in M$ are compatible in some appropriate sense (each $F_x$ should at least be diffeomophic to $F$; we shall add more restrictions below); however, such compatibility is not uncommon for many data sets of interest, as we shall see in Section \ref{sec:fibre-bundle-assumpt}. This picture of parametrizing a family of manifolds with an underlying manifold is reminiscent of the modern differential geometric concept of a \emph{fibre bundle}, which played an important role in the development of geometry, topology, and mathematical physics in the past century. Therefore, we shall refer to this geometric object as the underlying \emph{fibre bundle} of the data set. Adopting the terminology from differential geometry, we call $M$ the \emph{base manifold}, the universal template manifold $F$ the \emph{fibre}, and each $F_x$ a \emph{fibre at $x$}.

The probabilistic interpretation of HDM is a random walk on the fibre bundle. In one step, the transition occurs either between points on adjacent but distinct fibres, or within the same fibre. Since the fibre bundle is itself a manifold (referred to as the \emph{total manifold}, denoted as $E$), this looks so far no different from a direct application of DM, only on an augmented geometric object. However, HDM also incorporates the pairwise correspondences of data points in the fibre bundle formulation, by requiring transitions between distinct fibres to satisfy certain directional constraints imposed by the correspondences. The resulting random walk is no longer a direct analogy of its standard counterpart on the total manifold, but rather a ``lift'' of a random walk on the base manifold $M$. Under mild assumptions, its continuous limit is a diffusion process on the total manifold $E$, infinitesimally generated by a \emph{hypoelliptic differential operator}~\cite{Hoermander1967} (thus the name HDM). We can then embed the whole fibre bundle into a Euclidean space using the eigenvectors of this hypoelliptic differential operator; discretely this corresponds to solving for the eigenvectors of our new graph Laplacian, referred to as a \emph{hypoelliptic Laplacian of the graph}. It turns out that, by varying a couple of parameters in its construction, the family of graph hypoelliptic Laplacians contains the discrete analogue of several important and informative partial differential operators on the fibre bundle, relating the geometry of the base manifold with that of the total manifold. Our numerical experiments revealed interesting phenomena when embedding the fibre bundle using eigenvectors of these new graph Laplacians.

Though the HDM framework applies to general fibre bundles, the focus of this paper is the study of tangent and unit tangent bundles of Riemannian manifolds; in a sequel paper we shall study more general fibre bundles. Note that even though the fibre bundles in this paper are the same as for VDM, HDM for tangent bundle nevertheless differs from VDM; we shall come back to this below.

This paper is organized as follows: Section~\ref{sec:fibre-bundle-assumpt} sets up notations and terminology, and discusses the meaning of the \emph{fibre bundle assumption}; Section~\ref{sec:hypo-diff-formulation} describes the formulation of HDM in detail; Section~\ref{sec:hypo-diff-maps-tangent} characterizes the hypoelliptic graph Laplacians on tangent and unit tangent bundles, and studies their pointwise convergence from finite samples; some numerical experiments are shown in Section~\ref{sec:numer-exper}; finally we conclude with a brief discussion and propose potentially interesting directions for future work. In Appendix~\ref{app:geom-tang-bundl} we include preliminaries on the geometry of tangent bundles and (as their subbundles) unit tangent bundles.


%% file: fbassum.tex
For high-dimensional data generated by some implicit process with relatively fewer degrees of freedom, it is often reasonable to assume that the data lie approximately on a manifold of much lower dimension than the ambient space. In the literature on semi-supervised learning, this is often referred to as the \emph{manifold assumption}~\cite{BelkinNiyogi2004ML,Zhu05Survey}. The goal of semi-supervised learning is to build a classifier based on a partially labeled training set; learning the underlying manifold structure of high-dimensional data is often the first step in this practice, not only because it reduces the dimensionality, but also due because it simplifies the data and exposes the structure.

Our \emph{fibre bundle assumption} is a generalization of the manifold assumption. In differential geometry, a fibre bundle is a manifold itself, that is structured as a family of related manifolds parametrized by another underlying manifold. Following~\cite{Steenrod1951}, a fibre bundle consists of the following data\footnote{Strictly speaking, the definition given here is that of a \emph{coordinate bundle}~\cite[\S 2.3]{Steenrod1951}; fibre bundles are equivalence classes of coordinate bundles. This distinction is less crucial since in the HDM framework we describe the structure of a fibre bundle using coordinates. This is similar to how manifold learning uses the notion of a manifold.}:
\begin{enumerate}
\item\label{item:1} the \emph{total manifold} $E$;
\item\label{item:2} the \emph{base manifold} $M$;
\item\label{item:3} the \emph{bundle projection} $\pi$, a surjective smooth map from $E$ onto $M$;
\item\label{item:4} the \emph{fibre manifold} $F$, satisfying
\begin{enumerate}
\item\label{item:5} for any $x\in M$, $\pi^{-1}\left( x \right)$ is diffeomorphic to $F$;
\item\label{item:6} for any $x\in M$, there exists an open neighborhood $U$ of $x$ in $M$ and a diffeomorphism $\phi_U$ from $\pi^{-1}\left( U \right)$ to $U\times F$;
\end{enumerate}
\item\label{item:7} the \emph{structure group} $G$, a topological transformation group that acts effectively\footnote{$G$ acts effectively on $F$ if $g \left( f \right)=f$ for all $f\in F$ implies $g=e$, the identity element of $G$} on $F$, satisfying
\begin{enumerate}
\item\label{item:8} for any $x\in M$ and two open neighborhoods $U$ and $V$ that both satisfy \eqref{item:6}, the diffeomorphism on $F$, defined as ``freezing the first component as $x$'', obtained from $\phi_V\circ\phi_{U}^{-1}$ as
$$g^x_{UV}:=\left[\phi_V\circ\phi_U^{-1}\right]\left( x,\cdot \right):F\rightarrow F,$$
is an element $g^x_{UV}$ in $G$, and this correspondence $$x\mapsto g^x_{UV}$$ is continuous with respect to the topology on $G$;
\item\label{item:9} for any $x\in M$ and three open neighborhoods $U,V,W$ that all satisfy \eqref{item:6},
$$g^x_{UU}=\textrm{the identity element $e$ of $G$}$$ and $$g^x_{UV}\circ g^x_{VW}=g^x_{UW}.$$
\end{enumerate}
\end{enumerate}

The diffeomorphisms in~\eqref{item:6} are also known as \emph{local trivializations}. For each $x$ on the base manifold $M$, it is conventional to denote the \emph{fibre over $x$} as $F_x:=\pi^{-1}\left( x \right)$. The \emph{fibre bundle assumption} can now be stated as follows:
\begin{framed}
\begin{assumption}[The Fibre Bundle Assumption]
  The data lie approximately on a fibre bundle, in the sense that each data object is a subset of a fibre over some point on a base manifold.
\end{assumption}
\end{framed}
Note that in the special case where the fibre manifold $F$ is a single point, the fibre bundle is diffeomorphic to its base manifold, and our fibre bundle assumption reduces to the manifold assumption.

The definition of a fibre bundle is technical, especially for the part involving the structure group $G$. The key point is that a fibre bundle is locally a product manifold, and these local pieces are carefully patched together so that the product structures remain consistent when they intersect. Product manifolds are thus fibre bundles by definition, but the concept of a fibre bundle becomes interesting only when the global geometry gets twisted and exposes non-trivial topology. The M\"obius band, the Klein bottle, and the Hopf fibration are standard illustrations of this; see e.g. ~\cite[\S  1]{Steenrod1951}.

At a first glance, the fibre bundle assumption imposes strong restrictions on the data set structure. However, when understanding the structure of individual data points is equally as interesting as understanding the structure of the data set in the large, the framework based on the manifold assumption becomes insufficient. For instance, in geometric morphormetrics~\cite{Zelditch2004}, the data sets of interest are collections of \emph{shapes}, i.e., two-dimensional smooth surfaces in $\mathbb{R}^3$, and the central problem is to infer species and other biological information from shape variations. Under the assumption that these variations are governed by relatively few degrees of freedom, it is possible to learn manifold coordinates for each shape in the collection (e.g., applying the diffusion map to the shape collection based on some pairwise shape-distance, e.g.,~\cite{RangarajanChuiBookstein1997,Memoli2008,GhoshSharfAmenta2009,Mitteroecker2009,LipmanDaubechies2011,LipmanPuenteDaubechies13,CP13}). Yet it is difficult to infer shape variation from such coordinates, since the geometry of each individual shape is ``abstracted away'', collapsing each shape to a single point. To add interpretability to the manifold learning framework in this circumstance, it is a natural idea to learn different coordinates for distinct points on the same shape, and simultaneously keep similar the coordinates of points belonging to different shapes that are developmentally or functionally equivalent. This geometric intuition is embodied by the fibre bundle assumption. From this point of view, the fibre bundle assumption is but \emph{an extra level of indirection} (borrowing a term from Andrew Koenig's ``fundamental theorem of software engineering'') for the manifold assumption.

The example of shape analysis in geometric morphometics is particularly interesting, because it contains another source of ideas that naturally models the data set as a fibre bundle: the global registration problem. Geometric morphometricians typically select equal numbers of \emph{homologous landmarks} on each shape in a globally consistent manner, then reduce the analysis to investigation of the \emph{shape space}~\cite{Kendall1984,Kendall1989} of these landmark points. Along these lines, tools like the \emph{Generalized Procrustes Analysis} (GPA)~\cite{Goodall1991Procrustes,DrydenMardia1993,Kent1994,GowerDijksterhuis2004GPA} have been developed in statistical shape analysis~\cite{DrydenMardia1998SSA}, and software products~\cite{OHigginsJones1998,Evan2010} made available. (Recent progress in this area~\cite{Nemirovski2007,So2011,NaorRegevVidick2013,BandeiraKennedySinger2013,ChaudhuryKhooSinger2013} relates \emph{semidefinite programming} with the \emph{little Grothendieck problem}.) A common basis for these GPA-based methods is that the homology of landmark points depends on human input. Manually placing landmarks on each shape among a large collection is a tedious task, and the skill to perform it correctly typically requires years of professional training. Recently, automated methods have been proposed in this field, based on efficient and robust pairwise surface comparison algorithms~\cite{PNAS2011,LipmanDaubechies2011,LipmanPuenteDaubechies13,CP13,PuenteThesis2013,PuenteBoyerGlaDaubechies2013}. However, biological morphologists typically do not compare surfaces merely pairwise: in practice, an experienced morphologist uses a large database of anatomical structures to improve the consistency and accuracy of visual interpretation of biological features. This consistency can not be trivially achieved by any geometric algorithm that uses only pairwise comparison information, even when each pairwise comparison is of remarkably high quality. This is shown in Fig.\ref{fig:nonflat_connection}, where a small set of landmarks is propagated from a \emph{Microcebus} molar to a corresponding \emph{Lepilemur} molar, along three different paths. Though all surfaces A through E are fairly similar to each other (and hence the algorithm in \cite{CP13} guarantees high quality pairwise correspondences), direct propagation of landmarks via path $A \!\!\rightarrow\!\! B$ gives a different result from $A \!\!\rightarrow\!\! C\!\!\rightarrow\!\! B$ or $A\!\!\rightarrow\!\! D\!\!\rightarrow\!\! E\!\!\rightarrow\!\! B$. Using the collection $\left\{A,B,C,D,E\right\}$ leads to a more accurate correspondence between $A$ and $B$ then an isolated $A$-$B$ comparison would.
\begin{figure}[htps]
  \centering
  \includegraphics[width=0.6\textwidth]{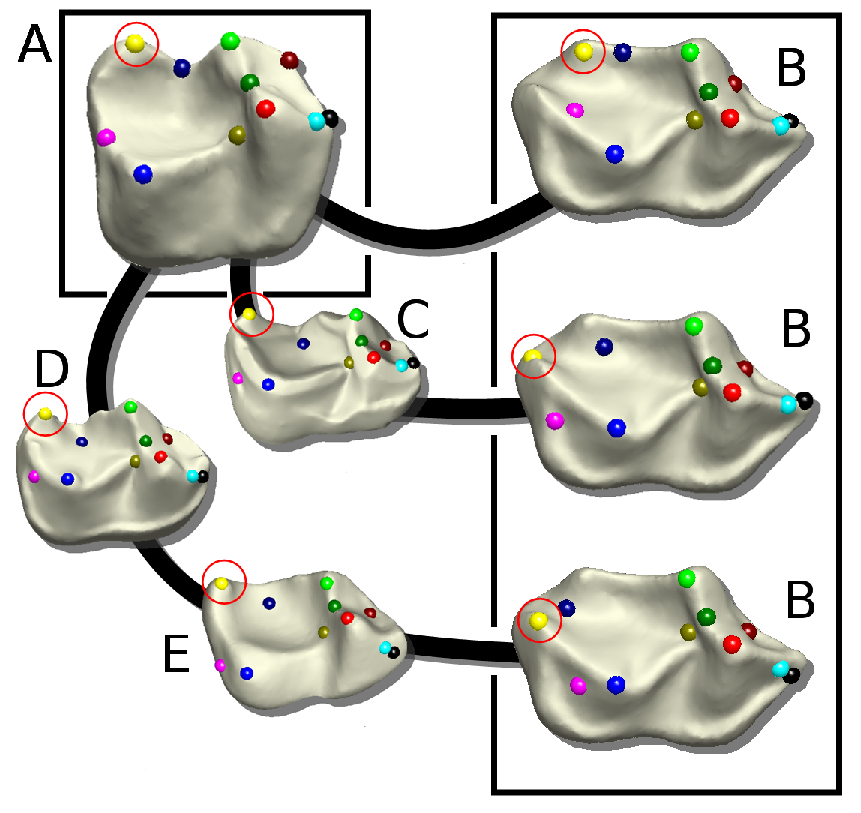}
  \caption{Non-Triviality in Analyzing a Collection of Teeth (c.f.~\cite{PNAS2011})}
  \label{fig:nonflat_connection}
\end{figure}

In the fibre bundle framework, the inherent inconsistency for pairwise-comparison-based global registration can be modeled using the concept of the \emph{holonomy} of \emph{connections}. In the sense of Ehresmann~\cite{Ehresmann1950Connexions}, a \emph{connection} is a choice of splitting the short exact sequence
\begin{equation}
  \label{eq:ses}
  0\rightarrow VE \rightarrow TE \rightarrow \pi^{*}TM \rightarrow 0
\end{equation}
In this short exact sequence, $TE$ is the tangent bundle of the total manifold $E$; $VE$ is the \emph{vertical tangent bundle} of $E$, a subbundle of $TE$ spanned by vectors that are tangent not only to $E$ at some point $u\in E$, but also to the fibre $F_{\pi \left( u \right)}$ over $\pi \left( u \right)\in M$; $\pi^{*}TM$ is the \emph{pullback bundle} of $TM$ to $TE$. The practical meaning of this definition is as follows: since the fibre $F_x$ over $x\in M$ carries manifold structure for itself, the notion of vectors that are ``tangent to the fibre'' is well-defined; they correspond to $VE$. The short exact sequence~\eqref{eq:ses} tells us that the quotient bundle of $TE$ by $VE$ is isomorphic to $\pi^{*}TM$, but there is no canonical way to choose a ``horizontal tangent bundle'' $HE$ for $TE$ such that $$HE\oplus VE = TE.$$
The definition of an Ehresmann connection is just the choice of such a subbundle $HE$. More concretely, a connection specifies for each point $u\in E$ a subspace $H_uE$ of $T_uE$, such that $H_uE$ together with all vertical tangent vectors at $u$ spans the entire tangent space $T_uE$ at $u$. Of course, the choice of subspaces $H_uE$ should depend smoothly on $u$. We shall call vectors in $H_uE$ \emph{horizontal}, while keeping in mind that this concept builds upon the connection.

As long as a connection is given on a fibre bundle, tangent vectors on the base manifold $M$ can always be canonically \emph{lifted} to $E$. That is, for any $u\in E$ and any tangent vector $X_{\pi \left( u \right)}\in T_{\pi \left( u \right)}M$, there exists in $H_uE$ a unique tangent vector $X^{\mathscr{L}}_u\in T_uE$. In fact, this follows immediately from the fact that $HE$ is isomorphic to $\pi^{*}TM$, as implied in the short exact sequence~\eqref{eq:ses}. Moreover, a smooth vector field $X$ on $M$ can be uniquely lifted to $E$, resulting in a vector field $X^{\mathscr{L}}$ on $E$ that is horizontal everywhere. This eventually enables us to lift any smooth curve $\gamma:\mathbb{R}\rightarrow M$ on the base manifold to a \emph{horizontal curve} $\tilde{\gamma}$ on $E$, defined by the ODE
\begin{equation*}
  \frac{d\tilde{\gamma}}{dt}\bigg|_{u \left( t \right)}=\left( \frac{d\gamma}{dt}\bigg|_{\pi \left( u \left( t \right) \right)} \right)^{\mathscr{L}}.
\end{equation*}
Note that the horizontal curve is uniquely determined once its starting point on $E$ is specified. Therefore, given a smooth curve $\gamma:\left[ 0,1 \right]\rightarrow M$ that connects $\gamma \left( 0 \right)$ to $\gamma \left( 1 \right)$ on $M$, there exists a smooth map from $F_{\gamma \left( 0 \right)}$ to $F_{\gamma \left( 1 \right)}$  (at least when $\gamma \left( 0 \right)$ and $\gamma \left( 1 \right)$ are sufficiently close), defined as
\begin{equation*}
  F_{\gamma \left( 0 \right)}\ni s\mapsto \tilde{\gamma}_s\left( 1 \right)\in F_{\gamma \left( 1 \right)},
\end{equation*}
where $\tilde{\gamma}_s$ denotes the horizontal lift of $\gamma$ with starting point $s$. Such constructed maps between neighboring fibres, obviously depending on the choice of path $\gamma$, is called the \emph{parallel transport along $\gamma$}. Like the concept of horizontal tangent vectors, parallel transport depends on the choice of the connection. We shall denote the parallel transport from fibre $F_y$ to fibre $F_x$ as $P^{\gamma}_{xy}:F_y\rightarrow F_x$. When $\gamma$ is a unique geodesic on $M$ that connects $y$ to $x$, we drop the super-index $\gamma$ and simply write $P_{xy}:F_y\rightarrow F_x$. We shall see later that the probabilistic interpretation of HDM (and even VDM) implicitly depends on lifting from the base manifold a path that is continuous but not necessarily smooth. Though this can not be trivially achieved by the ODE based approach, stochastic differential geometry has already prepared the appropriate tools for tackling this technicality (see e.g. ~\cite[\S 5.1.2]{Stroock2005AnalysisPaths}).

We now return to modeling the inherent inconsistency for geometric morphometrics. Similar to the diffusion map framework, where small distances are considered to approximate geodesic distances on the manifold, we assume, when the pairwise distance between surfaces $S_1,S_2$ is relatively small among all pairwise distances within the collection, that the shape distance is approximately equal to the geodesic distance on the base manifold. Moreover, under the fibre bundle assumption, we consider the pairwise correspondence map from $S_1$ to $S_2$ to approximate $P_{S_2,S_1}$, the parallel transport along the geodesic connection $S_1$ to $S_2$. By routing through different intermediates, one obtains different maps from $S_1$ to $S_2$, which is conceptually equivalent to parallel-transporting along different piecewise geodesics. Due to the dependency on the underlying path, the parallel transport typically does not define globally consistent maps. The inconsistency shown in Fig.\ref{fig:nonflat_connection}, caused by propagation along three different paths, fits into this geometric picture.

The inconsistency of parallel transport, closely related to the \emph{curvature} of the corresponding connection \cite{AMS1953}, is characterized by the notion of \emph{holonomy} \cite{Taubes2011DG,Bryant2006holonomy,Besse2007EinsteinManifolds}. If for all $x,y\in M$ the parallel transport $P_{xy}^{\gamma}:F_y\rightarrow F_x$ is independent of the choice of path $\gamma$, then the connection is said to be \emph{flat} or has \emph{trivial} holonomy; otherwise the connection is \emph{non-flat} or the holonomy is \emph{non-trivial}. Fig.\ref{fig:holonomy_on_sphere} illustrates the non-trivial holonomy of the \emph{Levi-Civita connection} on the unit sphere in $\mathbb{R}^3$: if we parallel transport a tangent vector $v\in T_AS^2$, first from $A$ to $C$ along the equator and then from $C$ to $B$ along the meridian, then the result $P_{BC}P_{CA}v$ is generally different from $P_{BA}v$, the result obtained by directly parallel transporting $v$ from $A$ to $B$ along the meridian that connects the two points.
\begin{figure}[htps]
  \centering
  \includegraphics[width=0.5\textwidth]{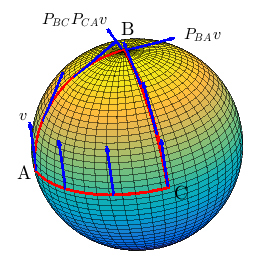}
  \caption{Holonomy on a Unit Sphere}
  \label{fig:holonomy_on_sphere}
\end{figure}

The fibre bundle of interest in Fig.\ref{fig:holonomy_on_sphere} is an example of a \emph{tangent bundle}. Generally, the tangent bundle $TM$ of a $d$-dimensional Riemannian manifold $M$ is a fibre bundle with base manifold $M$, fibre $\mathbb{R}^d$, and structure group $O \left( d \right)$ (the $d$-dimensional orthogonal group); the fibre over each $x\in M$ is $T_xM$, the tangent space of $M$ at $x$. On this bundle, there uniquely exists a canonical connection, the \emph{Levi-Civita connection}, that is simultaneously \emph{torsion-free} and \emph{compatible} with the Riemannian metric on $M$. The \emph{unit tangent bundle} $UTM$ is a subbundle of $TM$, with the same base manifold and structure group, but has a different type of fibre $S^d$, the unit $\left(d-1\right)$-dimensional sphere in $\mathbb{R}^d$; the fibre over each $x\in M$ consists of all tangent vectors of $M$ at $x$ with unit length. The Levi-Civita connection carries over to a canonical connection on $UTM$. We focus on analyzing HDM on these two types of fibre bundles in this paper.

Note that the tangent bundle is also of fundamental importance for VDM. However, as we shall see in Section~\ref{sec:hypo-diff-formulation}, HDM aims at a goal different from VDM's, even on tangent bundles: VDM acts on vector fields on $M$, or equivalently operates on \emph{sections} of $TM$ (denoted as $\Gamma \left( M,TM \right)$); HDM focuses on functions on $TM$, and thus operates on sections of the trivial line bundle $TM\times \mathbb{R}$ (denoted as $\Gamma \left( TM, \mathbb{R} \right)$). While VDM embeds the base manifold $M$ into a Euclidean space of lower dimension, HDM is more interested in how each fibre of $TM$ corresponds to its neighboring fibres. In short, VDM and HDM extend DM in two different directions.

The use of diffusion maps to solve the global registration problem was proposed earlier in the geometry processing community~\cite{SvKKZC2011,Kim12FuzzyCorr}, as was the concept of a ``template'' for a collection of shapes~\cite{WangHuangGuibas2013,SolomonBenChenButscherGuibas2011,NguyenBWYG2011,HuangZhangGHBG2012,HuangGuibas2013,CGH2014}. These approaches were quite successful, albeit based mostly on heuristics; the fibre bundle framework provides geometric interpretations and insights for many of them. For instance, cycle-consistency-based approaches ~\cite{NguyenBWYG2011,HuangGuibas2013} focus on improving the consistency of composed correspondence maps along $1,2,3$-cycles, which is implicitly an attempt to recover from condition~\eqref{item:9} the fibre bundle structure that underlies the shape collection; from this point of view, these method sample only one point from each coordinate patch on the base manifold, and likely suffer from an inaccurate recovery due to low sampling rate. \cite{Kim12FuzzyCorr} uses the diffusion map as a visualization tool, based on a dissimilarity score computed from local and global shape alignments. This is similar to the random walk HDM constructs on a fibre bundle; the geometric meaning of the fuzzy correspondence score in \cite{Kim12FuzzyCorr} is vague from a manifold learning point of view, but then, it was not the main focus of \cite{Kim12FuzzyCorr} to analyze the new graph Laplacian on the discretized fibre bundle.

From the fibre bundle point of view, the goal of many global registration problems is to learn the fibre bundle structure that underlies the collection of objects. Making an analogy with the terminology \emph{manifold learning}, we call this type of learning problems \emph{fibre learning}. A flat connection, or its induced parallel transport, is the key to resolving the problem. However, we remark that the existence of a flat connection on an arbitrary fibre bundle is not guaranteed: the geometry and topology of the fibre bundle may be an obstruction. For a discussion on tangent bundles, see~\cite{Milnor1958,Goldman2011}.


%% file: formulation.tex
\subsection{Basic Setup}
\label{sec:basic-setup}
The data set considered in the HDM framework is a triplet $\left( \mathscr{X},\rho,G \right)$, where
\begin{enumerate}
\item\label{item:10} The \emph{total} data set $\mathscr{X}$ is formed by the union
  \begin{equation*}
    \mathscr{X}=\bigcup_{j=1}^n X_j
  \end{equation*}
where each subset $X_j$ is referred to as the $j$-th \emph{fibre} of $\mathscr{X}$, containing $\kappa_j$ points
\begin{equation*}
  X_j=\left\{ x_{j,1},x_{j,2},\cdots,x_{j,\kappa_j} \right\}.
\end{equation*}
We call the collection of fibres the \emph{base} data set
\begin{equation*}
  \mathscr{B}=\left\{ X_1,X_2,\cdots,X_n \right\},
\end{equation*}
and let $\pi:\mathscr{X}\rightarrow\mathscr{B}$ be the \emph{canonical projection} from $\mathscr{X}$ to $\mathscr{B}$
\begin{equation*}
  \begin{aligned}
    \pi:\mathscr{X} & \longrightarrow \mathscr{B}\\
    x_{j,k} & \longmapsto X_j,\quad 1\leq j\leq n, 1\leq k\leq \kappa_j.
  \end{aligned}
\end{equation*}
We shall denote the total number of points in $\mathscr{X}$ as
\begin{equation*}
  \kappa=\kappa_1+\kappa_2+\cdots+\kappa_n.
\end{equation*}

\item\label{item:11} The \emph{similarity measure} $\rho$ is a real-valued function on $\mathscr{X}\times \mathscr{X}$, such that for all $\xi,\eta\in \mathscr{X}$
\begin{equation*}
  \begin{aligned}
    \rho \left( \xi, \eta\right)\geq 0,\quad \rho \left( \xi,\xi \right)=0,\quad \rho \left( \xi, \eta \right) = \rho \left( \eta, \xi \right).
  \end{aligned}
\end{equation*}
On the product set $X_i\times X_j$, we denote
\begin{equation*}
  \rho_{ij}\left( s,t \right)=\rho \left( x_{i,s},x_{j,t} \right);
\end{equation*}
then $\rho_{ij}$ is an $\kappa_i\times \kappa_j$ matrix on $\mathbb{R}$, to which we will refer as the \emph{similarity matrix} between $X_i$ and $X_j$.

\item\label{item:12} The \emph{affinity graph} $G= \left( V,E \right)$ has $K$ vertices, with each $v_{i,s}$ corresponding to a point $x_{i,s}\in \mathscr{X}$; without loss of generality, we shall assume $G$ is connected. (In our applications, each $x_{i,s}$ is typically connected to several $x_{j,t}$'s on neighboring fibres.) If there is an edge between $v_{i,s}$ and $v_{j,t}$ in $G$, then $x_{i,s}$ is a \emph{neighbor} of $x_{j,t}$ and $x_{j,t}$ is a neighbor of $x_{i,s}$. Moreover, we also call $X_i$ a \emph{neighbor} of $X_j$ (and similarly $X_j$ a neighbor of $X_i$) if there is an edge in $G$ linking one point in $X_i$ with one point in $X_j$; this terminology implicitly defines a graph $G_B = \left( V_B,E_B \right)$, where vertices of $V_B$ are in one-to-one correspondences with fibres of $\mathscr{X}$, and $E_B$ encodes the neighborhood relations between pairs of fibres. $G_B$ will be called as the \emph{base affinity graph}.
\end{enumerate}

\subsection{Graph Hypoelliptic Laplacians}
\label{sec:graph-hypo-lapl}

Let $W\in\mathbb{R}^{\kappa\times \kappa}$ be the \emph{weighted adjacency matrix} of the graph $G$, i.e., $W$ is a block matrix in which the $\left( i,j \right)$-th block
\begin{equation}
  \label{eq:hedm_W}
  W_{ij}=\rho_{ij}.
\end{equation}
The $\left( s,t \right)$ entry in $W_{ij}$ is thus the edge weight $\rho_{ij}\left( s,t \right)$ between $v_{i,s}$ and $v_{j,t}$. Note that $W$ is a symmetric matrix, since $\rho$ is symmetric. Let $D$ be the $\kappa\times \kappa$ diagonal matrix
\begin{equation}
  \label{eq:diagonal_D}
  D:=\mathrm{diag}\left\{\sum_{j=1}^n\sum_{t=1}^{\kappa_j}W_{1j}\left( 1,t \right),\cdots,\sum_{j=1}^n\sum_{t=1}^{\kappa_j}W_{n,j}\left( \kappa_n,t \right) \right\},
\end{equation}
then the \emph{graph hypoelliptic Laplacian} for the triplet $\left( \mathscr{X},\rho,G \right)$ is defined as the graph Laplacian of the graph $G$ with edge weights given by $W$, that is
\begin{equation}
  \label{eq:hypoelliptic_graph_laplacian}
  L^H:=D-W.
\end{equation}
Since $G$ is connected, the diagonal elements of $D$ are all non-zero, and we can define the \emph{random-walk} and \emph{normalized} version of $L^H$
\begin{equation}
  \label{eq:hypoelliptic_graph_laplacian_rw}
  L^H_{\textrm{rw}}:=D^{-1}L^H=I-D^{-1}W,
\end{equation}
\begin{equation}
  \label{eq:hypoelliptic_graph_laplacian_normalized}
  L^H_{*}:=D^{-1/2}L^HD^{-1/2} = I-D^{-1/2}WD^{-1/2}.
\end{equation}
Following~\cite{CoifmanLafon2006}, we can also repeat the constructions above on a renormalized graph of $G$. More precisely, let $Q_{\alpha}$ be the $K\times K$ diagonal matrix
\begin{equation}
  \label{eq:diagonal_Q}
  Q_{\alpha}:=\mathrm{diag}\left\{\left(\sum_{j=1}^N\sum_{t=1}^{K_j}W_{1j}\left( 1,t \right)\right)^{\alpha},\cdots,\left(\sum_{j=1}^N\sum_{t=1}^{K_j}W_{N,j}\left( K_N,t \right)\right)^{\alpha} \right\},
\end{equation}
where $\alpha$ is some constant between $0$ and $1$, and set
\begin{equation}
  \label{eq:W_alpha}
  W_{\alpha}:=Q_{\alpha}^{-1}WQ_{\alpha}^{-1}.
\end{equation}
The graph hypoelliptic Laplacians can then be constructed for $W_{\alpha}$ instead of $W$, by first forming the $K\times K$ diagonal matrix $D_{\alpha}$
\begin{equation}
  \label{eq:diagonal_D_alpha}
  D_{\alpha}:=\mathrm{diag}\left\{\sum_{j=1}^N\sum_{t=1}^{K_j}\left(W_{\alpha}\right)_{1j}\left( 1,t \right),\cdots,\sum_{j=1}^N\sum_{t=1}^{K_j}\left(W_{\alpha}\right)_{N,j}\left( K_N,t \right) \right\},
\end{equation}
and then set
\begin{equation}
  \label{eq:hypoelliptic_graph_laplacian_alpha}
  L_{\alpha}^H:=D_{\alpha}-W_{\alpha},
\end{equation}
\begin{equation}
  \label{eq:hypoelliptic_graph_laplacian_rw_alpha}
  L^{H}_{\alpha,\textrm{rw}}:=D_{\alpha}^{-1}L^{H}_\alpha=I-D_{\alpha}^{-1}W_{\alpha},
\end{equation}
\begin{equation}
  \label{eq:hypoelliptic_graph_laplacian_normalized_alpha}
  L^{H}_{\alpha,*}:=D_{\alpha}^{-1/2}L^H_{\alpha}D_{\alpha}^{-1/2} = I-D_{\alpha}^{-1/2}W_{\alpha}D_{\alpha}^{-1/2}.
\end{equation}

\subsection{Spectral Distances and Embeddings}
\label{sec:spectr-dist-embedd}

In order to define spectral distances, we shall use eigen-decompositions. This is the reason to consider the symmetric matrices $L^H_{\alpha,*}$; their eigen-decompositions lead to a natural representation for $L^H_{\alpha,\mathrm{rw}}$, since $L^{H}_{\alpha,*}$ and $L^H_{\alpha,\mathrm{rw}}$ are \emph{diagonal-similar}:
$$L_{\alpha,\textrm{*}}^H=D_{\alpha}^{1/2}L_{\alpha,\mathrm{rw}}^H D_{\alpha}^{-1/2}.$$

Let $v\in\mathbb{R}^{\kappa\times 1}$ be an eigenvector of $L_{\alpha,*}^H$ corresponding to eigenvalue $\lambda$; $v$ defines a function on the vertices of $G$, or equivalently on the data set $\mathscr{X}$. By the construction of $L_{\alpha,*}^H$, $v\in \mathbb{R}^{\kappa\times 1}$ can be written as the concatenation of $n$ segments of length $\kappa_1,\cdots,\kappa_n$,
\begin{equation*}
  v = \left( v^{\top}_{\left[ 1 \right]},\cdots, v_{\left[ n \right]}^{\top}\right)^{\top}
\end{equation*}
where $v_{\left[ j \right]}\in\mathbb{R}^{\kappa_j\times 1}$ defines a function on fibre $X_j$. Now let $\lambda_0\leq\lambda_1\leq\lambda_2\leq\cdots\leq\lambda_{\kappa-1}$ be the $\kappa$ eigenvalues of $L^H_{\alpha,*}$ in ascending order, and denote the eigenvector corresponding to eigenvalue $\lambda_j$ as $v_j$. By our connectivity assumption for $G$, we know from spectral graph theory~\cite{Chung1997} that $\lambda_0=0$, $\lambda_0<\lambda_1$, and $v_0$ is a constant vector with all entries equal to $1$; we have thus
\begin{equation*}
  0=\lambda_0<\lambda_1\leq \lambda_2\leq\cdots\leq \lambda_{\kappa-1}.
\end{equation*}
By the spectral decomposition of $L_{\alpha,*}^H$,
\begin{equation}
  \label{eq:spectral_decomposition}
  L_{\alpha,*}^H=\sum_{l=0}^{\kappa-1} \lambda_l v_lv_l^{\top},
\end{equation}
and for any fixed \emph{diffusion time} $t\in\mathbb{R}^{+}$,
\begin{equation}
  \label{eq:spectral_decomposition_t}
  \left(L_{\alpha,*}^H\right)^t=\sum_{l=0}^{\kappa-1} \lambda^t_l v_lv_l^{\top},
\end{equation}
with the $\left( i,j \right)$-block taking the form
\begin{equation}
  \label{eq:spectral_decomposition_block}
  \left(\left(L_{\alpha,*}^H\right)^t\right)_{ij}=\sum_{l=0}^{\kappa-1} \lambda^t_l v_{l \left[ i \right]} v_{l \left[ j \right]}^{\top}.
\end{equation}
In general, this block is not square. Its Frobenius norm can be computed as
\begin{equation}
  \label{eq:spectral_decomposition_block_HSnorm}
  \begin{aligned}
    \left\|\left(\left(L_{\alpha,*}^H\right)^t\right)_{ij}\right\|_{\mathrm{F}}^2&=\mathrm{Tr}\left[ \left(\left(L_{\alpha,*}^H\right)^t\right)_{ij}\left(\left(L_{\alpha,*}^H\right)^t\right)_{ij}^{\top} \right]\\
    &=\mathrm{Tr}\left[ \sum_{l,m=0}^{\kappa-1}\lambda_l^t\lambda_m^t v_{l \left[ i \right]}v_{l \left[ j \right]}^{\top}v_{m \left[ j \right]}v_{m \left[ i \right]}^{\top} \right]\\
    &=\mathrm{Tr}\left[ \sum_{l,m=0}^{\kappa-1}\lambda_l^t\lambda_m^t v_{m \left[ i \right]}^{\top}v_{l \left[ i \right]}v_{l \left[ j \right]}^{\top}v_{m \left[ j \right]} \right]\\
    &=\sum_{l,m=0}^{\kappa-1}\lambda_l^t\lambda_m^t v_{m \left[ i \right]}^{\top}v_{l \left[ i \right]}v_{l \left[ j \right]}^{\top}v_{m \left[ j \right]}.
  \end{aligned}
\end{equation}
Let us define the \emph{hypoelliptic base diffusion map}
\begin{equation}
  \label{eq:hbdm}
  \begin{aligned}
    V^t:\mathscr{B}&\longrightarrow \mathbb{R}^{\kappa^2}\\
    X_j &\longmapsto \left( \lambda_l^{t/2}\lambda_m^{t/2}v_{l \left[ j \right]}^{\top}v_{m \left[ j \right]} \right)_{0\leq l,m\leq \kappa-1}
  \end{aligned}
\end{equation}
then (denoting the standard Euclidean inner produce in $\mathbb{R}^{\kappa^2}$ as $\langle\cdot, \cdot\rangle$)
\begin{equation}
  \label{eq:HDM_innerproduct}
  \begin{aligned}
    \left\|\left(\left(L_{\alpha,*}^H\right)^t\right)_{ij}\right\|_{\mathrm{F}}^2 = \left\langle V^t \left( X_i \right), V^t \left( X_j \right) \right\rangle,
\end{aligned}
\end{equation}
with which we can define the \emph{hypoelliptic base diffusion distance} on $\mathscr{B}$ as
\begin{equation}
  \label{eq:hbdd}
  \begin{aligned}
    d_{\mathrm{HBDM},t}&\left( X_i,X_j \right) = \left\|V^t \left( X_i \right)- V^t \left( X_j \right) \right\|\\
    &=\left\{\left\langle V^t \left( X_i \right), V^t \left( X_i \right) \right\rangle+\left\langle V^t \left( X_j \right), V^t \left( X_j \right) \right\rangle-2\left\langle V^t \left( X_i \right), V^t \left( X_j \right) \right\rangle  \right\}^{\frac{1}{2}}.
  \end{aligned}
\end{equation}
The hypoelliptic base diffusion map embeds the base data set $\mathscr{B}$ into a Euclidean space using $G_B$, the base affinity graph with edges weighted by entry-wise non-negative matrices. In this sense, it is closely related to the vector diffusion maps~\cite{SingerWu2012VDM}: if
$$\kappa_1=\kappa_2=\cdots=\kappa_n=d$$
and (relaxing the constraint $\rho\geq 0$)
$$\left(v_i, v_j\right)\in E_B \Leftrightarrow \rho_{ij}=w_{ij}O_{ij}, \,\textrm{where}\,w_{ij}\geq 0\textrm{ and } O_{ij}\textrm{ is } d\times d \textrm{ orthogonal},$$
then the weighted adjacency matrix $W$ (as defined in \eqref{eq:hedm_W}) coincides with the adjacency matrix $S$ defined in \cite[\S 3]{SingerWu2012VDM}. In this case, the graph hypoelliptic Laplacian of $\left( \mathscr{X},\rho,G \right)$ reduces to the graph connection Laplacian for $\left(G_B,\left\{ w_{ij} \right\}, \left\{ O_{ij} \right\}\right)$. Note that in HDM we assume the non-negativity of the similarity measure $\rho$, which is generally not the case for vector diffusion maps. (The non-negativity of the eigenvalues of $L_{\alpha,*}^H$ allows us to consider arbitrary powers $L^H_{\alpha,*}$; in VDM, this is circumvented by considering powers of $S^2$.) From a different point of view, by the Riesz Representation Theorem, smooth vector fields on a manifold $M$ can be identified with linear functions on $TM$, thus VDM can be viewed as HDM restricted on the space of linear functions on $TM$.

In addition to embedding the base data set $\mathscr{B}$, HDM is also capable of embedding the total data set $\mathscr{X}$ into Euclidean spaces. Define for each \emph{diffusion time} $t\in\mathbb{R}^{+}$ the \emph{hypoelliptic diffusion map}
\begin{equation}
\label{eq:hdm}
  \begin{aligned}
    H^t:\mathscr{X}&\longrightarrow \mathbb{R}^{\kappa-1}\\
    x_{j,s}&\longmapsto \left(\lambda_1^t v_{1\left[ j \right]}\left( s \right),\lambda^t_2v_{2 \left[ j \right]}\left( s \right),\cdots,\lambda_{\kappa-1}^tv_{\left(\kappa-1\right)\left[ j \right]}\left( s \right) \right).
  \end{aligned}
\end{equation}
where $v_{l \left[ j \right]} \left( s \right)$ is the $s$-th entry of the $j$-th segment of the $l$-th eigenvector, with $j=1,\cdots,n, s=1,\cdots,\kappa_j$. We could also have written
\begin{equation*}
  v_{l \left[ j \right]} \left( s \right) = v_l \left( s_j+s \right),\quad\textrm{where } s_1=0 \textrm{ and } s_j=\sum_{p=1}^{j-1}\kappa_p \textrm{ for }j\geq 2.
\end{equation*}
Following a similar argument as in~\cite{CoifmanLafon2006}, we can define the \emph{hypoelliptic diffusion distance} on $\mathscr{X}$ as
\begin{equation}
  \label{eq:hdd}
  d_{\mathrm{HDM},t}\left( x_{i,s},x_{j,t} \right) = \left\|H^t \left( x_{i,s} \right) - H^t \left( x_{j,t} \right) \right\|.
\end{equation}
As a result, $H^t$ embeds the total data set $\mathscr{X}$ into a Euclidean space in such a manner that the hypoelliptic diffusion distance on $\mathscr{X}$ is preserved. Moreover, this embedding automatically suggests a global registration for all fibres, according to the similarity measure $\rho$. For simplicity of notations, let us write
\begin{equation*}
  H^t_j:=H^t\restriction_{X_j}
\end{equation*}
for the restriction of $H^t$ to fibre $X_j$, and call it the $j$-th \emph{component} of $H^t$. Up to scaling, the components of $H^t$ bring the fibres of $\mathscr{X}$ to a common ``template'', such that points $x_{i,s}$ and $x_{j,t}$ with a high similarity measure $\rho_{ij}\left( s,t \right)$ tend to be close to each other in the embedded Euclidean space. Pairwise correspondences between fibres $X_i,X_j$ can then be reconstructed from the hypoelliptic diffusion map. Indeed, assuming each $X_j$ is sampled from some manifold $F_j$, and a \emph{template fibre} $F\subset\mathbb{R}^m$ can be estimated from
$$H^t_1 \left( X_1 \right),\cdots,H^t_n \left( X_n \right),$$
then one can often extend (by interpolation) $H_j^t$ from a discrete correspondence to a continuous bijective map from $F_j$ to $F$, and build correspondence maps between an arbitrary pair $X_i,X_j$ by composing (the interpolated continuous maps) $H_i^t$ and $\left(H_j^t\right)^{-1}$. A similar construction was implicit in~\cite{Kim12FuzzyCorr}. Sometimes, it is more useful to consider a normalized version of hypoelliptic diffusion map that takes value on the standard unit sphere in $\mathbb{R}^{\kappa-1}$:
\begin{equation}
  \label{eq:hdm_sphere}
  \begin{aligned}
    \widetilde{H}^t:\mathscr{X}&\longrightarrow S^{\kappa-2}\subset \mathbb{R}^{\kappa}\\
    x_{j,s}&\longmapsto \frac{H_j^t \left( x_{j,s} \right)}{\left\|H_j^t \left( x_{j,s} \right)\right\|}.
  \end{aligned}
\end{equation}
We shall see an example that applies $\widetilde{H}^t$ to $\mathrm{SO}\left( 3 \right)$ in Section~\ref{sec:numer-exper}.


%% file: hdm_tangent.tex
The HDM framework is very flexible: if each fibre consists of one single point, the hypoelliptic graph Laplacian reduces to the graph Laplacian that underlies diffusion maps; if all the fibres have the same number of points and all similarity matrices (defined in Section~\ref{sec:spectr-dist-embedd}) are orthogonal (up to a multiplicative constant), the hypoelliptic graph Laplacian reduces to the graph connection Laplacian that underlies vector diffusion maps. The goal of this section is to relate HDM to some other partial differential operators of geometric importance on tangent and unit tangent bundles of (compact, closed) Riemannian manifolds. In a follow-up paper, we will extend the geometric setting to more general fibre bundles.

This section builds upon the fibre bundle assumption. Adopting notation in Section~\ref{sec:basic-setup}, we assume that $\mathscr{X}$ is sampled from a fibre bundle $E$, and each fibre $X_j$ is sampled from a fibre over some point on the base manifold $M$. Moreover, we shall assume that $E$ is the tangent bundle or unit tangent bundle of $M$, i.e., $E = TM$ or $E=UTM$. For the convenience of the reader, some basic properties about the geometry of these fibre bundles are reviewed in Appendix~\ref{app:geom-tang-bundl}.

\subsection{HDM on Tangent Bundles}
\label{sec:hdm-tangent-bundles}

Let $K:\mathbb{R}^2\rightarrow\mathbb{R}^{\geq 0}$ be a smooth \emph{kernel function} supported on the unit square $\left[ 0,1 \right]\times \left[ 0,1 \right]$. In all that follows, we shall assume that $M$ is a compact manifold without boundary, which, according to standard custom, we shall simply call a \emph{closed} manifold. Let the closed manifold $M$ be equipped with a Riemannian metric tensor $g$, which induces on $M$ a \emph{geodesic distance} function $d_M \left( \cdot,\cdot \right):M\times M\rightarrow \mathbb{R}^{\geq 0}$; $g$ defines an inner product on each tangent space of $M$, denoted as
\begin{equation}
  \label{eq:metric_induce_inner_product}
  \left\langle u,v\right\rangle_x = g_{jk}\left( x \right)u^jv^k,\quad u,v\in T_xM,
\end{equation}
where, and for the remainder of this section unless otherwise specified, we have adopted the Einstein summation convention. The vector norm on $T_xM$ with respect to this inner product shall be denoted as
\begin{equation}
  \label{eq:metric_induce_norm}
  \left\| u\right\|_x = \left(g_{jk}\left( x \right)u^ju^k\right)^{\frac{1}{2}},\quad u\in T_xM.
\end{equation}
We denote
\begin{equation}
  \label{eq:parallel_transport}
  P_{y,x}:T_xM\rightarrow T_yM
\end{equation}
for the \emph{parallel transport} from $x\in M$ to $y\in M$ with respect to the \emph{Levi-Civita connection} on $M$, along a \emph{geodesic segment} that connects $x$ to $y$. It is well known that a tangent vector can be parallel-transported along any smooth curve on $M$; since $M$ is compact, its \emph{injectivity radius} $\mathrm{Inj}\left( M \right)$ is positive, and thus any $x\in M$ lies within a \emph{geodesic normal neighborhood} in which any point $y$ can be connected to $x$ through a unique geodesic with length smaller than $\mathrm{Inj}\left( M \right)$. Therefore, $P_{y,x}$ is well-defined, at least for $x,y\in M$ with $d_M \left( x,y \right)<\mathrm{Inj}\left( M \right)$. Furthermore, for such $x,y\in M$, $P_{y,x}$ is an orientation-preserving isometry between the domain and target tangent planes~\cite[Exercise 2.1]{doCarmo1992RG}.

For \emph{bandwidth} parameters $\epsilon>0$, $\delta>0$, define for all $\left( x,v \right),\left( y,w \right)\in TM$
\begin{equation}
  \label{eq:kernel_TM}
  K_{\epsilon,\delta}\left( x,v; y,w \right):=K \left( \frac{d^2_M \left( x,y \right)}{\epsilon},\frac{\left\|P_{y,x}v-w\right\|_y^2}{\delta} \right).
\end{equation}
Note that the requirement that $\mathrm{supp}\left( K \right)\subset \left[ 0,1 \right]\times \left[ 0,1 \right]$ implies that $K_{\epsilon,\delta}\left( x,v;y,w \right)\neq 0$ only if $d_M \left( x,y \right)\leq \sqrt{\epsilon}$. It follows that $P_{y,x}$, and further $K_{\epsilon,\delta}\left( x,v;y,w \right)$, are well-defined when $\epsilon\leq \mathrm{Inj}\left( M \right)^2$; we shall restrict ourselves to such sufficiently small $\epsilon$. $K_{\epsilon,\delta}$ is symmetric because $P_{y,x}$ is an isometry between $T_xM$ and $T_yM$:
\begin{equation*}
  \begin{aligned}
    K_{\epsilon,\delta}\left( y,w; x,v \right) &= K \left( \frac{d^2_M \left( y,x \right)}{\epsilon}, \frac{\left\|P_{x,y}w-v  \right\|^2_x}{\delta} \right)=K \left( \frac{d^2_M \left( x,y \right)}{\epsilon}, \frac{\left\|P_{y,x}\left(P_{x,y}w-v\right)  \right\|^2_y}{\delta} \right)\\
    &=K \left( \frac{d^2_M \left( x,y \right)}{\epsilon}, \frac{\left\|w-P_{y,x}v  \right\|^2_y}{\delta} \right)=K_{\epsilon,\delta}\left( x,v; y,w \right).
  \end{aligned}
\end{equation*}
This symmetry is of particular importance for the definition of symmetric diffusion semigroups~\cite{Stein1970}.

Let $p\in C^{\infty}\left( TM \right)$ be the \emph{probability density function} according to which we shall sample. Assume $p$ is bounded from both above and below (away from $0$):
\begin{equation}
  \label{eq:density}
  0<p_m\leq p \left( x,v \right)\leq p_M<\infty,\quad\forall \left( x,v \right)\in TM.
\end{equation}
Define
\begin{equation}
  \label{eq:density_empirical}
  p_{\epsilon,\delta}\left( x,v \right):=\int_{TM}K_{\epsilon,\delta}\left( x,v;y,w \right)p \left( y,w \right)\,d\mu \left( y,w \right).
\end{equation}
where $d\mu$ is the standard volume form on $TM$. As in Appendix~\ref{app:geom-tang-bundl}, $d\mu$ is a product of $dV_y \left( w \right)$ (the standard translation- and rotation-invariant Borel measure on $T_yM$) and $d\mathrm{vol}_M \left( y \right)$ (the standard Riemannian volume element on $M$). If we fix $x\in M$ and integrate $p \left( x,v \right)$ along $T_xM$, then
\begin{equation}
  \label{eq:density_on_base}
  \overline{p}\left( x \right):=\int_{T_xM}p \left( x,v \right)\,dV_x \left( v \right)
\end{equation}
is a density function on $M$, since
\begin{equation*}
  \int_M \overline{p} \left( x \right)\,d\mathrm{vol}\left( x \right)=\int_M\!\int_{T_xM}p \left( x,v \right)\,dV_x \left( v \right) \,d\mathrm{vol}\left( x \right)=\int_{TM}p \left( x,v \right)\,d\mu \left( x,v \right)=1.
\end{equation*}
We call $\overline{p}$ the \emph{projection} of $p$ on $M$. Furthermore, dividing $p \left( x,v \right)$ by $\overline{p}\left( x \right)$ yields a \emph{conditional probability density function} on $T_xM$
\begin{equation}
  \label{eq:conditional_density_on_fibre}
  p \left( v \mid x\right) = \frac{p \left( x,v \right)}{\overline{p}\left( x \right)},
\end{equation}
since
\begin{equation*}
  \int_{T_xM}p \left( v \mid x \right)\,dV_x \left( v \right)=\frac{\int_{T_xM}p \left( x,v \right)\,dV_x \left( v \right)}{\overline{p}\left( x \right)}=\frac{\overline{p}\left( x \right)}{\overline{p}\left( x \right)}=1.
\end{equation*}
For any $f\in C^{\infty}\left( TM \right)$, we can define its \emph{average along fibres} on $TM$, with respect to the conditional probability density functions, as the following function on $M$:
\begin{equation}
  \label{eq:average_along_fibres}
  \overline{f}\left( x \right)=\int_{T_xM}f \left( x,v \right)p \left( v\mid x \right)\,dV_x \left( v \right),\quad \forall x\in M.
\end{equation}

Finally, for any $0\leq\alpha\leq 1$, define the $\alpha$-normalized kernel
\begin{equation}
  \label{eq:alpha_normalized_kernel}
  K_{\epsilon,\delta}^{\alpha}\left(x,v;y,w \right):=\frac{K_{\epsilon,\delta}\left(x,v;y,w \right)}{p_{\epsilon,\delta}^{\alpha}\left( x,v \right)p_{\epsilon,\delta}^{\alpha}\left( y,w \right)}.
\end{equation}
We are now ready to define a family of \emph{hypoellitpic diffusion operators} on $TM$ as
\begin{equation}
  \label{eq:hdo_tm}
  H_{\epsilon,\delta}^{\alpha}f \left( x,v \right):=\frac{\displaystyle\int_{TM}K_{\epsilon,\delta}^{\alpha}\left( x,v;y,w \right)f \left( y,w \right)p \left( y,w \right)\,d\mu \left( y,w \right)}{\displaystyle\int_{TM}K_{\epsilon,\delta}^{\alpha}\left( x,v;y,w \right)p \left( y,w \right)\,d\mu \left( y,w \right)}
\end{equation}
for any $f\in C^{\infty}\left( TM \right)$.

We are interested in the asymptotic behavior of $H_{\epsilon,\delta}^{\alpha}$ in the limit $\epsilon\rightarrow0$, $\delta\rightarrow0$. It turns out that this depends on the relative rate with which $\epsilon$ and $\delta$ approach $0$. For simplicity of notation, let us write $\gamma=\delta/\epsilon$.

\begin{theorem}[HDM on Tangent Bundles]
\label{thm:main_tm}
  Let $M$ be a closed Riemannian manifold, and $H_{\epsilon,\delta}^{\alpha}:C^{\infty}\left( TM \right)\rightarrow C^{\infty}\left( TM \right)$ defined as in \eqref{eq:hdo_tm}. If $\delta=O \left( \epsilon \right)$ as $\epsilon\rightarrow0$, or equivalently if $\gamma=\delta/\epsilon$ is asymptotically bounded, then for any $f\in C^{\infty}\left( TM \right)$, as $\epsilon\rightarrow0$ (and thus $\delta\rightarrow 0$),
\begin{equation}
  \label{eq:main_tm_i}
  \begin{aligned}
    H_{\epsilon,\delta}^{\alpha}f\left( x,v \right) &= f \left( x,v \right)+\epsilon \frac{m_{21}}{2m_0}\left[\frac{\Delta^H\left[fp^{1-\alpha}\right]\left( x,v \right)}{p^{1-\alpha}\left( x,v \right)}-f \left( x,v \right) \frac{\Delta^Hp^{1-\alpha}\left( x,v \right)}{p^{1-\alpha}\left( x,v \right)}  \right]\\
    &+\delta \frac{m_{22}}{2m_0}\left[\frac{\Delta^V\left[fp^{1-\alpha}\right]\left( x,v \right)}{p^{1-\alpha}\left( x,v \right)}-f \left( x,v \right) \frac{\Delta^Vp^{1-\alpha}\left( x,v \right)}{p^{1-\alpha}\left( x,v \right)}  \right]+O\left(\epsilon^2+\delta^2 \right),\\
\end{aligned}
\end{equation}
where $m_0$, $m_{21}$, $m_{22}$ are positive constants depending only on the kernel $K$.
\end{theorem}
Theorem~\ref{thm:main_tm} is the tangent bundle version of Theorem~\ref{thm:main_utm} (which applies to unit tangent bundles), and the proofs for these two theorems are essentially identical. We included a proof of Theorem~\ref{thm:main_utm} in Appendix~\ref{app:proof-main-theorem}, from which a proof of Theorem~\ref{thm:main_tm} can be easily adapted.

\begin{proposition}
\label{prop:gamma_infty}
Let $M$ and $H_{\epsilon,\delta}^{\alpha}$ be as in Theorem~{\rm \ref{thm:main_tm}}. If $\delta=\gamma\epsilon$, then for any $f\in C^{\infty}\left( TM \right)$ and sufficiently small $\epsilon>0$,
\begin{equation}
  \label{eq:main_tm_ii_base}
  \begin{aligned}
  \lim_{\gamma\rightarrow\infty}H^{\alpha}_{\epsilon,\gamma\epsilon}f \left( x,v \right)=\overline{f}\left( x \right)+\epsilon \frac{m'_2}{2m'_0} \left[\frac{\Delta_M\left[\overline{f}\overline{p}^{1-\alpha}\right]\left( x \right)}{\overline{p}^{1-\alpha}\left( x \right)}-\overline{f}\left( x \right)\frac{\Delta_M\overline{p}^{1-\alpha}\left( x \right)}{\overline{p}^{1-\alpha}\left( x \right)} \right]+O \left( \epsilon^2 \right),
  \end{aligned}
\end{equation}
where $\Delta_M$ is the Laplace-Beltrami operator on $M$, $\overline{p}$ is the projected density function on $M$, $\overline{f}\left( x \right)$ is the average of $f$ along fibres on $TM$, and $m_0'$, $m_2'$ are constants that only depend on the kernel function $K$.
\end{proposition}

\begin{corollary}
\label{cor:gamma_infty}
  Let $M$ and $H_{\epsilon,\delta}^{\alpha}$ be as in Theorem~{\rm \ref{thm:main_tm}}. If $\delta/\epsilon=\gamma\rightarrow\infty$ as $\epsilon\rightarrow0$, $\delta\rightarrow 0$, then for any $f\in C^{\infty}\left( TM \right)$, in general
\begin{equation*}
  \lim_{\gamma\rightarrow\infty}\lim_{\delta\rightarrow0}\frac{H_{\epsilon,\gamma\epsilon}^{\alpha}f\left( x,v \right)-f \left( x,v \right)}{\epsilon}\neq \lim_{\delta\rightarrow0}\lim_{\gamma\rightarrow\infty}\frac{H_{\epsilon,\gamma\epsilon}^{\alpha}f\left( x,v \right)-f \left( x,v \right)}{\epsilon},
\end{equation*}
and thus an asymptotic expansion of $H_{\epsilon,\delta}^{\alpha}f \left( x,v \right)$ for small $\epsilon, \delta$ is not well-defined. In fact, for each fixed $\gamma$, as in \eqref{eq:main_tm_i},
\begin{equation}
  \label{eq:main_tm_ii}
  \begin{aligned}
    H_{\epsilon,\gamma\epsilon}^{\alpha}f\left( x,v \right) = &f \left( x,v \right)+\epsilon\frac{m_{21}}{2m_0}\left[\frac{\Delta^H\left[fp^{1-\alpha}\right]\left( x,v \right)}{p^{1-\alpha}\left( x,v \right)}-f \left( x,v \right) \frac{\Delta^Hp^{1-\alpha}\left( x,v \right)}{p^{1-\alpha}\left( x,v \right)}  \right]\\
    +\gamma \epsilon&\frac{m_{22}}{2m_0}\left[\frac{\Delta^V\left[fp^{1-\alpha}\right]\left( x,v \right)}{p^{1-\alpha}\left( x,v \right)}-f \left( x,v \right) \frac{\Delta^Vp^{1-\alpha}\left( x,v \right)}{p^{1-\alpha}\left( x,v \right)} \right]+O \left( \epsilon^2 \right),
\end{aligned}
\end{equation}
whereas by Proposition{\rm ~\ref{prop:gamma_infty}}
\begin{equation}
  \begin{aligned}
  \lim_{\gamma\rightarrow\infty}H^{\alpha}_{\epsilon,\gamma\epsilon}f \left( x,v \right)=\overline{f}\left( x \right)+\epsilon \frac{m'_2}{2m'_0} \left[\frac{\Delta_M\left[\overline{f}\overline{p}^{1-\alpha}\right]\left( x \right)}{\overline{p}^{1-\alpha}\left( x \right)}-\overline{f}\left( x \right)\frac{\Delta_M\overline{p}^{1-\alpha}\left( x \right)}{\overline{p}^{1-\alpha}\left( x \right)} \right]+O \left( \epsilon^2 \right).
  \end{aligned}
\end{equation}
\end{corollary}

\begin{corollary}
  Under the same assumptions and notation as in Theorem{\rm~\ref{thm:main_tm}}, if $\alpha=1$, then
\begin{enumerate}[(i)]
\item\label{item:15} If $\delta=O \left( \epsilon \right)$ as $\epsilon\rightarrow0$, then for any $f\in C^{\infty}\left( TM \right)$, as $\epsilon\rightarrow0$ (and thus $\delta\rightarrow0$),
  \begin{equation}
    \label{eq:corollary_alpha1}
    H_{\epsilon,\delta}^1f \left( x,v \right)=f \left( x,v \right)+\epsilon \frac{m_{21}}{2m_0}\Delta^Hf \left( x,v \right)+\delta \frac{m_{22}}{2m_0}\Delta^Vf \left( x,v \right)+O \left( \epsilon^2+\delta^2 \right);
  \end{equation}
\item\label{item:16} For any $f\in C^{\infty}\left( TM \right)$,
  \begin{equation}
    \label{eq:corollary_alpha1_base}
    \lim_{\gamma\rightarrow\infty}H_{\epsilon,\gamma\epsilon}^1f \left( x,v \right)=\overline{f}\left( x \right)+\epsilon \frac{m_2'}{2m_0'}\Delta_M\overline{f}\left( x \right)+O \left( \epsilon^2 \right).
  \end{equation}
\end{enumerate}
\end{corollary}

\subsection{HDM on Unit Tangent Bundles}
\label{sec:hdm-unit-tangent}

The construction of HDM for unit tangent bundles is very similar to the construction in Section~\ref{sec:hdm-tangent-bundles}. We only need to replace the volume element $d\mu$ on $TM$ with $d\Theta$, the \emph{Liouville measure} on $UTM$ (see, e.g., \cite[Chapter VII]{Chavel2006}), and modify the definition of $K_{\epsilon,\delta}$ in \eqref{eq:kernel_TM} into
\begin{equation}
  \label{eq:kernel_UTM}
  K_{\epsilon,\delta}\left( x,v; y,w \right):=K \left( \frac{d^2_M \left( x,y \right)}{\epsilon},\frac{d^2_{S_y}\left( P_{y,x}v,w \right)}{\delta} \right),
\end{equation}
where $S_y$ is the unit sphere in $T_yM$, and $d_{S_y}\left( \cdot,\cdot \right)$ is the geodesic distance on $S_y$ under the induced metric from $T_yM$. Note that $K_{\epsilon,\delta}$ as defined in \eqref{eq:kernel_UTM} is still symmetric. Abusing notation, we shall not distinguish the $K_{\epsilon,\delta}$ in \eqref{eq:kernel_TM} with the unit tangent bundle version \eqref{eq:kernel_UTM}, whenever the specification can be inferred from contexts. Similarly, notation
\begin{equation}
  \label{eq:density_on_base_utm}
  \overline{p}\left( x \right):=\int_{S_x}p \left( x,v \right)\,dV_x \left( v \right),
\end{equation}
\begin{equation}
  \label{eq:conditional_density_on_fibre_utm}
  p \left( v \mid x\right) = \frac{p \left( x,v \right)}{\overline{p}\left( x \right)},
\end{equation}
and
\begin{equation}
  \label{eq:average_along_fibres_utm}
  \overline{f}\left( x \right)=\int_{S_x}f \left( x,v \right)p \left( v\mid x \right)\,dV_x \left( v \right),\quad \forall f\in C^{\infty}\left( UTM \right), \forall x\in M
\end{equation}
will stay the same as in \eqref{eq:density_on_base}, \eqref{eq:conditional_density_on_fibre}, and \eqref{eq:average_along_fibres}. Like in~\eqref{eq:hdo_tm}, now we can define a family of \emph{hypoellitpic diffusion operators} on $UTM$ for any $0\leq \alpha\leq 1$ as
\begin{equation}
  \label{eq:hdo_utm}
  H_{\epsilon,\delta}^{\alpha}f \left( x,v \right):=\frac{\displaystyle\int_{UTM}K_{\epsilon,\delta}^{\alpha}\left( x,v;y,w \right)f \left( y,w \right)p \left( y,w \right)\,d\Theta \left( y,w \right)}{\displaystyle\int_{UTM}K_{\epsilon,\delta}^{\alpha}\left( x,v;y,w \right)p \left( y,w \right)\,d\Theta \left( y,w \right)}
\end{equation}
for any $f\in C^{\infty}\left( UTM \right)$.

\begin{theorem}[HDM on Unit Tangent Bundles]
\label{thm:main_utm}
  Let $M$ be a closed Riemannian manifold, and $H_{\epsilon,\delta}^{\alpha}:C^{\infty}\left( UTM \right)\rightarrow C^{\infty}\left( UTM \right)$ defined as in \eqref{eq:hdo_utm}. If $\delta=O \left( \epsilon \right)$ as $\epsilon\rightarrow0$, or equivalently if $\gamma=\delta/\epsilon$ is asymptotically bounded, then for any $f\in C^{\infty}\left( UTM \right)$, as $\epsilon\rightarrow0$ (and thus $\delta\rightarrow0$),
\begin{equation}
  \label{eq:main_utm_i}
  \begin{aligned}
    H_{\epsilon,\delta}^{\alpha}f\left( x,v \right) &= f \left( x,v \right)+\epsilon \frac{m_{21}}{2m_0}\left[\frac{\Delta_S^H\left[fp^{1-\alpha}\right]\left( x,v \right)}{p^{1-\alpha}\left( x,v \right)}-f \left( x,v \right) \frac{\Delta_S^Hp^{1-\alpha}\left( x,v \right)}{p^{1-\alpha}\left( x,v \right)}  \right]\\
    &+\delta \frac{m_{22}}{2m_0}\left[\frac{\Delta_S^V\left[fp^{1-\alpha}\right]\left( x,v \right)}{p^{1-\alpha}\left( x,v \right)}-f \left( x,v \right) \frac{\Delta_S^Vp^{1-\alpha}\left( x,v \right)}{p^{1-\alpha}\left( x,v \right)}  \right]+O\left(\epsilon^2+\delta^2 \right),\\
\end{aligned}
\end{equation}
where $m_0$, $m_{21}$, $m_{22}$ are positive constants depending only on the kernel $K$.
\end{theorem}
We included a proof of Theorem~\ref{thm:main_utm} in Appendix~\ref{app:proof-main-theorem}. The proof of Theorem~\ref{thm:main_tm} is essentially the same.

\begin{proposition}
\label{prop:gamma_infty_utm}
Let $M$ and $H_{\epsilon,\delta}^{\alpha}$ be as in Theorem~{\rm \ref{thm:main_utm}}. If $\delta=\gamma\epsilon$, then for any $f\in C^{\infty}\left( UTM \right)$ and sufficiently small $\epsilon>0$,
\begin{equation}
\label{eq:main_utm_ii_base}
  \begin{aligned}
  \lim_{\gamma\rightarrow\infty}H^{\alpha}_{\epsilon,\gamma\epsilon}f \left( x,v \right)=\overline{f}\left( x \right)+\epsilon \frac{m'_2}{2m'_0} \left[\frac{\Delta_M\left[\overline{f}\overline{p}^{1-\alpha}\right]\left( x \right)}{\overline{p}^{1-\alpha}\left( x \right)}-\overline{f}\left( x \right)\frac{\Delta_M\overline{p}^{1-\alpha}\left( x \right)}{\overline{p}^{1-\alpha}\left( x \right)} \right]+O \left( \epsilon^2 \right),
  \end{aligned}
\end{equation}
where $\Delta_M$ is the Laplace-Beltrami operator on $M$, $\overline{p}$ is the projected density function on $M$, $\overline{f}\left( x \right)$ is the average of $f$ along fibres on $UTM$, and $m_0'$, $m_2'$ are constants that only depend on the kernel function $K$.
\end{proposition}

\begin{corollary}
\label{cor:gamma_infty_utm}
Let $M$ and $H^{\alpha}_{\epsilon,\delta}$ be as in Theorem{\rm ~\ref{thm:main_utm}}. If $\delta/\epsilon=\gamma\rightarrow\infty$ as $\epsilon\rightarrow0$, $\delta\rightarrow0$, then for any $f\in C^{\infty}\left( UTM \right)$, in general
\begin{equation*}
  \lim_{\gamma\rightarrow\infty}\lim_{\delta\rightarrow0}\frac{H_{\epsilon,\gamma\epsilon}^{\alpha}f\left( x,v \right)-f \left( x,v \right)}{\epsilon}\neq \lim_{\delta\rightarrow0}\lim_{\gamma\rightarrow\infty}\frac{H_{\epsilon,\gamma\epsilon}^{\alpha}f\left( x,v \right)-f \left( x,v \right)}{\epsilon}.
\end{equation*}
and thus an asymptotic expansion of $H_{\epsilon,\delta}^{\alpha}f \left( x,v \right)$ for small $\epsilon,\delta$ is not well-defined. In fact, for each fixed $\gamma$, as in \eqref{eq:main_utm_i},
\begin{equation}
  \label{eq:main_utm_ii}
  \begin{aligned}
    H_{\epsilon,\gamma\epsilon}^{\alpha}f\left( x,v \right) = &f \left( x,v \right)+\epsilon\frac{m_{21}}{2m_0}\left[\frac{\Delta_S^H\left[fp^{1-\alpha}\right]\left( x,v \right)}{p^{1-\alpha}\left( x,v \right)}-f \left( x,v \right) \frac{\Delta_S^Hp^{1-\alpha}\left( x,v \right)}{p^{1-\alpha}\left( x,v \right)}  \right]\\
    +\gamma \epsilon&\frac{m_{22}}{2m_0}\left[\frac{\Delta_S^V\left[fp^{1-\alpha}\right]\left( x,v \right)}{p^{1-\alpha}\left( x,v \right)}-f \left( x,v \right) \frac{\Delta_S^Vp^{1-\alpha}\left( x,v \right)}{p^{1-\alpha}\left( x,v \right)} \right]+O \left( \epsilon^2 \right),
\end{aligned}
\end{equation}
whereas by Proposition{\rm ~\ref{prop:gamma_infty_utm}}
\begin{equation}
  \begin{aligned}
  \lim_{\gamma\rightarrow\infty}H^{\alpha}_{\epsilon,\gamma\epsilon}f \left( x,v \right)=\overline{f}\left( x \right)+\epsilon \frac{m'_2}{2m'_0} \left[\frac{\Delta_M\left[\overline{f}\overline{p}^{1-\alpha}\right]\left( x \right)}{\overline{p}^{1-\alpha}\left( x \right)}-\overline{f}\left( x \right)\frac{\Delta_M\overline{p}^{1-\alpha}\left( x \right)}{\overline{p}^{1-\alpha}\left( x \right)} \right]+O \left( \epsilon^2 \right).
  \end{aligned}
\end{equation}
\end{corollary}

\begin{corollary}
  Under the same assumptions and notation as in Theorem{\rm~\ref{thm:main_utm}}, if $\alpha=1$, then
\begin{enumerate}[(i)]
\item\label{item:17} If $\delta=O \left( \epsilon \right)$ as $\epsilon\rightarrow0$, then for any $f\in C^{\infty}\left( UTM \right)$, as $\epsilon\rightarrow0$ (and thus $\delta\rightarrow0$),
  \begin{equation}
    \label{eq:corollary_alpha1_utm}
    H_{\epsilon,\delta}^1f \left( x,v \right)=f \left( x,v \right)+\epsilon \frac{m_{21}}{2m_0}\Delta_S^Hf \left( x,v \right)+\delta \frac{m_{22}}{2m_0}\Delta_S^Vf \left( x,v \right)+O \left( \epsilon^2+\delta^2 \right);
  \end{equation}
\item\label{item:18} For any $f\in C^{\infty}\left( UTM \right)$,
  \begin{equation}
    \label{eq:corollary_alpha1_base_utm}
    \lim_{\gamma\rightarrow\infty}H_{\epsilon,\gamma\epsilon}^1f \left( x,v \right)=\overline{f}\left( x \right)+\epsilon \frac{m_2'}{2m_0'}\Delta_M\overline{f}\left( x \right)+O \left( \epsilon^2 \right).
  \end{equation}
\end{enumerate}
\end{corollary}

\subsection{Finite Sampling on Unit Tangent Bundles}
\label{sec:finite-sampling-unit}

Though the theory of hypoelliptic diffusion maps on unit tangent bundles is completely parallel to its counterpart on tangent bundles, in practice it is usually much easier to sample from the unit tangent bundle since it is compact whenever the base manifold is. It thus makes much more sense to study finite sampling on unit tangent bundles. In this section, we first consider sampling without noise, i.e. where we sample exactly on unit tangent bundles; next, we study the case where the tangent spaces are empirically estimated from samples on the base manifold. The latter scenario is a proof-of-concept for applying the hypoelliptic diffusion map framework to much more general fibre bundles in practice, where data representing each fibre are often acquired with noise. The proofs of Theorem~\ref{thm:utm_finite_sampling_noiseless} and Theorem~\ref{thm:utm_finite_sampling_noise} can be found in Appendix~\ref{app:proof-main-theorem}. In Section~\ref{sec:numer-exper}, we shall demonstrate a numerical experiment that addresses the difference between the two sampling strategies.

\subsubsection{Sampling without Noise}
\label{sec:sampl-with-noise}

We begin with some assumptions and definitions. Assumption~\ref{assum:technical} includes our technical assumptions, and Assumption~\ref{assum:utm_finite_sampling_noiseless} specifies the noiseless sampling strategy.
\begin{assumption}
\label{assum:technical}
\begin{enumerate}
\item\label{item:13} $\iota: M\hookrightarrow \mathbb{R}^D$ is an isometric embedding of a $d$-dimensional closed Riemannian manifold into $\mathbb{R}^D$, with $D\gg d$.

\item\label{item:20} Let the two-variable smooth function $K:\mathbb{R}^2\rightarrow\mathbb{R}^{\geq0}$ be compactly supported on the unit square $\left[ 0,1 \right]\times \left[ 0,1 \right]$. The partial derivatives $\partial_1K$, $\partial_2K$ are therefore automatically compactly supported on the unit square as well. (In fact, a similar result still holds if $K$ and its first order derivatives decay faster at infinity than any inverse polynomials; we avoid such technicalities and focus on demonstrating the idea, using compactly supported $K$.)
\end{enumerate}
\end{assumption}

\begin{assumption}
\label{assum:utm_finite_sampling_noiseless}
The $\left( N_B\times N_F \right)$ data points
  \begin{equation*}
    \begin{matrix}
      x_{1,1}, &x_{1,2}, &\cdots, &x_{1,N_F}\\
      x_{2,1}, &x_{2,2}, &\cdots, &x_{2,N_F}\\
      \vdots & \vdots &\cdots &\vdots\\
      x_{N_B,1}, &x_{N_B,2}, &\cdots, &x_{N_B,N_F}
    \end{matrix}
  \end{equation*}
are sampled from $UTM$ with respect to a probability density function $p \left( x,v \right)$ satisfying \eqref{eq:density}, following a two-step strategy: (i) sample $N_B$ points $\xi_1,\cdots,\xi_{N_B}$ i.i.d. on $M$ with respect to $\overline{p}$, the projection of $p$ on $M$  (recall \eqref{eq:density_on_base_utm}); (ii) sample $N_F$ points $x_{j,1},\cdots,x_{j,N_F}$ on $S_{\xi_j}$ with respect to $p \left(\cdot\mid\xi_j\right)$, the conditional probability density on the fibre (recall \eqref{eq:conditional_density_on_fibre_utm}).
\end{assumption}

\begin{definition}
\label{defn:utm_finite_sampling_noiseless}
\begin{enumerate}
\item\label{item:34} For $\epsilon>0$, $\delta>0$ and $1\leq i,j\leq N_B$, $1\leq r,s\leq N_F$, define
  \begin{equation*}
    \hat{K}_{\epsilon,\delta} \left( x_{i,r}, x_{j,s} \right) =
   \begin{cases}
    \displaystyle K \left( \frac{\|\xi_i-\xi_j\|^2}{\epsilon}, \frac{\|P_{\xi_j,\xi_i}x_{i,r}-x_{j,s}\|^2}{\delta} \right), & i\neq j,\\
   0,&i=j.
    \end{cases}
  \end{equation*} 
where $P_{\xi_j,\xi_i}:S_{\xi_i}\rightarrow S_{\xi_j}$ is the parallel transport from $S_{\xi_i}$ to $S_{\xi_j}$. Note the difference between $\hat{K}_{\epsilon,\delta}$ and $K_{\epsilon,\delta}$ (defined in~\eqref{eq:kernel_UTM}): $\hat{K}_{\epsilon,\delta}$ uses Euclidean distance while $K_{\epsilon,\delta}$ uses geodesic distance.
\item\label{item:35} For $0\leq \alpha\leq 1$, define
\begin{equation*}
    \hat{p}_{\epsilon,\delta}\left( x_{i,r} \right) = \sum_{j=1}^{N_B}\sum_{s=1}^{N_F}\hat{K}_{\epsilon,\delta} \left( x_{i,r}, x_{j,s} \right)
  \end{equation*}
and the \emph{empirical $\alpha$-normalized kernel} $\hat{K}_{\epsilon,\delta}^{\alpha}$
  \begin{equation*}
    \hat{K}_{\epsilon,\delta}^{\alpha} \left( x_{i,r}, x_{j,s} \right) = \frac{\hat{K}_{\epsilon,\delta}\left( x_{i,r}, x_{j,s} \right)}{\hat{p}^{\alpha}_{\epsilon,\delta}\left( x_{i,r}\right)\hat{p}^{\alpha}_{\epsilon,\delta}\left( x_{j,s} \right)},\quad 1\leq i, j\leq N_B, 1\leq r,s\leq N_F.
  \end{equation*}
\item\label{item:36} For $0\leq \alpha\leq 1$ and $f\in C^{\infty}\left( UTM \right)$, denote the \emph{$\alpha$-normalized empirical hypoelliptic diffusion operator} by
  \begin{equation*}
    \hat{H}_{\epsilon,\delta}^{\alpha}f\left( x_{i,r} \right)=\frac{\displaystyle \sum_{j=1}^{N_B}\sum_{s=1}^{N_F}\hat{K}_{\epsilon,\delta}^{\alpha} \left( x_{i,r}, x_{j,s}\right)f \left( x_{j,s} \right)}{\displaystyle \sum_{j=1}^{N_B}\sum_{s=1}^{N_F}\hat{K}_{\epsilon,\delta}^{\alpha} \left( x_{i,r}, x_{j,s}\right)}.
  \end{equation*}
\end{enumerate}
\end{definition}

\begin{theorem}[Finite Sampling without Noise]
\label{thm:utm_finite_sampling_noiseless}
  Under Assumption~{\rm\ref{assum:technical}} and Assumption~{\rm\ref{assum:utm_finite_sampling_noiseless}}, if
\begin{enumerate}[(i)]
\item\label{item:41} $\delta=O \left( \epsilon \right)$ as $\epsilon\rightarrow0$;
\item\label{item:42}
\begin{equation*}
  \lim_{N_B\rightarrow\infty\atop N_F\rightarrow\infty}\frac{N_F}{N_B}=\rho\in \left( 0,\infty \right),
\end{equation*}
\end{enumerate}
then for any $x_{i,r}$ with $1\leq i\leq N_B$ and $1\leq r\leq N_F$, as $\epsilon\rightarrow0$ (and thus $\delta\rightarrow0$), with high probability
\begin{equation}
\label{eq:finite_sample_theorem}
\begin{aligned}
    \hat{H}_{\epsilon,\delta}^{\alpha}f\left( x_{i,r} \right) &= f \left( x_{i,r} \right)+\epsilon \frac{m_{21}}{2m_0}\left[\frac{\Delta_S^H\left[fp^{1-\alpha}\right]\left( x_{i,r} \right)}{p^{1-\alpha}\left( x_{i,r} \right)}-f \left( x_{i,r} \right) \frac{\Delta_S^Hp^{1-\alpha}\left( x_{i,r} \right)}{p^{1-\alpha}\left( x_{i,r} \right)}  \right]\\
    &+\delta \frac{m_{22}}{2m_0}\left[\frac{\Delta_S^V\left[fp^{1-\alpha}\right]\left( x_{i,r} \right)}{p^{1-\alpha}\left( x_{i,r} \right)}-f \left( x_{i,r} \right) \frac{\Delta_S^Vp^{1-\alpha}\left( x_{i,r} \right)}{p^{1-\alpha}\left( x_{i,r} \right)}  \right]\\
    &+O \left( \epsilon^2+\delta^2+\theta_{*}^{-1}N_B^{-\frac{1}{2}}\epsilon^{-\frac{d}{4}} \right),
\end{aligned}
\end{equation}
where
\begin{equation*}
  \theta_{*}=1-\frac{\displaystyle 1}{\displaystyle 1+\epsilon^{\frac{d}{4}}\delta^{\frac{d-1}{4}}\sqrt{\frac{N_F}{N_B}}}.
\end{equation*}
\end{theorem}
We give a proof of Theorem~\ref{thm:utm_finite_sampling_noiseless} in Appendix~\ref{sec:proof-finite-sampling-theorem}.

\subsubsection{Sampling from Empirical Tangent Spaces}
\label{sec:sampl-estim-tang}

In practice, it has been shown in~\cite{SingerWu2012VDM} that, under the manifold assumption, a local PCA procedure can be used for estimating tangent spaces from a point cloud; we are using PCA here as a procedure that determines the dimension of a local good linear approximation to the manifold, and also, conveniently, provides a good basis, which can be viewed as a basis for each tangent plane. To sample on these tangent spaces, it suffices to repeatedly sample coordinate coefficients from a fixed standard unit sphere; each sample can be interpreted as giving the coordinates of a point (approximately) on the tangent space. Parallel-transports will take the corresponding point that truly lies on the tangent space at $\xi$ to the tangent space at $\zeta$, another point on the manifold. This new tangent space is, however, again known only approximately; points in this approximate space are characterized by coordinates with respect to the local PCA basis at $\zeta$. We can thus express the whole (approximate) parallel-transport procedure by maps between coordinates with respect to PCA basis at $\xi$ to sets of coordinates at $\zeta$; these changes of coordinates incorporate information on the choices of basis at each end as wells as on the parallel-transport itself.

Let us now describe this in more detail, setting up notation simultaneously. Throughout this section, Assumption~\ref{assum:technical} still holds. Let $\left\{ \xi_1,\cdots,\xi_{N_B} \right\}$ be a collection of i.i.d. samples from $M$; then the local PCA procedure can be summarized as follows: for any $\xi_j$, $1\leq j\leq N_B$, let $\xi_{j_1},\cdots,\xi_{j_k}$ be its $k$ nearest neighboring points. Then
\begin{equation*}
  X_j=\left[ \xi_{j_1}-\xi_j,\cdots, \xi_{j_k}-\xi_j\right]
\end{equation*}
is a $D\times k$ matrix. Let $K_{\mathrm{PCA}}$ be a positive monotonic decreasing function supported on  the unit interval, e.g., the Epanechnikov kernel
\begin{equation*}
  K_{\mathrm{PCA}}\left( u \right)=\left( 1-u^2 \right)\chi_{\left[ 0,1 \right]},
\end{equation*}
where $\chi$ is the indicator function. Fix a scale parameter $\epsilon_{\mathrm{PCA}}>0$, let $D_j$ be the $k\times k$ diagonal matrix
\begin{equation*}
  D_j = \mathrm{diag}\left( \sqrt{K_{\mathrm{PCA}}\left( \frac{\left\|\xi_j-\xi_{j_1}\right\|}{\sqrt{\epsilon_{\mathrm{PCA}}}} \right)}, \cdots, \sqrt{K_{\mathrm{PCA}}\left( \frac{\left\|\xi_j-\xi_{j_k}\right\|}{\sqrt{\epsilon_{\mathrm{PCA}}}} \right)} \right)
\end{equation*}
and carry out the singular value decomposition (SVD) of matrix $X_jD_j$ as
\begin{equation*}
  X_jD_j=U_j\Sigma_jV_j^{\top}.
\end{equation*}
An estimated basis $B_j$ for the local tangent plane at $\xi_j$ is formed by the first $d$ left singular vectors (corresponding to the $d$ largest singular values in $\Sigma_j$), arranged into a matrix as follows:
\begin{equation*}
  B_j=\left[ u_j^{\left( 1 \right)},\cdots, u_j^{\left( d \right)} \right]\in \mathbb{R}^{D\times d}.
\end{equation*}
Note that the intrinsic dimension $d$ is generally not known \emph{a priori}; \cite{SingerWu2012VDM} proposed estimating dimension locally from the decay of singular values in $\Sigma_j$, and then take the median of all local dimensions to estimate $d$; \cite{LittleMaggioniRosasco2011} proposed a different approach based on multi-scale singular value decomposition.

Once a pair of estimated bases $B_i,B_j$ is obtained for neighboring points $\xi_i,\xi_j$, one estimates a parallel-transport from $T_{\xi_i}M$ to $T_{\xi_j}M$ as
\begin{equation*}
  O_{ji}:=\argmin_{O\in O \left( d \right)}\left\|O-B_j^{\top}B_i\right\|_{\mathrm{HS}},
\end{equation*}
where $\left\|\cdot\right\|_{\mathrm{HS}}$ is the Hilbert-Schmidt norm. Though this minimization problem is non-convex, it has a efficient closed-form solution via the SVD of $B_i^{\top}B_j$, namely
\begin{equation*}
  O_{ji}=UV^{\top},\quad\textrm{where }B_j^{\top}B_i=U\Sigma V^{\top} \textrm{ is the SVD of $B_j^{\top}B_i$.}
\end{equation*}
It is worth noting that $O_{ji}$ depends on the bases; it operates on the coordinates of tangent vectors under $B_i$ and $B_j$, as explained above. $O_{ji}$ approximates the true parallel-transport $P_{\xi_j,\xi_i}$ (composed with the bases-expansions) with an error of $O \left( \epsilon_{\mathrm{PCA}} \right)$, in the sense of \cite[lemma B.1]{SingerWu2012VDM}.

We summarize our sampling strategy for this section (with some new notations) in the following definition.
\begin{definition}
\label{defn:utm_finite_sampling_noise}
\begin{enumerate}
\item\label{item:31} Let $\left\{ \xi_1,\cdots,\xi_{N_B} \right\}$ be a collection of samples from the base manifold $M$, i.i.d. with respect to some probability density function $\overline{p}\in C^{\infty}\left( M \right)$. For each $\xi_j$, $1\leq j\leq N_B$, sample $N_F$ points uniformly from the $\left(d-1\right)$-dimensional standard unit sphere $S^{d-1}$ in $\mathbb{R}^d$, and denote the set of samples as $\mathscr{C}_j = \left\{ c_{j,1},\cdots,c_{j,N_F} \right\}$, where each $c_{j,s}$ is a $d\times 1$ column vector. Using the basis $B_j$  estimated from the local PCA procedure, each $c_{j,s}$ corresponds to an ``approximate tangent vector at $\xi_j$'', denoted as
  \begin{equation*}
    \tau_{j,s}:=B_jc_{j,s}.
  \end{equation*}
We use the notation $\mathscr{S}_j$ for the unit sphere in the estimated tangent space (i.e., the column space of $B_j$). Note that the $\tau_{j,1},\cdots,\tau_{j,N_F}$ are uniformly distributed on $\mathscr{S}_j$.
\item\label{item:40} By \cite[lemma B.1]{SingerWu2012VDM}, for any $B_j$ there exists a $D\times d$ matrix $Q_j$, such that the columns of $Q_j$ constitutes an orthonormal basis for $\iota_{*}T_{\xi_j}M$ and
  \begin{equation*}
    \left\|B_j-Q_j\right\|_{\mathrm{HS}}=O \left( \epsilon_{\mathrm{PCA}} \right).
  \end{equation*}
We define the \emph{tangent projection} from $\iota_{*}S_{\xi_j}$ to the estimated tangent plane as
\begin{equation*}
  \tau_{j,s}\mapsto\overline{\tau}_{j,s}= \frac{Q_jQ_j^{\top}\tau_{j,s}}{\left\|Q_jQ_j^{\top}\tau_{j,s}\right\|}.
\end{equation*}
This map is well-defined for sufficiently small $\epsilon_{\mathrm{PCA}}$, and then it is an isometry. Its inverse is given by
\begin{equation*}
  \overline{\tau}_{j,s}\mapsto \tau_{j,s}=\frac{B_jB_j^{\top}\overline{\tau}_{j,s}}{\left\|B_jB_j^{\top}\overline{\tau}_{j,s}\right\|}.
\end{equation*}
Note that we have
\begin{equation*}
  \left\|\tau_{j,s}-\overline{\tau}_{j,s}\right\|\leq C\epsilon_{\mathrm{PCA}}
\end{equation*}
for some constant $C>0$ independent of indices $j,s$. Since we sample each $\mathscr{S}_j$ uniformly and the projection map $\tau_{j,s}\mapsto\bar{\tau}_{j,s}$ is an isometry, the points $\left\{ \overline{\tau}_{j,1},\cdots,\overline{\tau}_{j,N_F} \right\}$ are also uniformly distributed on $S_{\xi_j}$. The points
\begin{equation*}
  \begin{matrix}
    \overline{\tau}_{1,1}, &\overline{\tau}_{1,2}, &\cdots, &\overline{\tau}_{1,N_F}\\
    \overline{\tau}_{2,1}, &\overline{\tau}_{2,2}, &\cdots, &\overline{\tau}_{2,N_F}\\
    \vdots & \vdots &\cdots &\vdots\\
    \overline{\tau}_{N_B,1}, &\overline{\tau}_{N_B,2}, &\cdots, &\overline{\tau}_{N_B,N_F}
  \end{matrix}
\end{equation*}
are therefore distributed on $UTM$ according to a joint probability density function $p$ on $UTM$ defined as
\begin{equation*}
  p \left( x,v \right)=\overline{p}\left( x \right),\quad \forall \left( x,v \right)\in UTM.
\end{equation*}
As in Assumption~\ref{assum:utm_finite_sampling_noiseless}, we assume $p$ satisfies \eqref{eq:density}, i.e.,
\begin{equation*}
  0<p_m\leq p \left( x,v \right)=\overline{p}\left( x \right)\leq p_M <\infty,\quad \forall \left( x,v \right)\in UTM
\end{equation*}
for positive constants $p_m,p_M$.
\item\label{item:14} For $\epsilon>0$, $\delta>0$ and $1\leq i,j\leq N_B$, $1\leq r,s\leq N_F$, define
  \begin{equation*}
    \mathscr{K}_{\epsilon,\delta} \left( \overline{\tau}_{i,r}, \overline{\tau}_{j,s} \right) =
   \begin{cases}
    \displaystyle K \left( \frac{\|\xi_i-\xi_j\|^2}{\epsilon}, \frac{\|O_{ji}c_{i,r}-c_{j,s}\|^2}{\delta} \right), & i\neq j,\\
   0,&i=j.
    \end{cases}
  \end{equation*} 
where $O_{ji}$ is the estimated parallel-transport from $T_{\xi_i}M$ to $T_{\xi_j}M$.
\item\label{item:29} For $0\leq \alpha\leq 1$, define
\begin{equation*}
    \hat{q}_{\epsilon,\delta}\left( \overline{\tau}_{i,r} \right) = \sum_{j=1}^{N_B}\sum_{s=1}^{N_F}\mathscr{K}_{\epsilon,\delta} \left( \overline{\tau}_{i,r}, \overline{\tau}_{j,s} \right)
  \end{equation*}
and
  \begin{equation*}
    \mathscr{K}_{\epsilon,\delta}^{\alpha} \left( \overline{\tau}_{i,r}, \overline{\tau}_{j,s} \right) = \frac{\mathscr{K}_{\epsilon,\delta}\left( \overline{\tau}_{i,r}, \overline{\tau}_{j,s} \right)}{\hat{q}^{\alpha}_{\epsilon,\delta}\left( \overline{\tau}_{i,r}\right)\hat{q}^{\alpha}_{\epsilon,\delta}\left( \overline{\tau}_{j,s} \right)},\quad 1\leq i, j\leq N_B, 1\leq r,s\leq N_F.
  \end{equation*}
\item\label{item:30} For $0\leq \alpha\leq 1$ and $f\in C^{\infty}\left( UTM \right)$, denote
  \begin{equation*}
    \mathscr{H}_{\epsilon,\delta}^{\alpha}f\left( \overline{\tau}_{i,r} \right)=\frac{\displaystyle \sum_{j=1}^{N_B}\sum_{s=1}^{N_F}\mathscr{K}_{\epsilon,\delta}^{\alpha} \left( \overline{\tau}_{i,r}, \overline{\tau}_{j,s}\right)f \left( \overline{\tau}_{j,s} \right)}{\displaystyle \sum_{j=1}^{N_B}\sum_{s=1}^{N_F}\mathscr{K}_{\epsilon,\delta}^{\alpha} \left( \overline{\tau}_{i,r}, \overline{\tau}_{j,s}\right)}.
  \end{equation*}
\end{enumerate}
\end{definition}

\begin{theorem}[Finite Sampling from Empirical Tangent Planes]
\label{thm:utm_finite_sampling_noise}
  In addition to Assumption~{\rm\ref{assum:technical}}, suppose
\begin{enumerate}[(i)]
\item\label{item:33} $\epsilon_{\textrm{PCA}} = O \left( N_B^{-\frac{2}{d+2}} \right)$ as $N_B\rightarrow\infty$;
\item\label{item:32} As $\epsilon\rightarrow0$, $\delta=O \left( \epsilon \right)$ and $\delta\gg \left( \epsilon_{\mathrm{PCA}}^{\frac{1}{2}}+\epsilon^{\frac{3}{2}} \right)$;
\item\label{item:39}
\begin{equation*}
  \lim_{N_B\rightarrow\infty\atop N_F\rightarrow\infty}\frac{N_F}{N_B}=\rho\in \left( 0,\infty \right).
\end{equation*}
\end{enumerate}
Then for any $\tau_{i,r}$ with $1\leq i\leq N_B$ and $1\leq r\leq N_F$, as $\epsilon\rightarrow0$ (and thus $\delta\rightarrow0$), with high probability
\begin{equation}
\label{eq:finite_sample_theorem_noise}
\begin{aligned}
    \mathscr{H}_{\epsilon,\delta}^{\alpha}f\left( \overline{\tau}_{i,r} \right) &= f \left( \overline{\tau}_{i,r} \right)+\epsilon \frac{m_{21}}{2m_0}\left[\frac{\Delta_S^H\left[fp^{1-\alpha}\right]\left( \overline{\tau}_{i,r} \right)}{p^{1-\alpha}\left( \overline{\tau}_{i,r} \right)}-f \left( \overline{\tau}_{i,r} \right) \frac{\Delta_S^Hp^{1-\alpha}\left( \overline{\tau}_{i,r} \right)}{p^{1-\alpha}\left( \overline{\tau}_{i,r} \right)}  \right]\\
    &+\delta \frac{m_{22}}{2m_0}\left[\frac{\Delta_S^V\left[fp^{1-\alpha}\right]\left( \overline{\tau}_{i,r} \right)}{p^{1-\alpha}\left( \overline{\tau}_{i,r} \right)}-f \left( \overline{\tau}_{i,r} \right) \frac{\Delta_S^Vp^{1-\alpha}\left( \overline{\tau}_{i,r} \right)}{p^{1-\alpha}\left( \overline{\tau}_{i,r} \right)}  \right]\\
    &+O \left( \epsilon^2+\delta^2+\theta_{*}^{-1}N_B^{-\frac{1}{2}}\epsilon^{-\frac{d}{4}}+\delta^{-1}\left( \epsilon_{\mathrm{PCA}}^{\frac{1}{2}}+\epsilon^{\frac{3}{2}} \right) \right),
\end{aligned}
\end{equation}
where
\begin{equation*}
  \theta_{*}=1-\frac{\displaystyle 1}{\displaystyle 1+\epsilon^{\frac{d}{4}}\delta^{\frac{d-1}{4}}\sqrt{\frac{N_F}{N_B}}}.
\end{equation*}
\end{theorem}
We give a proof of Theorem~\ref{thm:utm_finite_sampling_noise} in Appendix~\ref{sec:proof-finite-sampling-theorem}.


%% file: numerical.tex
In this section, we consider a numerical experiment on $\mathrm{SO}\left( 3 \right)$, the special linear group of dimension $3$, realized as the unit tangent bundle of the standard sphere $S^2$ in $\mathbb{R}^3$. We shall compare both sampling strategies covered in Section~\ref{sec:hypo-diff-maps-tangent}.

\subsection{Sampling without Noise}
\label{sec:sampl-with-noise-numerical}
In the first step, we uniformly sample $N_B=2,000$ points $\left\{\xi_1,\cdots,\xi_{N_B} \right\}$ on the unit sphere $S^2$, and find for each sample point the $K_B=100$ nearest neighbors in the point cloud. Next, we sample $N_F=50$ vectors of unit length tangent to the unit sphere at each sample point (which in this case is a circle), thus collecting a total of $N_B\times N_F=100,000$ points on $UTS^2=\mathrm{SO}\left( 3 \right)$, denoted as
\begin{equation*}
\left\{ x_{j,s}\mid 1\leq j\leq N_B,1\leq s\leq N_F \right\}.
\end{equation*}
The hypoelliptic diffusion matrix $H$ is then constructed as an $N_B\times N_B$ block matrix with block size $N_F\times N_F$, and $H_{ij}$ (the $\left( i,j \right)$-th block of $H$) is non-zero only if the sample points $\xi_i,\xi_j$ are each among the $K_B$-nearest neighbors of the other; when $H_{ij}$ is non-zero, its $\left( r,s \right)$-entry ($1\leq r,s\leq N_F$) is non-zero only if $P_{\xi_j,\xi_i}x_{i,r}$ and $x_{j,s}$ are each among the $K_F=50$ nearest neighbors of the other, and in that case
\begin{equation}
\label{eq:matrix_entry_noiseless}
  \begin{aligned}
    H_{ij} \left( r,s \right)&=\exp \left( -\frac{\left\|\xi_i-\xi_j\right\|^2}{\epsilon} \right)\exp \left( -\frac{\left\|P_{\xi_j,\xi_i}x_{i,r}-x_{j,s}\right\|^2}{\delta} \right)\\
    &=\exp \left[ -\left(\frac{\left\|\xi_i-\xi_j\right\|^2}{\epsilon}+\frac{\left\|P_{\xi_j,\xi_i}x_{i,r}-x_{j,s}\right\|^2}{\delta}\right) \right],
  \end{aligned}
\end{equation}
where the choices of $\epsilon,\delta$ will be explained below. Note that for the unit sphere $S^2$ the parallel-transport from $T_{\xi_i}S^2$ to $T_{\xi_j}S^2$ can be explicitly constructed as a rotation along the axis $\xi_i\times \xi_j$. Finally, we form the $\alpha$-normalized hypoelliptic diffusion matrix $H_{\alpha}$ by
\begin{equation}
\label{eq:normalize_diffusion_matrix}
  \left(H_{\alpha}\right)_{ij}\left( r,s \right)=\frac{H_{ij} \left( r,s \right)}{\left( \displaystyle \sum_{l=1}^{N_B}\sum_{m=1}^{N_F}H_{il} \left( r,m \right) \right)^{\alpha}\left( \displaystyle \sum_{k=1}^{N_B}\sum_{n=1}^{ N_F}H_{jk} \left( r,n \right) \right)^{\alpha}},
\end{equation}
and solve the eigenvalue problem
\begin{equation}
\label{eq:generalized_eigen_problem}
  \left(D^{-\frac{1}{2}}H_{\alpha}D^{-\frac{1}{2}}\right)U = U\Lambda
\end{equation}
where $D$ is the $\left(N_BN_F\right)\times\left(N_BN_F\right)$ diagonal matrix with entry $\left( k,k \right)$ equal to the $k$-th column sum of $H_{\alpha}$:
\begin{equation*}
  D \left( k,k \right) = \sum_{v=1}^{N_B N_F}H_{\alpha}\left( k,v \right),
\end{equation*}
and $\Lambda$ is a diagonal matrix of the same dimensions. Throughout this experiment, we fix $\alpha=1$, $\epsilon=0.2$ and choose various values of $\delta$ ranging from $0.0005$ to $50$, and observe the spacing of the eigenvalues stored in $\Lambda$.

The purpose of this experiment is to investigate the influence of the ratio $\gamma=\delta/\epsilon$ on the spectral behavior of graph hypoelliptic Laplacians. As shown in Figure~\ref{fig:laplacians}, the spacing in the spectrum of these graph hypoelliptic Laplacians follow patterns similar to the multiplicities of the eigenvalues of corresponding Laplacians on $\mathrm{SO}\left( 3 \right)$ (governed by the relative size of $\delta$ and $\epsilon$). In Figure~\ref{fig:laplacians}(a), $\delta\ll\epsilon$, hence the graph hypoelliptic Laplacian approximates the horizontal Laplacian on $\mathrm{SO}\left( 3 \right)$ (according to Theorem~\ref{thm:main_utm}), in which the smallest eigenvalues have  multiplicities $1,6,13,\cdots$; in Figure~\ref{fig:laplacians}(b), $\delta=O \left(\epsilon \right)$, hence the graph hypoelliptic Laplacian approximates the total Laplacian on $\mathrm{SO}\left( 3 \right)$ (according to Theorem~\ref{thm:main_utm}), with eigenvalue multiplicities $1,9,25,\cdots$); in Figure~\ref{fig:laplacians}(c), $\delta\gg\epsilon$, hence the graph hypoelliptic Laplacian approximates the Laplacian on the base manifold $S^2$ (according to Corollary~\ref{cor:gamma_infty_utm}), with eigenvalue multiplicities $1,3,5,\cdots$). Note that in Figure~\ref{fig:laplacians}(c) we fixed $\epsilon$ and pushed $\delta$ to $\infty$, which is essentially equivalent to the limit process in \eqref{eq:main_utm_ii_base} rather than \eqref{eq:main_utm_ii}. Moreover, if in each figure we divide the sequence of eigenvalues by the smallest non-zero eigenvalue, the resulting sequence coincide with the eigenvalues of the corresponding manifold Laplacian up to numerical error. For a description of the spectrum of these partial differential operators, see \cite[Chapter 2]{Taylor1990NHA}.
\begin{figure}[htps]
  \centering
\begin{tabular}{ccc}
  \includegraphics[width=0.31\textwidth]{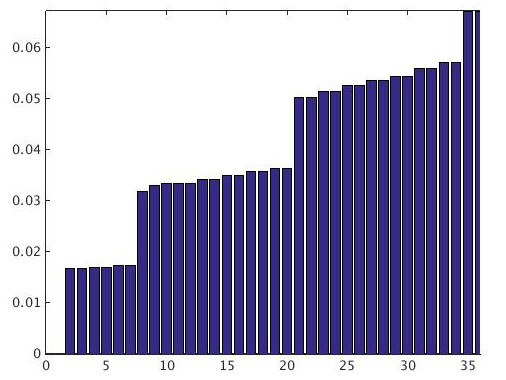}&
  \includegraphics[width=0.31\textwidth]{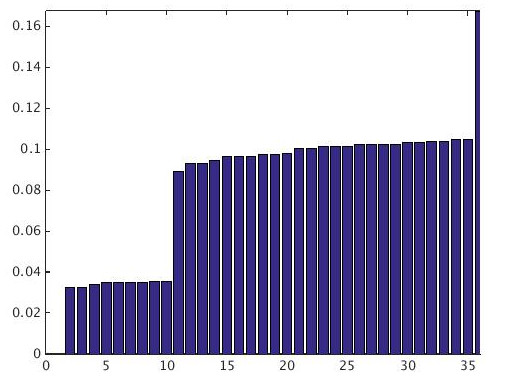}&
  \includegraphics[width=0.31\textwidth]{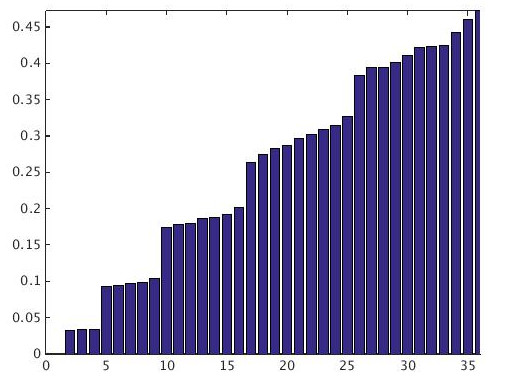}\\
  (a) $\delta=0.002$ & (b) $\delta=0.015$ & (c) $\delta=20$\\
\end{tabular}
  \caption{Bar plots of the smallest $36$ eigenvalues of $3$ graph hypoelliptic Laplacians with fixed $\epsilon=0.2$ and varying $\delta$ (sampling without noise)}
  \label{fig:laplacians}
\end{figure}

\subsection{Sampling from Empirical Tangent Spaces}
\label{sec:sampl-from-empir-numerical}
Similar to sampling without noise, we uniformly sample $N_B$ points $\left\{\xi_1,\cdots,\xi_{N_B} \right\}$ on the unit sphere in the first step, then construct the $K_B$-nearest-neighbor-graph for the point cloud with $K_B=100$; the only difference is that here $N_B=4,000$. (This finer discretization is necessary since we know from Theorem~\ref{thm:utm_finite_sampling_noise} and Theorem~\ref{thm:utm_finite_sampling_noiseless} that sampling from empirically estimated tangent spaces results in a slower convergence rate for HDM on unit tangent bundles. For the same reason we choose a larger $N_F$, see below.) Next, we perform local PCA (with $K_{\mathrm{PCA}}\left( u \right)=e^{-5u^2}\chi_{\left[ 0,1 \right]}$) in the $K_B$-neighborhood around each sample point $\xi_j$, and solve for $O_{ij}$ from the local PCA bases $B_i,B_j$ whenever $\xi_i,\xi_j$ are among the $K_B$-nearest-neighbors of each other. We then sample $N_F=100$ points from the standard unit circle $S^1$ in $\mathbb{R}^2$ for each $\xi_j$, and denote them as
\begin{equation*}
  \left\{ c_{j,s}\mid 1\leq j\leq N_B,1\leq s\leq N_F \right\}.
\end{equation*}
The block construction of the hypoelliptic diffusion matrix $H$ is similar to the noiseless case, but with non-zero $H_{ij}\left( r,s \right)$ replaced with
\begin{equation}
\label{eq:matrix_entry_noise}
  \begin{aligned}
    H_{ij} \left( r,s \right)&=\exp \left( -\frac{\left\|\xi_i-\xi_j\right\|^2}{\epsilon} \right)\exp \left( -\frac{\left\|O_{ji}c_{i,r}-c_{j,s}\right\|^2}{\delta} \right)\\
    &=\exp \left[ -\left(\frac{\left\|\xi_i-\xi_j\right\|^2}{\epsilon}+\frac{\left\|O_{ji}c_{i,r}-c_{j,s}\right\|^2}{\delta}\right) \right].
  \end{aligned}
\end{equation}
Finally, we construct $H_{\alpha}$ as in \eqref{eq:normalize_diffusion_matrix}, set $\alpha=1$, and solve the same generalized eigenvalue problems \eqref{eq:generalized_eigen_problem} with fixed $\epsilon=0.2$ and varying $\delta$. As shown in Figure~\ref{fig:laplacians_noise}, the spacing of the spectrum of graph hypoelliptic Laplacians is quite similar to what was obtained in sampling without noise.
\begin{figure}[htps]
  \centering
\begin{tabular}{ccc}
  \includegraphics[width=0.31\textwidth]{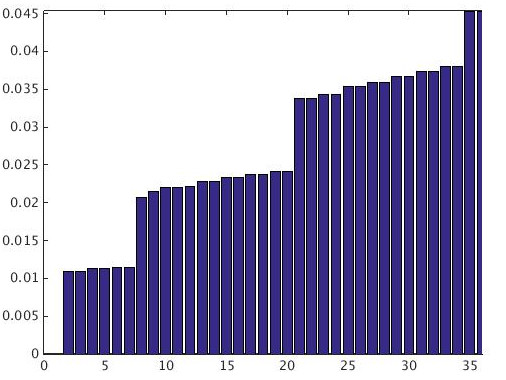}&
  \includegraphics[width=0.31\textwidth]{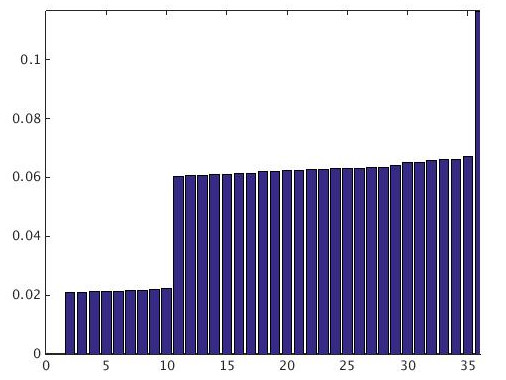}&
  \includegraphics[width=0.31\textwidth]{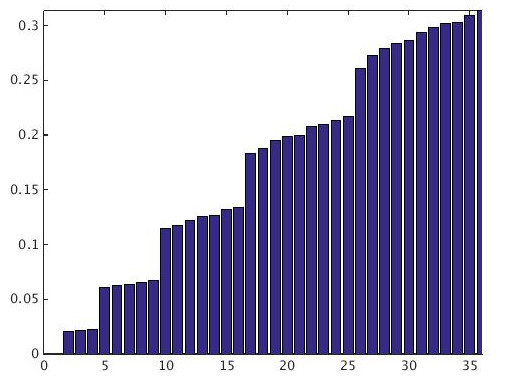}\\
  (a) $\delta=0.002$ & (b) $\delta=0.042$ & (c) $\delta=20$\\
\end{tabular}
  \caption{Bar plots of the smallest $36$ eigenvalues of $3$ graph hypoelliptic Laplacians with fixed $\epsilon=0.2$ and varying $\delta$ (sampling from empirical tangent spaces)}
  \label{fig:laplacians_noise}
\end{figure}

\subsection{As-Flat-As-Possible (AFAP) Connections}
\label{sec:as-flat-as-numerical}
The purpose of this experiment is to provide some insights into the embeddings introduced in \eqref{eq:hbdm} and \eqref{eq:hdm}. As mentioned in Section~\ref{sec:spectr-dist-embedd}, the embedding resulting from hypoelliptic diffusion maps tends to map ``similar points'' (where the similarity is specified by the pairwise correspondences) on different fibres to points on a common ``template'' fibre that are close to each other with respect to the Euclidean space into which the embedding takes place. This is illustrated in Figure~\ref{fig:afap} using the HDM obtained from the unit sphere example in Section~\ref{sec:sampl-from-empir-numerical}, with $N_B=4,000$, $N_F=100$, $K_B=100$, $K_F=50$, $\epsilon=0.2$, and $\delta=0.042$.
\begin{figure}[htps]
  \centering
  \includegraphics[width=0.4\textwidth]{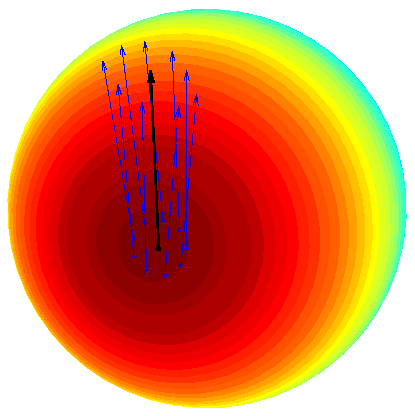}
  \caption{The vector field (near $\xi_j$) on $S^2$ determined by \eqref{eq:afap_construction}}
  \label{fig:afap}
\end{figure}
We pick an arbitrary point $x_{j,s}$ on $S_{\xi_j}$, the $j$-th fibre sampled from the unit tangent bundle, which stands for a unit tangent vector (see the black arrow in Figure~\ref{fig:afap}) to the unit sphere at $\xi_j$, the $j$-th sample. On each fibre $S_{\xi_k}$ where $1\leq k\neq j\leq N_B$, we then look for
\begin{equation}
\label{eq:afap_construction}
  \widetilde{P}_{\xi_k,\xi_j}x_{j,s}:=\argmin_{x_{k,r}\in S_{\xi_k}}\left\|\widetilde{H}^t_k \left( x_{k,r} \right)-\widetilde{H}^t_j \left( x_{j,s} \right)\right\|
\end{equation}
where $\widetilde{H}^t_k$ is defined in \eqref{eq:hdm_sphere}, and we choose $t=1$. The resulting collection of unit tangent vectors
\begin{equation}
  \label{eq:parallel_vector_field}
  \widetilde{\Gamma}:=\left\{ x_{j,s} \right\}\bigcup\left\{ \widetilde{P}_{\xi_k,\xi_j}x_{j,s}\mid 1\leq k\neq j\leq N_B\right\}
\end{equation}
gives rise to a discretization of a section on the unit tangent bundle $UTS^2$; this discretization can then be extended to the entire $S^2$ by interpolation. Since the connection we used in this construction of HDM is the canonical Levi-Civita connection, the similarity between two points on different fibres are measured according to their deviation from being \emph{parallel along geodesic segments} to each other; therefore, each $\widetilde{P}_{\xi_k,\xi_j}x_{j,s}$ stands for the unit tangent vector in the fibre $S_{\xi_k}$ that is the closest to $P_{\xi_k,\xi_j}x_{j,s}$ among all discrete samples $\left\{ x_{k,l}\mid 1\leq l\leq N_F \right\}$. As shown in Figure~\ref{fig:afap}, near $\xi_j$ the vector field $\widetilde{\Gamma}$ (extended by interpolation) is approximately constructed from parallel-transporting $x_{j,s}$ to its neighboring fibres along geodesic segments. HDM thus implicitly constructs vector fields on $S^2$ that are locally as close to a parallel vector field as possible, though we know from differential topology that, globally, there is no ``truly parallel'' unit-norm vector field on the manifold $S^2$. In this particular example, the ``as-parallel-as-possible vector fields'' produced by HDM can also be interpreted as generated by a connection that is \emph{as flat as possible} (AFAP), or \emph{as close to trivial as possible}. A related construction on triangular meshes can be found in \cite{CDS2010}, which relies heavily on the connectivity information stored in the mesh structure. It is worthwhile to note that our approach using HDM is fundamentally different, in that our computational approach uses only a random neighborhood graph of the point cloud constituted by approximate samples of the manifold, as opposed to a structured triangular mesh.

\subsection{An Excursion to the Riemannian Adiabatic Limits}
\label{sec:an-excurs-riem}
The formulae \eqref{eq:matrix_entry_noiseless} and \eqref{eq:matrix_entry_noise} provide an alternative interpretation (other than the one given by the HDM framework developed in this paper) for our numerical experiments in Section~\ref{sec:sampl-with-noise-numerical} and Section~\ref{sec:sampl-from-empir-numerical}, as follows. If we set $\gamma=\delta/\epsilon$, then
\begin{align*}
  &\exp \left[ -\left(\frac{\left\|\xi_i-\xi_j\right\|^2}{\epsilon}+\frac{\left\|P_{\xi_j,\xi_i}x_{i,r}-x_{j,s}\right\|^2}{\delta}\right) \right]\\
  &=\exp \left[ -\frac{1}{\epsilon}\left( \left\|\xi_i-\xi_j\right\|^2+\frac{1}{\gamma}\left\|P_{\xi_j,\xi_i}x_{i,r}-x_{j,s}\right\|^2 \right) \right],
\end{align*}
and our numerical experiments can be understood as applying the standard diffusion map (with bandwidth parameter $\epsilon$) to the total manifold of $\mathrm{SO}\left( 3 \right)$, except that the total manifold is equipped with a family of metrics different from the canonical bi-invariant one. These metrics all rely on the \emph{splitting} of the tangent bundle of $S^2$ by the Levi-Civita connection (as defined in Appendix~\ref{app:geom-tang-bundl}, see \eqref{eq:splitting}), and are formed by recombining the horizontal and vertical components of the Sasaki metric tensor using a parameter $\gamma>0$ that controls the relative weight of the two components. This is in contrast to the interpretation given by the hypoelliptic diffusion map framework: assuming $\gamma>0$ is fixed (implying $\delta=O \left( \epsilon \right)$), recall from Theorem~\ref{thm:main_utm} that
\begin{align}
\label{eq:infinitesimal_generator}
  \lim_{\epsilon\rightarrow0}\frac{H_{\epsilon,\gamma\epsilon}^1 f \left( x \right)-f \left( x \right)}{\epsilon}=\frac{m_{21}}{2m_0} \left( \Delta_S^H+\frac{m_{22}}{m_{21}}\gamma\Delta_S^V\right)f \left( x \right),
\end{align}
thus $\gamma$ controls the infinitesimal generator of the diffusion process in consideration, while the metric on the total manifold of $\mathrm{SO}\left( 3 \right)$ is fixed. (Though we do not fix $\gamma$ in our numerical experiments, the limit in \eqref{eq:infinitesimal_generator} still provides insights for small values of $\epsilon>0$.) This duality between metrics and infinitesimal generators, reflected in the change of the relative size of the bandwidth parameters, is a natural consequence from a differential geometric point of view: the Laplace-Beltrami operator on a Riemannian manifold depends on the choice of the Riemannian metric tensor; while the bandwidth parameters are characteristics of the chosen diffusion map, they can equivalently be interpreted as deformations of the underlying metric tensor. We would also like to mention related work that investigated the link between bandwidth and kernel density estimation~\cite{BGK2010}, as well as recent progress in analyzing diffusion kernels with data-dependent bandwidth~\cite{BerryHarlim2015}; their relation with HDM will be explored in more detail in future work.

The decomposition of tangent spaces of the total manifold is not only an essential element in the HDM framework but also the source of the duality relation discussed above. In differential geometry, such a decomposition can be studied in the broader context of \emph{Riemannian submersions}, the purpose of which is to study the index theory for a family of smooth manifolds (parametrized by a base manifold). It is then important to ``blow-up'' the horizontal component of the metric so as to extract the fibre information; the approach adopted there is formally similar to the metric deformation we utilize in HDM, except that the parameter $\gamma$ multiplies the horizontal component of the metric tensor and sent to $\infty$ in the limit process (known as the \emph{Riemannian Adiabatic Limit}~\cite{BismutCheeger1989,Bismut2013}). Though there is thus a close relation between that approach and HDM, we emphasize that our main focus here is the spectral geometry of the fibres rather than their topological invariants.


%% file: discussion.tex
This paper introduced \emph{hypoelliptic diffusion maps} (HDM), a novel semi-supervised learning framework for the analysis and organization of a class of complex data sets, in which individual structures at each data point carry abundant information that can not be easily extracted away by a pairwise similarity measure. We also introduced the fibre bundle assumption, a generalization of the manifold assumption, and proved that under this assumption HDM provides embeddings for both the base and the total manifold; furthermore, the flexibility of the HDM framework enables us to view VDM and the standard diffusion maps (DM) as special cases. The rest of the paper focused on analyzing HDM on the tangent and unit tangent bundles of closed Riemannian manifolds, with convergence rate estimated for finite sampling on unit tangent bundles. These results provide the mathematical foundation for HDM on tangent bundles, and motivate further studies concerning both wider applicability and deeper mathematical understanding of the algorithmic framework. We conclude this paper with a few potential directions for further exploration.

\begin{enumerate}[1)]
\item\label{item:16} \emph{HDM on General Fibre Bundles.} We are interested in providing a more general mathematical framework for studying HDM on a wider class of fibre bundles. This is necessary and interesting, since data sets of interest to HDM (such as shape collections or persistent diagrams) are naturally modeled on fibre bundles that are more general than tangent and unit tangent bundles. The theory of shape spaces~\cite{Mumford2012} is of particular importance in this direction, since the concepts of horizontal and vertical Laplacians are readily available in the \emph{Sub-Riemannian} literature~\cite{Montgomery2006TourSubRG,Rifford2014,Baudoin2014}.
\item\label{item:17} \emph{Spectral Convergence of HDM.} The convergence results in this paper are pointwise; as in~\cite{BelkinNiyogi2007,SingerWu2013}, we believe that it is possible to show the convergence of the eigenvalues and eigenvectors of the graph hypoelliptic Laplacians to the eigenvalues and eigenvectors of the manifold hypoelliptic Laplacians, thus establishing the mathematical foundation for the spectral analysis of the HDM framework. Moreover, the hypoelliptic diffusion maps differ from diffusion maps and vector diffusion maps in that the fibres tend to be registered to a common ``template'', which, to our knowledge, is a new phenomenon that is addressed here for the first time.
\item\label{item:18} \emph{Spectral Clustering and Cheeger-Type Inequalities.} An important application of graph Laplacian is spectral clustering (graph partitioning). In a simple case, for a connected graph, the eigenvector corresponding to the smallest positive eigenvalue of the graph Laplacian partitions the graph vertices into two similarly sized subsets, in such a way that the number of edges across the subsets is as small as possible. In spectral graph theory~\cite{Chung1997}, the classical Cheeger's Inequality provides upper and lower bounds for the performance of the partition; recently, \cite{Bandeira2013} established similar results for the graph connection Laplacian, the central object of VDM. We believe that similar inequalities can be established for graph hypoelliptic Laplacians as well, with potentially more interesting behavior of the eigenvectors. For instance, we observed in practice that the eigenvector corresponding to the smallest positive eigenvalue of the graph hypoelliptic Laplacian stably partitions all the fibres in a globally consistent manner.
\item\label{item:19} \emph{Multiscale Analysis and Hierarchical Coarse-Graining.} Multiscale representation of massive, complex data sets based on similarity graphs is an interesting and fruitful application of diffusion operators~\cite{LafonLee2006,CoifmanMaggioni2006}. Based on HDM, one can build a similar theory for data sets possessing fibre bundle structures, providing a natural framework for coarse-graining that is meaningful (or even possible) only when performed simultaneously on the base and fibre manifolds. Moreover, since the hypoelliptic diffusion matrix is often of high dimensionality, an efficient approach to store and compute its powers will significantly improve the applicability of the HDM algorithm. We thus expect to develop a theory of \emph{hypoelliptic diffusion wavelets} and investigate their performance on real data sets with underlying fibre bundle structures.
\end{enumerate}


%% file: tangent_bundle.tex
In this appendix, we briefly summarize some preliminaries on the geometry of tangent and unit tangent bundles. For the Sasaki metric~\cite{Sasaki1958,Sasaki1962}, readers may find useful the expositions in \cite{Gudmundsson2002,Dombrowski1962,Musso1988}, or jump start from~\cite[Exercise 3.2]{doCarmo1992RG}; for the unit tangent bundle, some results are collected in \cite{Boeckx1997,Boeckx2005} and the references therein. We define horizontal differential operators by directly lifting vector fields from the base manifold to the fibre bundle, which in principle applies to any diffusion operators~\cite{ELL2010Filtering}.

\subsection{Coordinate Charts on Tangent Bundles}
\label{sec:coord-charts-tang}
Let $M$ be a $d$-dimensional Riemannian manifold, $TM$ its tangent bundle, and $\pi:TM\rightarrow M$ the canonical projection from $TM$ to $M$. In a local coordinate chart $\left( U; x^1,\cdots,x^d \right)$ of $M$, $\left\{ \partial\slash\partial x^1,\cdots, \partial\slash\partial x^d\right\}$ is a local frame, and we write the basis for $T_xM$ as
\begin{equation*}
  \left\{\frac{\partial}{\partial x^1}\Bigg|_x,\cdots,\frac{\partial}{\partial x^d}\Bigg|_x\right\}.
\end{equation*}
A trivialization for $TM$ on $U$ can be chosen as
\begin{equation*}
  \begin{aligned}
  \left( x,v \right)&\mapsto \left( x^1,\cdots,x^d,v^1,\cdots,v^d \right),\quad x\in U, v\in T_xM,
  \end{aligned}
\end{equation*}
and we write
\begin{equation*}
  \left\{\frac{\partial}{\partial x^1}\Bigg|_{\left( x,v \right)},\cdots,\frac{\partial}{\partial x^d}\Bigg|_{\left( x,v \right)},\frac{\partial}{\partial v^1}\Bigg|_{\left( x,v \right)},\cdots,\frac{\partial}{\partial v^d}\Bigg|_{\left( x,v \right)}\right\},
\end{equation*}
for a natural basis for $T_{\left( x,v \right)}TM$. It is immediately verifiable that
\begin{equation*}
  d\pi_{\left( x,v \right)}\left(\frac{\partial}{\partial x^j}\Bigg|_{\left( x,v \right)}\right)=\frac{\partial}{\partial x^j}\Bigg|_x,\quad d\pi_{\left( x,v \right)}\left(\frac{\partial}{\partial v^j}\Bigg|_{\left( x,v \right)}\right)=0,\quad j=1,\cdots,d,
\end{equation*}
where $d\pi_{\left( x,v \right)}:T_{\left( x,v \right)}TM\rightarrow T_xM$ denotes the differential of the canonical projection $\pi$ at $\left( x,v \right)\in TM$. Note our usage of ``$|_x$'' and ``$|_{\left( x,v \right)}$'' to distinguish tangent vectors in $T_xM$ or $T_{\left( x,v \right)}TM$.

Even when a connection is not present, the \emph{vertical tangent vectors} to $TM$ are well-defined. It suffices to take the subspace spanned by tangent vectors along the $v$-coordinates
\begin{equation}
  \label{eq:vertical_tangent_vector}
  VT_{\left( x,v \right)}M:=\textrm{span}\left\{\frac{\partial}{\partial v^j}\Bigg|_{\left( x,v \right)},\quad j=1\cdots,d \right\}=\mathrm{Ker}\left( d\pi_{\left( x,v \right)} \right).
\end{equation}
We call the subbundle of $TM$ consisting of all vertical tangent vectors the \emph{vertical tangent bundle}
\begin{equation}
  \label{eq:vertical_tangent_bundle}
  VTM:=\coprod_{\left(x,v\right)\in TM}VT_{\left( x,v \right)}M=\mathrm{Ker}\left( d\pi \right).
\end{equation}
This immediately gives
\begin{equation*}
  T\left(TM\right) \Big\slash VTM\simeq\pi^{*}TM.
\end{equation*}
Using the Levi-Civita connection on $M$, we can find another subbundle of $T \left( TM \right)$, called the \emph{horizontal tangent bundle}
\begin{equation}
  \label{eq:horizontal_tangent_bundle}
  HTM:=\coprod_{\left(x,v\right)\in TM}HT_{\left( x,v \right)}M
\end{equation}
where
\begin{equation}
  \label{eq:horizontal_tangent_vector}
  HT_{\left( x,v \right)}M:=\textrm{span}\left\{\frac{\partial}{\partial x^j}\Bigg|_{\left( x,v \right)}\!\!\!\!-\Gamma_{\alpha j}^{\beta}\left( x \right)v^{\alpha} \frac{\partial}{\partial v^{\beta}}\Bigg|_{\left( x,v \right)},\quad j=1\cdots,d \right\}.
\end{equation}
The symbols $\Gamma_{\alpha j}^{\beta}$ are the \emph{connection coefficients} of the Levi-Civita connection, or the \emph{Christoffel symbols}. The tangent bundle of $TM$ splits into the direct sum of its horizontal and vertical components
\begin{equation}
  \label{eq:splitting}
  T\left(TM\right)=HTM\oplus VTM.
\end{equation}

The \emph{Sasaki metric} is a \emph{natural metric}~\cite{Gudmundsson2002} on $TM$. For two tangent vectors $X,Y\in T_{\left( x,v \right)}TM$, choose curves in $TM$
\begin{equation*}
  \alpha:t\mapsto \left( p \left( t \right),u \left( t \right) \right),\quad\beta:s\mapsto \left(q \left( s \right),w \left( s \right)  \right),
\end{equation*}
such that
\begin{equation*}
  p \left( 0 \right) = q \left( 0 \right) = x, \quad u \left( 0 \right)=w \left( 0 \right)=v,
\end{equation*}
and define
\begin{equation*}
  \langle X,Y \rangle_{\left(x,v\right)} = \langle d\pi \left( X \right),d\pi \left( Y \right)\rangle_x+\left\langle \frac{D u}{dt}\left( 0 \right),\frac{D w}{ds}\left( 0 \right)\right\rangle_x,
\end{equation*}
where $Du/dt$ and $Dw/ds$ are covariant derivatives. Using the horizontal-vertical splitting of $T \left( TM \right)$, this metric can be equivalently defined as
\begin{equation}
  \label{eq:sasaki_metric}
  \begin{aligned}
    \langle X,Y \rangle_{\left( x,v \right)} &= \langle d\pi\left(X\right),d\pi\left(Y\right) \rangle_x\quad \textrm{if } X,Y\in HT_{\left( x,v \right)}M,\\
    \langle X,Y \rangle_{\left( x,v \right)} &= \langle X,Y \rangle_x\quad \textrm{if } X,Y\in VT_{\left( x,v \right)}M,\\
    \langle X,Y\rangle_{\left( x,v \right)} &= 0\quad \textrm{if } X\in HT_{\left( x,v \right)}M, Y\in VT_{\left( x,v \right)}M.
  \end{aligned}
\end{equation}
In words, the Sasaki metric imposes orthogonality between horizontal and vertical tangent bundles, and adopts metrics on $HTM$ and $VTM$ induced from the Riemannian metric on $M$.

\subsection{Horizontal and Vertical Differential Operators on Tangent Bundles}
\label{sec:horiz-vert-diff}
Let $\Gamma$ denote the Christoffel symbols of the Levi-Civita connection on $M$. Define the \emph{horizontal lift operator} $\mathscr{L}: T_xM\rightarrow T_{\left( x,v \right)}TM$, from the tangent space of $M$ at $x\in M$ to the tangent space of $TM$ at $\left( x,v \right)\in TM$, by
\begin{equation*}
  \mathscr{L}\left(\frac{\partial}{\partial x^j}\Bigg|_x\right)= \frac{\partial}{\partial x^j}\Bigg|_{\left( x,v \right)}\!\!\!\!-\Gamma_{\alpha j}^{\beta}\left( x \right)v^{\alpha} \frac{\partial}{\partial v^{\beta}}\Bigg|_{\left( x,v \right)}.
\end{equation*}
It is direct to verify that this definition is independent of coordinates, as for \eqref{eq:horizontal_tangent_vector}.

$\mathscr{L}$ can be used to define ``horizontal'' first order differential operators on $TM$: just lift vector fields from $M$ to $TM$. For instance, the \emph{gradient operator} on $M$, denoted as $\nabla:C^{\infty}\left( M \right)\rightarrow\Gamma \left( TM \right)$, is defined as
\begin{equation*}
  \begin{aligned}
      \langle \nabla f,v \rangle_g = df \left( v \right)\quad\textrm{ for all }v\in TM,
  \end{aligned}
\end{equation*}
or in local coordinates
\begin{equation*}
  \begin{aligned}
    g_{jk}\left( \nabla f \right)^jv^k&=\frac{\partial f}{\partial x^k}v^k\\
    \Rightarrow g_{jk}\left( \nabla f \right)^j&=\frac{\partial f}{\partial x^k},k=1,\cdots,d\\
    \Rightarrow \left(\nabla f\right)\big|_x&=g^{ik}\left( x \right)\frac{\partial f}{\partial x^k}\Bigg|_x\frac{\partial}{\partial x^i}\Bigg|_x.
  \end{aligned}
\end{equation*}
Note that here $f$ is a smooth function on $M$. The \emph{horizontal gradient operator} on $TM$ can be defined using $\mathscr{L}$ as follows
\begin{equation}
\label{eq:horizontal_gradient}
\begin{aligned}
  &\nabla^Hf \left( x,v \right)=g^{ik}\left( x \right)\mathscr{L}\left(\frac{\partial}{\partial x^k}\Bigg|_x\right)\mathscr{L}\left(\frac{\partial}{\partial x^i}\Bigg|_x\right)\\
  &=g^{ik}\left( x \right)\left( \frac{\partial f}{\partial x^k}\Bigg|_{\left( x,v \right)}\!\!\!\!\!-\Gamma_{\alpha k}^{\beta}\left( x \right)v^{\alpha} \frac{\partial f}{\partial v^{\beta}}\Bigg|_{\left( x,v \right)}\right)\left( \frac{\partial}{\partial x^i}\Bigg|_{\left( x,v \right)}\!\!\!\!\!-\Gamma_{\alpha i}^{\beta}\left( x \right)v^{\alpha} \frac{\partial}{\partial v^{\beta}}\Bigg|_{\left( x,v \right)}\right).
\end{aligned}
\end{equation}
Similarly, for the \emph{Laplace-Beltrami operator} on $M$
\begin{equation*}
  \Delta f\left( x \right)=\frac{1}{\sqrt{\left| g \left( x \right) \right|}}\frac{\partial}{\partial x^j}\Bigg|_x \left( \sqrt{\left| g \left( x \right) \right|}\,g^{jk}\left( x \right)\frac{\partial f}{\partial x^k}\Bigg|_x \right),
\end{equation*}
(where $\left| g \right|=\left|\mathrm{det}\,g\right|$), its horizontal counterpart is
\begin{equation}
  \label{eq:horizontal_laplace_beltrami}
  \begin{aligned}
    &\Delta^Hf \left( x,v \right)=\frac{1}{\sqrt{\left| g \left( x \right) \right|}}\mathscr{L}\left(\frac{\partial}{\partial x^j}\Bigg|_x\right) \left[ \sqrt{\left| g \left( x \right) \right|}\,g^{jk}\left( x \right)\mathscr{L}\left(\frac{\partial f}{\partial x^k}\Bigg|_x\right)\right]\\
    &=\frac{1}{\sqrt{\left| g \left( x \right) \right|}}\left( \frac{\partial}{\partial x^j}\Bigg|_{\left( x,v \right)}\!\!\!\!\!-\Gamma_{\alpha j}^{\beta}\left( x \right)v^{\alpha} \frac{\partial}{\partial v^{\beta}}\Bigg|_{\left( x,v \right)} \right)\\
    &\qquad\qquad\qquad\left[ \sqrt{\left| g \left( x \right) \right|}\,g^{jk}\left( x \right)\left(\frac{\partial f}{\partial x^k}\Bigg|_{\left( x,v \right)}\!\!\!\!\!-\Gamma_{\alpha k}^{\beta}\left( x \right)v^{\alpha} \frac{\partial f}{\partial v^{\beta}}\Bigg|_{\left( x,v \right)}  \right) \right].
  \end{aligned}
\end{equation}
In a geodesic normal neighborhood centered at some fixed $x\in M$, $\Gamma_{ij}^k \left( x \right)=0$ for all indices $1\leq i,j,k\leq d$. Consequently, \eqref{eq:horizontal_gradient} and \eqref{eq:horizontal_laplace_beltrami} simplify as
\begin{equation}
  \label{eq:horizontal_gradient_geodesic_normal}
  \nabla^Hf \left( x,v \right)=\sum_{k=1}^d\frac{\partial f}{\partial x^k}\Bigg|_{\left( x,v \right)}\frac{\partial}{\partial x^k}\Bigg|_{\left( x,v \right)}.
\end{equation}
and
\begin{equation}
\label{eq:horizontal_laplace_beltrami_geodesic_normal}
    \Delta^Hf \left( x,v \right)=\sum_{k=1}^d \frac{\partial^2 f}{\partial \left( x^k \right)^2}\Bigg|_{\left( x,v \right)}+\frac{1}{3}R_{\alpha\beta}\left( x \right)v^{\alpha}\frac{\partial f}{\partial v^{\beta}}\Bigg|_{\left( x,v \right)}.
\end{equation}

The definition of \emph{vertical differential operators} does not depend on the Levi-Civita connection. The \emph{vertical gradient operator} on $TM$ is simply
\begin{equation}
  \label{eq:vertical_gradient_operator}
  \nabla^Vf \left( x,v \right) = g^{ik}\left( x \right)\frac{\partial f}{\partial v^k}\Bigg|_{\left( x,v \right)}\frac{\partial}{\partial v^i}\Bigg|_{\left( x,v \right)},
\end{equation}
and the \emph{vertical Laplace-Beltrami operator} on $TM$
\begin{equation}
  \label{eq:vertical_laplace_beltrami}
  \begin{aligned}
    \Delta^Vf \left( x,v \right)&=\frac{1}{\sqrt{\left| g \left( x \right) \right|}}\frac{\partial}{\partial v^j}\Bigg|_{\left( x,v \right)} \left( \sqrt{\left| g \left( x \right) \right|}\,g^{jk}\left( x \right)\frac{\partial f}{\partial v^k}\Bigg|_{\left( x,v \right)} \right)\\
    &=g^{jk}\left( x \right)\frac{\partial^2 f}{\partial v^j\partial v^k}\Bigg|_{\left( x,v \right)}.
  \end{aligned}
\end{equation}
The coordinate independence of these vertical differential operators follows from the observation that the $v$-components of the coordinates ``behave like'' the $x$-components under change of coordinates:
\begin{equation}
\label{eq:vertical_change_coordinates}
  \begin{aligned}
  v&=v^j \frac{\partial}{\partial x^j}\Bigg|_x=v^j \frac{\partial\tilde{x}^k}{\partial x^j}\Bigg|_x\frac{\partial}{\partial\tilde{x}^k}\Bigg|_x=\tilde{v}^k \frac{\partial}{\partial \tilde{x}^k}\Bigg|_{x}=\tilde{v}^k \frac{\partial x^j}{\partial \tilde{x}^k}\Bigg|_{x}\frac{\partial}{\partial x^j}\Bigg|_{x},\\
   &\quad\Rightarrow v^j=\tilde{v}^k \frac{\partial x^j}{\partial \tilde{x}^k}\Bigg|_{x}, \tilde{v}^j=v^k \frac{\partial \tilde{x}^j}{\partial x^k}\Bigg|_{x},\quad\Rightarrow \frac{\partial v^j}{\partial\tilde{v}^k}=\frac{\partial x^j}{\partial\tilde{x}^k}, \frac{\partial\tilde{v}^j}{\partial v^k}=\frac{\partial\tilde{x}^j}{\partial x^k},
  \end{aligned}
\end{equation}
or equivalently, they have the same Jacobian. Therefore, the coordinate invariance of $\nabla$ and $\Delta$ is equivalent to the coordinate invariance of $\nabla^V$ and $\Delta^V$. For the same reason, the volume form on $T_xM$
\begin{equation}
  \label{eq:tangent_space_measure}
  dV_x \left( v \right)  = \sqrt{\left| g \left( x \right) \right|}\,dv^1\wedge\cdots\wedge dv^d,
\end{equation}
is also coordinate invariant, since the volume form on $M$
\begin{equation}
  \label{eq:Riemannian_volume_form}
  d\mathrm{vol}_M \left( x \right)=\sqrt{\left| g \left( x \right) \right|}\,dx^1\wedge\cdots\wedge dx^d
\end{equation}
is well-defined. The volume element on $TM$ with respect to the Sasaki metric is locally the product of the volume forms in \eqref{eq:tangent_space_measure} and \eqref{eq:Riemannian_volume_form}. According to \emph{Liouville's Theorem}, this volume element is invariant under geodesic flows (see, e.g.,\cite[Exercise 3.14]{doCarmo1992RG}). Note that the existence of a volume element on a general fibre bundle is not guaranteed; the tangent bundle is special in that its total manifold is always orientable regardless of the orientability of its base manifold (see, e.g., \cite[Exercise 0.2]{doCarmo1992RG}). The volume element on the tangent bundle of $M$ induces a volume element on the unit tangent bundle of $M$, as we shall see below.

\subsection{The Unit Tangent Bundle as a Subbundle}
\label{sec:unit-tangent-bundle}

The hypoelliptic diffusion map constructed in Section \ref{sec:hypo-diff-formulation} works for both compact and non-compact manifolds. In practice, due to the constraint of finite sampling, we prefer to apply HDM to a compact object. Assuming $M$ is compact is not sufficient, since $TM$ is always non-compact. This motivates us to apply HDM to the unit tangent bundle $UTM$, a natural subbundle of $TM$ with compact fibres.

Following the notations used in Section \ref{sec:coord-charts-tang}, a \emph{unit tangent bundle} over a Riemannian manifold $\left(M,g\right)$ is a subbundle of $TM$, with fibre over $x\in M$ consisting of the tangent vectors in $T_xM$ with unit length:
\begin{equation*}
  UTM := \coprod_{x\in M}S_x,\quad S_x:=\left\{ v\in T_xM\mid g_x \left( v,v \right)=1  \right\} \subset T_xM,
\end{equation*}
where $g_x \left( \cdot,\cdot \right)$ denotes for the inner product on $T_xM$, defined by the Riemannian metric on $M$. Note that $UTM$ is a hypersurface of $TM$, and thus induces a metric from that of $TM$. The volume form on $UTM$ with respect to the induced metric, known as the \emph{Liouville measure} or the \emph{kinematic density}~\cite[Chapter VII]{Chavel2006}, is the only invariant measure on $UTM$ under geodesic flows.

We can define the gradient and the Laplace-Beltrami operator on $S_x$ with respect to the metric on $UTM$, denoted as $\nabla_{S_x}$ and $\Delta_{S_x}$, respectively. The \emph{vertical spherical gradient} $\nabla_S^V$ and \emph{vertical spherical Laplace-Beltrami operator} $\nabla_S^V$ can be defined through $\nabla_{S_x}$ and $\Delta_{S_x}$, similar to \eqref{eq:vertical_gradient_operator} and \eqref{eq:vertical_laplace_beltrami}:
\begin{equation}
  \label{eq:vertical_spherical_gradient_laplace_beltrami}
  \begin{aligned}
    \nabla_S^Vf \left( x,v \right)&:=\nabla_{S_x}f \left( x,v \right),\quad x\in M,\,\,v\in S_x,\\
  \Delta_S^Vf \left( x,v \right)&:=\Delta_{S_x}f \left( x,v \right),\quad x\in M,\,\,v\in S_x,
  \end{aligned}
\end{equation}
where we implicitly identified $f$ with its restriction on $S_x$ when  and  are applied to it, as long as no confusion arises.

In order to define the horizontal lifts of $\nabla$ and $\Delta$ from $M$ to $UTM$, we take advantage of the fact that $UTM$ is a subbundle of $TM$, and set
\begin{equation}
  \label{eq:horizontal_spherical_gradient_laplace_beltrami}
  \begin{aligned}
    \nabla_S^Hf \left( x,v \right)&:=\nabla^H\hat{f}\left( x,v \right),\\
    \Delta_S^Hf \left( x,v \right)&:=\Delta^H\hat{f}\left( x,v \right),
  \end{aligned}
\end{equation}
where $\hat{f}\left( x,u \right)$ is an arbitrary extension of $f$ from $C^{\infty}\left( UTM \right)$ to $C^{\infty}\left( TM \right)$.

We can show that the definition in \eqref{eq:horizontal_spherical_gradient_laplace_beltrami} does not depend on any specific choice of extensions. In fact, from \eqref{eq:horizontal_gradient}\eqref{eq:horizontal_laplace_beltrami} it is clear that the value of $\nabla^H\hat{f}, \Delta^H\hat{f}$ at $\left( x,v \right)\in S_x$ only depends on the data of $\hat{f}$ along the flow generated by vector fields
\begin{equation}
\label{eq:basic_lifts}
  \frac{\partial}{\partial x^j}\Bigg|_{\left( x,v \right)}\!\!\!\!\!-\Gamma_{\alpha j}^{\beta}\left( x \right)v^{\alpha} \frac{\partial}{\partial v^{\beta}}\Bigg|_{\left( x,v \right)},\quad j=1,\cdots,d.
\end{equation}
Thus it suffices to show that such a flow, if started at a point $\left( x,v \right)\in UTM$, will remain on $UTM$ for all time. Consider a curve $\gamma:t\mapsto TM$ starting at $\left( x,v \right)\in UTM$ that follows along the direction of one of the vector fields in~\eqref{eq:basic_lifts}, say the one indexed by $j$. Write $\gamma \left( t \right)$ in coordinates as
\begin{equation*}
  \gamma \left( t \right)=\left( x \left( t \right),v \left( t \right) \right)= \left( x^1 \left( t \right),\cdots,x^d \left( t \right),v^1 \left( t \right),\cdots,v^d \left( t \right) \right).
\end{equation*}
By construction,
\begin{equation*}
  \begin{aligned}
    \gamma' \left( t \right)&=\sum_{k=1}^d\left(\frac{dx^k \left( t \right)}{dt}\frac{\partial}{\partial x^k}\Bigg|_{\left(x \left( t \right),v \left( t \right)\right)}+\frac{dv^k \left( t \right)}{dt}\frac{\partial}{\partial v^k}\Bigg|_{\left(x \left( t \right),v \left( t \right)\right)}\right)\\
    &=\frac{\partial}{\partial x^j}\Bigg|_{\left( x \left( t \right),v \left( t \right) \right)}\!\!\!\!\!-\Gamma_{\alpha j}^{\beta}\left( x \left( t \right) \right)v^{\alpha} \left( t \right) \frac{\partial}{\partial v^{\beta}}\Bigg|_{\left( x \left( t \right),v \left( t \right)\right)},
  \end{aligned}
\end{equation*}
which implies
\begin{equation*}
  \frac{dx^j \left( t \right)}{dt}=1,\quad\frac{dx^k \left( t \right)}{dt}=0,k\neq j,k=1,\cdots,d,
\end{equation*}
and
\begin{equation*}
  \frac{dv^j \left( t \right)}{dt}=-\Gamma_{\alpha j}^{\beta}\left( x \left( t \right) \right)v^{\alpha} \left( t \right),\quad\frac{dv^k \left( t \right)}{dt}=0,k\neq j,k=1,\cdots,d,
\end{equation*}
by linear independence. In other words, $v \left( t \right)$ is indeed the parallel transport of $v \left( 0 \right)$ along curve $\pi\circ\gamma: t\mapsto x \left( t \right)$ on $M$, since $v \left( t \right)$ satisfies
\begin{equation*}
  \frac{dv^k \left( t \right)}{dt}+\Gamma_{ij}^k \left( x \left( t \right) \right)v^j \left( t \right) \frac{dx^i \left( t \right)}{dt}=0,\quad k=1,\cdots,d.
\end{equation*}
which is the same equation that defines the parallel transport on $M$, from $v \left( 0 \right)$ and long the curve $\pi\circ\gamma \left( t \right)=x \left( t \right)$. In other words, if $\gamma\left(0 \right)= \left( x \left( 0 \right),v \left( 0 \right) \right)\in UTM$, then $\gamma \left( t \right)= \left( x \left( t \right),v \left( t \right) \right)$ will stay on $UTM$ for all time, since the parallel transport is an isometry between tangent spaces and thus preserves the unit length of $v \left( 0 \right)=v$.


%% file: proof.tex
\subsection{Proofs of Theorem~\ref{thm:main_tm} and Theorem~\ref{thm:main_utm}}
\label{sec:proofs-bias-terms}

We include in this section a proof of Theorem~\ref{thm:main_utm}. The proof of Theorem~\ref{thm:main_tm} can be similarly constructed. The Einstein summation convention is assumed everywhere unless otherwise specified.

Our starting point is the following lemma, in reminiscent of \cite[Lemma 8]{CoifmanLafon2006} and \cite[Lemma B.10]{SingerWu2012VDM}.
\begin{lemma}
\label{lem:basic_kernel_integral}
  Let $\Phi:\mathbb{R}\rightarrow\mathbb{R}$ be a smooth function compactly supported in $\left[ 0,1 \right]$. Assume $M$ is a $d$-dimensional compact Riemannian manifold without boundary, with injectivity radius $\mathrm{Inj}\left( M \right)>0$. For any $\epsilon>0$, define kernel function
  \begin{equation}
    \label{eq:epsilon_kernel_family}
    \Phi_{\epsilon}\left( x,y \right)=\Phi \left( \frac{d^2_M \left( x,y \right)}{\epsilon} \right)
  \end{equation}
on $M\times M$, where $d^2_M \left( \cdot,\cdot \right)$ is the geodesic distance on $M$. If the parameter $\epsilon$ is sufficiently small such that $0\leq\epsilon\leq\sqrt{\mathrm{Inj}\left( M \right)}$, then the integral operator associated with kernel $\Phi_{\epsilon}$
\begin{equation}
  \label{eq:integral_operator}
  \left(\Phi_{\epsilon}\,g\right)\left( x \right):=\int_M \Phi_{\epsilon}\left( x,y \right)g \left( y \right)d\mathrm{vol}_M \left( y \right)
\end{equation}
has the following asymptotic expansion as $\epsilon\rightarrow 0$
\begin{equation}
  \label{eq:basic_asymp_expansion}
  \left(\Phi_{\epsilon}\,g\right)\left( x \right) = \epsilon^{\frac{d}{2}}\left[ m_0 g \left( x \right)+\epsilon \frac{m_2}{2}\left( \Delta_M g \left( x \right)-\frac{1}{3}\mathrm{Scal}\left( x \right)g \left( x \right) \right)+O \left( \epsilon^2 \right) \right],
\end{equation}
where $m_0,m_2$ are constants that depend on the moments of $\Phi$ and the dimension $d$ of the Riemannian manifold $M$, $\Delta_M$ is the Laplace-Beltrami operator on $M$, and $\mathrm{Scal}\left( x \right)$ is the scalar curvature of $M$ at $x$.
\end{lemma}

\begin{proof}
  We put everything in geodesic normal coordinates centered at $x\in M$. If $d_M \left( x,y \right)=r$, let $y$ have geodesic normal coordinates $\left(s^1,\cdots,s^d\right)$ such that $\left(s^1\right)^2+\cdots+\left(s^d\right)^2=r^2$. Note that
  \begin{equation}
  \label{lem1:geodesicnormalcoords}
    \begin{aligned}
      \int_M \Phi_{\epsilon}\left( x,y \right)g \left( y \right)\,d\mathrm{vol}_M \left( y \right) &= \int_{B_{\sqrt{\epsilon}}\left( x \right)} \Phi\left(\frac{d^2_M \left( x,y \right)}{\epsilon}\right)g \left( y \right)\,d\mathrm{vol}_M \left( y \right)\\
      &=\int_{B_{\sqrt{\epsilon}}\left( 0 \right)}\Phi \left( \frac{r^2}{\epsilon} \right)\tilde{g} \left(s\right)\,d\mathrm{vol}_M \left(s\right)
    \end{aligned}
  \end{equation}
where
\begin{equation*}
  \begin{aligned}
    \tilde{g}\left( s\right)=\tilde{g}\left( s^1,\cdots,s^d \right)&=g \circ \mathrm{exp}_x \left( s^1e_1+\cdots+s^de_d \right),\\
    d\mathrm{vol}_M \left( s \right)=d\mathrm{vol}_M \left( s^1,\cdots,s^d \right)&=d\mathrm{vol}_M \left( \mathrm{exp}_x \left( s^1e_1+\cdots+s^de_d \right) \right)
  \end{aligned}
\end{equation*}
 with $\left\{e_1,\cdots,e_d\right\}$ being an orthonormal basis for $T_xM$. A further change of variables
$$\tilde{s}^1=\frac{s^1}{\sqrt{\epsilon}},\cdots,\tilde{s}^d=\frac{s^d}{\sqrt{\epsilon}}; \quad\tilde{r}=\frac{r}{\sqrt{\epsilon}}$$
leads to 
\begin{equation}
\label{lem1:changeofcoordinates}
  \begin{aligned}
    \int_{B_{\sqrt{\epsilon}}\left( 0 \right)} & \Phi \left( \frac{r^2}{\epsilon} \right)\tilde{g} \left(s \right)\,d\mathrm{vol}_M \left(s \right)=\int_{B_1\left( 0 \right)}\Phi \left(\tilde{r}^2 \right)\tilde{g} \left(\sqrt{\epsilon}\,\tilde{s} \right)d\mathrm{vol}_M \left(\sqrt{\epsilon}\,\tilde{s}\right)\\
  &=\int_{B_1\left( 0 \right)}\Phi \left(\tilde{r}^2 \right)\tilde{g} \left(\sqrt{\epsilon}\,\tilde{s}^1,\cdots,\sqrt{\epsilon}\,\tilde{s}^d  \right)d\mathrm{vol}_M \left(\sqrt{\epsilon}\,\tilde{s}^1,\cdots,\sqrt{\epsilon}\,\tilde{s}^d \right).
  \end{aligned}
\end{equation}
Recall \cite{Petersen2006} that in geodesic normal coordinates
\begin{equation*}
  d\mathrm{vol}_M \left( s^1,\cdots,s^d \right)=\left[1-\frac{1}{6}R_{kl}\left( x \right)s^ks^l+O \left( r^3 \right)\right]ds^1\cdots ds^d
\end{equation*}
where $R_{kl}$ is the Ricci curvature tensor
\begin{equation*}
  R_{kl} \left( x \right)=g^{ij}R_{kilj}\left( x \right).
\end{equation*}
Thus
\begin{equation}
\label{lem1:dvol}
  d\mathrm{vol}_M \left( \sqrt{\epsilon}\,\tilde{s}^1,\cdots,\sqrt{\epsilon}\,\tilde{s}^d \right)=\left[1-\frac{\epsilon}{6}R_{kl}\left( x \right)\tilde{s}^k\tilde{s}^l+O \left(\epsilon^{\frac{3}{2}}\tilde{r}^3 \right)\right]\cdot \epsilon^{\frac{d}{2}}d\tilde{s}^1\cdots d\tilde{s}^d.
\end{equation}
In the meanwhile, the Taylor expansion of $\tilde{g}\left( s^1,\cdots,s^d \right)$ near $x$ reads
\begin{equation*}
  \tilde{g}\left( s^1,\cdots,s^d \right)=\tilde{g}\left( 0 \right)+\frac{\partial\tilde{g}}{\partial s^j}\left( 0 \right)s^j+ \frac{1}{2}\frac{\partial^2\tilde{g}}{\partial s^k\partial s^l}\left( 0 \right)s^ks^l+O \left( r^3 \right)
\end{equation*}
and thus
\begin{equation}
\label{lem1:taylorexpansion}
  \tilde{g} \left(\sqrt{\epsilon}\,\tilde{s}^1,\cdots,\sqrt{\epsilon}\,\tilde{s}^d  \right)=g \left( x \right)+\sqrt{\epsilon}\cdot \frac{\partial\tilde{g}}{\partial s^j}\left( 0 \right)\tilde{s}^j+\epsilon\cdot \frac{1}{2} \frac{\partial^2\tilde{g}}{\partial s^k\partial s^l}\left( 0 \right)\tilde{s}^k\tilde{s}^l+O \left( \epsilon^{\frac{3}{2}}\tilde{r}^3 \right).
\end{equation}
Combining (\ref{lem1:geodesicnormalcoords})--(\ref{lem1:taylorexpansion}) and noting that
\begin{equation*}
  \int_{B_1 \left( 0 \right)}\Phi \left( \tilde{r}^2 \right)\cdot\sqrt{\epsilon}\cdot \frac{\partial\tilde{g}}{\partial s^j}\left( 0 \right)\tilde{s}^j\cdot\epsilon^{\frac{d}{2}}d\tilde{s}^1\cdots d\tilde{s}^d=0
\end{equation*}
by the symmetry of the kernel and the domain of integration $B_1 \left( 0 \right)$, we have
\begin{equation*}
  \begin{aligned}
    \int_M & \Phi_{\epsilon}\left( x,y \right)g \left( y \right)\,d\mathrm{vol}_M \left( y \right)=\int_{B_1\left( 0 \right)}\!\!\Phi \left(\tilde{r}^2\right)\tilde{g} \left(\sqrt{\epsilon}\,\tilde{s}^1,\cdots,\sqrt{\epsilon}\,\tilde{s}^d  \right)\,d\mathrm{vol}_M \left(\sqrt{\epsilon}\,\tilde{s} \right)\\
    &=\int_{B_1 \left( 0 \right)}\!\!\Phi \left( \tilde{r}^2 \right)\left[g \left( x \right)+\sqrt{\epsilon}\cdot \frac{\partial\tilde{g}}{\partial s^j}\left( 0 \right)\tilde{s}^j+\epsilon\cdot \frac{1}{2} \frac{\partial^2\tilde{g}}{\partial s^k\partial s^l}\left( 0 \right)\tilde{s}^k\tilde{s}^l+O \left( \epsilon^{\frac{3}{2}}\tilde{r}^3 \right)  \right]\cdot\\
    &\phantom{aaaaaaaaaaaaaaaaaaaaaaaa}\left[1-\frac{\epsilon}{6}R_{kl}\left( x \right)\tilde{s}^k\tilde{s}^l+O \left(\epsilon^{\frac{3}{2}}\tilde{r}^3 \right)\right]\cdot \epsilon^{\frac{d}{2}}d\tilde{s}^1\cdots d\tilde{s}^d\\
    &=\epsilon^{\frac{d}{2}}\left[ g \left( x \right)\int_{B_1 \left( 0 \right)}\!\!\Phi \left( \tilde{r}^2 \right)\,d\tilde{s}+\epsilon\left(\frac{1}{2} \frac{\partial^2\tilde{g}}{\partial s^k\partial s^l}\left( 0 \right)-\frac{1}{6}g \left( x \right)R_{kl}\left( x \right) \right)\int_{B_1 \left( 0 \right)}\!\!\Phi \left( \tilde{r}^2\right)\tilde{s}^k\tilde{s}^l \,d\tilde{s}+O \left( \epsilon^2 \right) \right]\\
    &=\epsilon^{\frac{d}{2}}\left[ g \left( x \right)\int_{B_1 \left( 0 \right)}\!\!\Phi \left( \tilde{r}^2\right)\,d\tilde{s} +\epsilon\left( \frac{1}{2} \frac{\partial^2\tilde{g}}{\partial \left(s^k\right)^2}\left( 0 \right)-\frac{1}{6}g \left( x \right)R_{kk}\left( x \right) \right)\int_{B_1 \left( 0 \right)}\!\!\Phi \left( \tilde{r}^2 \right)\left( \tilde{s}^k \right)^2d\tilde{s}+O \left( \epsilon^2 \right) \right]
  \end{aligned}
\end{equation*}
where the last equality follows from the observation that
\begin{equation*}
  \int_{B_1 \left( 0 \right)}\!\!\Phi \left( \tilde{r}^2 \right)\tilde{s}^k\tilde{s}^l\,d\tilde{s}^1\cdots d\tilde{s}^d=0\quad\textrm{ if }k\neq l
\end{equation*}
again by the symmetry of the kernel and the domain of integration $B_1 \left( 0 \right)$; the $O \left( \epsilon^{\frac{3}{2}} \right)$ term vanishes due to the symmetry of the kernel (as argued in~\cite[\S 2]{Singer2006ConvergenceRate}). The constants $m_0,m_2$ can now be explicitly characterized:
\begin{equation*}
  \begin{aligned}
    m_0&:=\int_{B_1 \left( 0 \right)}\!\!\Phi \left( \tilde{r}^2 \right)\,d\tilde{s}^1\cdots d\tilde{s}^d=\int_0^1\Phi \left( \tilde{r}^2 \right)\tilde{r}^{d-1}d\tilde{r}\int_{S_1 \left( 0 \right)}d\sigma=\omega^{d-1}\int_0^1\Phi \left( \tilde{r}^2 \right)\tilde{r}^{d-1}d\tilde{r},\\
    m_2 &:= \int_{B_1 \left( 0 \right)}\!\!\Phi \left( \tilde{r}^2 \right)\left( \tilde{s}^k \right)^2\,d\tilde{s}^1\cdots d\tilde{s}^d\quad\textrm{independent of $k\in \left\{ 1,\cdots,d \right\}$ by symmetry.}
  \end{aligned}
\end{equation*}
Finally, recall that in geodesic normal coordinates
\begin{equation*}
  \begin{aligned}
    \sum_{k=1}^d \frac{\partial^2\tilde{g}}{\partial \left( s^k \right)^2}\left( 0 \right)=\Delta_M g \left( x \right),\quad\sum_{k=1}^dR_{kk}\left( x \right)=\mathrm{Scal}\left( x \right),
  \end{aligned}
\end{equation*}
from which the conclusion follows:
\begin{equation*}
  \begin{aligned}
    \int_M \Phi_{\epsilon}\left( x,y \right)g \left( y \right)\,d\mathrm{vol}_M \left( y \right)&=\epsilon^{\frac{d}{2}}\left[ m_0 g \left(x \right) +\epsilon m_2\left( \frac{1}{2} \Delta_Mg \left( x \right)-\frac{1}{6}\mathrm{Scal}\left( x \right)g \left( x \right) \right) +O \left( \epsilon^2 \right) \right]\\
      &=\epsilon^{\frac{d}{2}}\left[ m_0 g \left(x \right) +\epsilon \frac{m_2}{2}\left( \Delta_Mg \left( x \right)-\frac{1}{3}\mathrm{Scal}\left( x \right)g \left( x \right) \right) +O \left( \epsilon^2 \right) \right].
  \end{aligned}
\end{equation*}
\end{proof}

In order to study $K_{\epsilon,\delta}$, we need an expansion for the parallel transport term $P_{y,x}v$. This is established in Lemma~\ref{lem:taylor_parallel_transport}.
\begin{lemma}
\label{lem:taylor_parallel_transport}
  Let $M$ be a Riemannian manifold, $x\in M$, $v\in T_xM$. In geodesic normal coordinates around $x$, the parallel transport of $v$ along a geodesic $\gamma:t\mapsto \mathrm{exp}_xt\theta$ ($t\in \left[ 0,\epsilon \right],\|\theta\|_{T_xM}=1$) has the following asymptotic expansion:
\begin{equation*}
  \left(P_{\mathrm{exp}_xt\theta,x}\left(v\right)\right)^j=v^j-\frac{t^2}{6}\theta^k\theta^{\sigma}v^l \left( R_{l\sigma k}^{\phantom{l\sigma k}j}\left( x \right)+R_{k\sigma l}^{\phantom{l\sigma k}j}\left( x \right) \right)+O \left( t^3 \right)
\end{equation*}
where $P_{y,x}:T_xM\rightarrow T_yM$ denotes the parallel transport from $T_xM$ to $T_yM$ along the geodesic segment connecting $x$ and $y$.
\end{lemma}
\begin{proof}
  Let $\left\{s^1,\cdots,s^d\right\}$ be a geodesic normal coordinate chart centered at $x\in M$. Assume $v\in T_xM$ takes the expression
  \begin{equation*}
    v = \sum_{j=1}^dv^j \frac{\partial}{\partial s^j}\Bigg|_x
  \end{equation*}
and let $V:\left[ 0,\epsilon \right]\rightarrow TM$ be the parallel transported vector field along the given geodesic $\gamma$
\begin{equation*}
  V \left( t \right)=\sum_{j=1}^dV^j \left( t \right)\frac{\partial}{\partial s^j}\Bigg|_{\gamma \left( t \right)}.
\end{equation*}
Note that $\left( s^1 \right)^2+\cdots+\left( s^d \right)^2=t^2$. Recall that $V$ being parallel along $\gamma$ means
\begin{equation*}
  0 = \nabla_{\gamma' \left( t \right)}V \left( t \right)=\sum_{j=1}^d \left( \frac{dV^j}{dt}\left( t \right)+\sum_{k,l=1}^d \frac{d\gamma^k}{dt}\left( t \right)V^l \left( t \right)\Gamma_{kl}^j \left( \gamma \left( t \right) \right) \right)\frac{\partial}{\partial x^j}\Bigg|_{\gamma \left( t \right)}
\end{equation*}
or equivalently
\begin{equation}
\label{eq:paralleltransport}
  \frac{dV^j}{dt}\left( t \right)+\sum_{k,l=1}^d \frac{d\gamma^k}{dt}\left( t \right)V^l \left( t \right)\Gamma_{kl}^j \left( \gamma \left( t \right)\right) = 0\quad\textrm{ for all }j=1,\cdots,d.
\end{equation}
Let $\left( t;\theta^1,\cdots,\theta^d \right)$ be the geodesic polar coordinates corresponding to $\left( s^1,\cdots,s^d \right)$. By iteratively using (\ref{eq:paralleltransport}),
\begin{equation*}
  \begin{aligned}
    V^j \left( 0 \right)&=v^j\\
    \frac{dV^j}{dt}\left( 0 \right)&=-\sum_{k,l=1}^d \frac{d\gamma^k}{dt}\left( 0 \right)V^l \left( 0 \right)\Gamma_{kl}^j \left( \gamma \left( 0 \right)\right)=0\\
    \frac{d^2V^j}{dt^2}\left( 0 \right)&=-\sum_{k,l=1}^d \frac{d}{dt}\Bigg|_{t=0} \left(\frac{d\gamma^k}{dt}\left( t \right) V^l \left( t \right)\right)\Gamma_{kl}^j \left( \gamma \left( 0 \right)\right)-\sum_{k,l=1}^d \frac{d\gamma^k}{dt}\left( 0 \right)V^l \left( 0 \right)\frac{d}{dt}\Bigg|_{t=0} \Gamma_{kl}^j \left( \gamma \left( t \right)\right)\\
    &=-\sum_{k,l=1}^d\theta^k v^l\sum_{\sigma=1}^d\partial_\sigma\Gamma_{kl}^j \left( \gamma \left( 0 \right) \right)\theta^{\sigma}=-\frac{1}{3}\theta^k\theta^{\sigma} v^l \left( R_{l\sigma k}^{\phantom{l\sigma k}j}\left( x \right)+R_{k\sigma l}^{\phantom{k\sigma l}j}\left( x \right) \right)
  \end{aligned}
\end{equation*}
where the last equality for $\displaystyle \frac{d^2V^j}{dt^2}$ follows from a simple calculation of Christoffel symbols, as follows. Recall that in geodesic normal coordinates
\begin{equation*}
  g_{ij}=\delta_{ij}+\frac{1}{3}R_{iklj}\left( x \right)s^ks^l+O \left( t^3 \right)
\end{equation*}
hence
\begin{equation*}
  \partial_\mu g_{ij}=\frac{1}{3}R_{i\mu lj}\left( x \right) s^l+\frac{1}{3}R_{ik\mu j}\left( x \right)s^k+O \left( t^2 \right)=\frac{1}{3}s^k \left(R_{i\mu k j}\left( x \right)+R_{ik\mu j}\left( x \right)  \right)+O \left( t^2 \right).
\end{equation*}
Plugging these partial derivatives into the expression of Christofeel symbols to obtain
\begin{equation*}
  \begin{aligned}
    \Gamma_{kl}^j&=\frac{1}{2}g^{\nu j}\left( \partial_k g_{l\nu}+\partial_l g_{k\nu}-\partial_{\nu}g_{kl} \right)\\
    &=\frac{1}{2}\left[ \delta^{\nu j}+O \left( t^2 \right) \right]\times\\
    &\qquad\frac{1}{3}\left[s^{\sigma}\left( R_{lk\sigma\nu}\left( x \right)+R_{l\sigma k\nu}\left( x \right)+R_{kl\sigma\nu}\left( x \right)+R_{k\sigma l\nu}\left( x \right)-R_{k\nu \sigma l}\left( x \right)-R_{k\sigma\nu l}\left( x \right) \right)  +O \left( t^2 \right)\right]\\
    &=\frac{1}{6}\delta^{\nu j}s^{\sigma}\left[2R_{l\sigma k\nu}\left( x \right)+2R_{k\sigma l\nu}\left( x \right)  \right]+O \left( t^2 \right)=\frac{1}{3}s^{\sigma}\left(R_{l\sigma k}^{\phantom{l\sigma k}j} \left( x \right)+R_{k\sigma l}^{\phantom{k\sigma l}j}\left( x \right)\right)+O \left( t^2 \right),
  \end{aligned}
\end{equation*}
therefore
\begin{equation*}
  \partial_{\sigma}\Gamma_{kl}^j \left( x \right)=\frac{1}{3}\left(R_{l\sigma k}^{\phantom{l\sigma k}j}\left( x \right)+R_{k\sigma l}^{\phantom{k\sigma l}j}\left( x \right)\right)+O \left( t \right)
\end{equation*}
which verifies the last equality for $\displaystyle\frac{\mathrm{d}^2V^j}{\mathrm{d}t^2}\left( 0 \right)$. Therefore, we have the following Taylor expansion for $V^j \left( t \right)$ up to the second order:
\begin{equation*}
  \begin{aligned}
  V^j \left( t \right)&=v^j-\frac{t^2}{6}\theta^k\theta^{\sigma}v^l \left( R_{l\sigma k}^{\phantom{l\sigma k}j}\left( x \right)+R_{k\sigma l}^{\phantom{l\sigma k}j}\left( x \right) \right)+O \left( t^3 \right),
  \end{aligned}
\end{equation*}
which establishes the desired conclusion.
\end{proof}

\begin{remark}
\label{rem:asymptotic_expansion_parallel_transport}
It is also useful to note that
\begin{equation*}
  V^j \left( t \right)=v^j-\frac{1}{6}s^ks^{\sigma}v^l \left( R_{l\sigma k}^{\phantom{l\sigma k}j}\left( x \right)+R_{k\sigma l}^{\phantom{l\sigma k}j}\left( x \right) \right)+O \left( t^3 \right).
\end{equation*}
where $\left( s^1,\cdots,s^d \right)$ are the geodesic normal coordinates of $y$ in the proof of Lemma~\ref{lem:taylor_parallel_transport}. Moreover, it is not hard to obtain higher order terms in the asymptotic expansion. For instance, differentiating both sides of \eqref{eq:paralleltransport} twice, we have
\begin{equation}
\label{eq:third_order_derivative}
  \begin{aligned}
    \frac{d^3V^j}{dt^3}\left( 0 \right)&=-\sum_{k,l=1}^d \frac{d^2}{dt^2}\bigg|_{t=0} \left(\frac{d\gamma^k}{dt}\left( t \right)V^l \left( t \right)\Gamma_{kl}^j \left( \gamma \left( t \right)\right)\right)\\
    &=-\sum_{k,l=1}^d \frac{d^2}{dt^2}\Bigg|_{t=0}\!\!\!\left(\frac{d\gamma^k}{dt}\left( t \right) V^l \left( t \right)\right)\Gamma_{kl}^j \left( \gamma \left( 0 \right)\right)-\sum_{k,l=1}^d \frac{d\gamma^k}{dt}\left( 0 \right)V^l \left( 0 \right)\frac{d^2}{dt^2}\Bigg|_{t=0}\!\!\!\Gamma_{kl}^j \left( \gamma \left( t \right)\right)\\
    &\qquad-2\sum_{k,l=1}^d \frac{d}{dt}\Bigg|_{t=0}\!\!\!\left(\frac{d\gamma^k}{dt}\left( t \right) V^l \left( t \right)\right)\frac{d}{dt}\Bigg|_{t=0}\!\!\!\Gamma_{kl}^j \left( \gamma \left( t \right)\right)
  \end{aligned}
\end{equation}
Note that $\Gamma_{kl}^j \left( \gamma \left( 0 \right) \right)=0$ since the coordinate system is geodesic, the first term in the right hand side of ~\eqref{eq:third_order_derivative} vanishes. In the meanwhile, since in geodesic normal coordinates the parametrization of geodesic $\gamma$ is linear, we have
\begin{equation*}
  \frac{d^2\gamma^k}{dt^2}\left( 0 \right) = 0,\quad k=1,\cdots,d.
\end{equation*}
Combining this observation with the computation
\begin{equation*}
  \frac{dV^j}{dt}\left( 0 \right)=-\sum_{k,l=1}^d \frac{d\gamma^k}{dt}\left( 0 \right)V^l \left( 0 \right)\Gamma_{kl}^j \left( \gamma \left( 0 \right)\right)=0,
\end{equation*}
we conclude that the last term in the right hand side of ~\eqref{eq:third_order_derivative} also vanishes. Therefore, ~\eqref{eq:third_order_derivative} is left with
\begin{equation*}
  \begin{aligned}
    \frac{d^3V^j}{dt^3}\left( 0 \right)&=-\sum_{k,l=1}^d \frac{d\gamma^k}{dt}\left( 0 \right)V^l \left( 0 \right)\frac{d^2}{dt^2}\Bigg|_{t=0}\!\!\!\Gamma_{kl}^j \left( \gamma \left( t \right)\right)\\
    &=\theta^kv^l\partial_{\sigma\omega}\Gamma_{kl}^j \left( x \right)\theta^{\sigma}\theta^{\omega}=\theta^k\theta^{\sigma}\theta^{\omega}v^l\partial_{\sigma\omega}\Gamma_{kl}^j \left( x \right).
  \end{aligned}
\end{equation*}
We can compute the second order derivative of the Christoffel symbol at $x$ (see, e.g., ~\cite{GuarreraJohnsonWolfe2002}) and completely determine the $O \left( t^3 \right)$ term; but this is less crucial for our application. The only point in carrying through the computation of the third order derivative of $V^j \left( t \right)$ is that the $O \left( t^3 \right)$ term in the expansion of Lemma~\ref{lem:taylor_parallel_transport} is a third order homogeneous polynomial in the geodesic coordinates $\left( s^1,\cdots,s^d \right)$:
\begin{equation*}
  t^3\frac{d^3V^j}{dt^3}\left( 0 \right)=t^3\theta^k\theta^{\sigma}\theta^{\omega}v^l\partial_{\sigma\omega}\Gamma_{kl}^j \left( x \right)=s^ks^{\sigma}s^{\omega}v^l\partial_{\sigma\omega}\Gamma_{kl}^j \left( x \right),
\end{equation*}
an observation that is necessary for dropping the higher order error term in Lemma~\ref{lem:kernel_parallel_transport} from $O \left( \epsilon^{\frac{3}{2}} \right)$ to $\epsilon^2$, as we will see below.
\end{remark}

Armed with Lemma~\ref{lem:basic_kernel_integral} and Lemma~\ref{lem:taylor_parallel_transport}, we are ready to take a step at analyzing $K_{\epsilon,\delta}$. Lemma~\ref{lem:kernel_parallel_transport} starts our investigation of kernel functions incorporated with parallel-transports. Virtually it only deals with $K_{\epsilon,\delta}$ with $\delta\rightarrow0$, but we'll soon see that it opens the door for understanding much more general cases.
\begin{lemma}
\label{lem:kernel_parallel_transport}
  Following Lemma \ref{lem:basic_kernel_integral}, let $P_{y,x}:T_xM\rightarrow T_yM$ denote the parallel transport from $T_xM$ to $T_yM$ determined by the Levi-Civita connection on $M$. For any function $f\in C^{\infty}\left( TM \right)$, as $\epsilon\rightarrow 0$,
\begin{equation}
  \label{eq:kernel_parallel_transport}
  \begin{aligned}
    \int_M &\Phi_{\epsilon}\left( x,y \right)f \left( y,P_{y,x}v \right)\,d\mathrm{vol}_M \left( y \right)\\
    &=\epsilon^{\frac{d}{2}}\left\{m_0 f \left( x,v \right)+\epsilon \frac{m_2}{2}\left[ \Delta^H f \left( x,v \right)-\frac{1}{3}\mathrm{Scal}\left( x \right)f \left( x,v \right) \right]+O \left( \epsilon^2 \right) \right\},
  \end{aligned}
\end{equation}
where $\Delta^H$ is the horizontal Laplacian on $TM$ defined in~\eqref{eq:horizontal_laplace_beltrami}.
\end{lemma}

\begin{proof}
Consider the geodesic normal neighborhood around $x\in M$, with $\epsilon>0$ sufficiently small such that a geodesic ball of radius $\sqrt{\epsilon}$ centered at $x$ is contained in this neighborhood. The integral is actually supported only on such a geodesic ball, due to the compact support of $\Phi$
\begin{equation*}
  \begin{aligned}
    \int_M \Phi_{\epsilon}\left( x,y \right)f \left( y,P_{y,x}v \right)\,d\mathrm{vol}_M \left( y \right) &= \int_M \Phi\left( \frac{d^2_M \left( x,y \right)}{\epsilon} \right)f \left( y,P_{y,x}v \right)\,d\mathrm{vol}_M \left( y \right)\\
    &=\int_{B_{\sqrt{\epsilon}}\left( x \right)} \Phi\left( \frac{d^2_M \left( x,y \right)}{\epsilon} \right)f \left( y,P_{y,x}v \right)\,d\mathrm{vol}_M \left( y \right).
  \end{aligned}
\end{equation*}
Express $y\in B_{\sqrt{\epsilon}}\left( x \right)$ in geodesic coordinates $\left( s^1,\cdots,s^d \right)$, with
\begin{equation*}
  \left( s^1 \right)^2+\cdots+\left( s^d \right)^2=r^2=d^2_M \left( x,y \right),
\end{equation*}
and put $y$ into polar coordinates
\begin{equation*}
  y = \exp_xr\theta,\quad \theta\in T_xM,\,\,\|\theta\|_x=1.
\end{equation*}
Recall from Lemma \ref{lem:taylor_parallel_transport} that the $j$-th coordinate component ($j=1,\cdots,d$) of $P_{y,x}v$ is
\begin{equation*}
  \begin{aligned}
    \left( P_{y,x}v \right)^j&=v^j-\frac{r^2}{6}\theta^k\theta^{\sigma}v^l \left( R_{l\sigma k}^{\phantom{l\sigma k}j}\left( x \right)+R_{k\sigma l}^{\phantom{l\sigma k}j}\left( x \right) \right)+O \left( r^3 \right)\\
    &=v^j-\frac{1}{6}s^ks^{\sigma}v^l \left( R_{l\sigma k}^{\phantom{l\sigma k}j}\left( x \right)+R_{k\sigma l}^{\phantom{l\sigma k}j}\left( x \right) \right)+O \left( r^3 \right).
  \end{aligned}
\end{equation*}
A further change of coordinates
\begin{equation*}
  \tilde{s}^1=\frac{s^1}{\sqrt{\epsilon}},\cdots,\tilde{s}^d=\frac{s^d}{\sqrt{\epsilon}};\quad \tilde{r}=\frac{r}{\sqrt{\epsilon}}
\end{equation*}
leads to
\begin{equation}
\label{eq:localized_geodesic_coordinates}
  \begin{aligned}
    \int_{B_{\sqrt{\epsilon}}\left( x \right)} &\Phi\left( \frac{d^2_M \left( x,y \right)}{\epsilon} \right)f \left( y,P_{y,x}v \right)\,d\mathrm{vol}_M \left( y \right)\\
    &=\int_{B_{\sqrt{\epsilon}}\left( 0 \right)} \Phi \left( \frac{r^2}{\epsilon} \right)\tilde{f} \left( s^1,\cdots,s^d, \left( P_{y,x}v \right)^1,\cdots,\left( P_{y,x}v \right)^d \right)d\mathrm{vol}_M \left( s^1,\cdots,s^d \right)\\
    &=\int_{B_1 \left( 0 \right)}\Phi \left( \tilde{r}^2 \right)\tilde{f} \left( \epsilon^{\frac{1}{2}}\tilde{s}^1,\cdots,\epsilon^{\frac{1}{2}}\tilde{s}^d,\left( P_{y,x}v \right)^1,\cdots,\left( P_{y,x}v \right)^d \right)d\mathrm{vol}_M \left( \tilde{s}^1,\cdots,\tilde{s}^d \right)
  \end{aligned}
\end{equation}
where $\tilde{f}$ denotes for $f$ in these geodesic normal coordinates,
\begin{equation*}
  \begin{aligned}
    \left( P_{y,x}v \right)^j = v^j-\frac{\epsilon}{6}\tilde{s}^k\tilde{s}^{\sigma}v^l \left( R_{l\sigma k}^{\phantom{l\sigma k}j}\left( x \right)+R_{k\sigma l}^{\phantom{l\sigma k}j}\left( x \right) \right)+O \left( \epsilon^{\frac{3}{2}}\tilde{r}^3 \right)
  \end{aligned}
\end{equation*}
and
\begin{equation*}
  \begin{aligned}
    &d\mathrm{vol}_M \left( s^1,\cdots,s^d \right)=\left[1-\frac{1}{6}R_{kl}\left( x \right)s^ks^l+O \left( r^3 \right)\right]ds^1\cdots ds^d,\\
    &d\mathrm{vol}_M \left( \tilde{s}^1,\cdots,\tilde{s}^d \right)=\left[1-\frac{\epsilon}{6}R_{kl}\left( x \right)\tilde{s}^k\tilde{s}^l+O \left( \epsilon^{\frac{3}{2}}\tilde{r}^3 \right)\right]\cdot\epsilon^{\frac{d}{2}}d\tilde{s}^1\cdots d\tilde{s}^d.
  \end{aligned}
\end{equation*}
Taylor expanding $\tilde{f}$ at $\left(0,v\right)\in T_xM$ in these coordinates, we have
\begin{equation*}
  \begin{aligned}
    &\tilde{f}\left( \epsilon^{\frac{1}{2}}\tilde{s}^1,\cdots,\epsilon^{\frac{1}{2}}\tilde{s}^d,\left( P_{y,x}v \right)^1,\cdots,\left( P_{y,x}v \right)^d \right)\\
    &=\tilde{f}\left( 0,v \right)+\epsilon^{\frac{1}{2}}\tilde{s}^j \frac{\partial\tilde{f}}{\partial s^j}\left( 0,v \right)+\left[ \left( P_{y,x}v \right)^j-v^j \right]\frac{\partial\tilde{f}}{\partial v^j}\left( 0,v \right)+\frac{\epsilon}{2}\tilde{s}^k\tilde{s}^l \frac{\partial^2\tilde{f}}{\partial s^k\partial s^l}\left( 0,v \right)\\
    &\quad+\frac{1}{2}\left[ \left( P_{y,x}v \right)^k-v^k \right]\left[ \left( P_{y,x}v \right)^l-v^l \right]\frac{\partial^2\tilde{f}}{\partial v^k\partial v^l}\left( 0,v \right)+\frac{\epsilon^{\frac{1}{2}}}{2}\left[ \left( P_{y,x}v \right)^j-v^j \right]\tilde{s}^m \frac{\partial^2\tilde{f}}{\partial v^js^m}\left( 0,v \right)+O \left( \epsilon^{\frac{3}{2}} \right).
  \end{aligned}
\end{equation*}
The rest of the proof follows from simply substituting this Taylor expansion into \eqref{eq:localized_geodesic_coordinates} and integrate term-by-term. For simplicity of notations, let us write
\begin{equation*}
  m_0:=\int_{B_1 \left( 0 \right)}\Phi \left( \tilde{r}^2 \right)d\tilde{s}^1\cdots d\tilde{s}^d.
\end{equation*}
By the symmetry of the domain of integration,
\begin{equation*}
  \begin{aligned}
    &\int_{B_1 \left( 0 \right)}\Phi \left( \tilde{r}^2 \right)\tilde{s}^jd\tilde{s}^1\cdots d\tilde{s}^d=0,\quad j=1,\cdots,d,\\
    &\int_{B_1 \left( 0 \right)}\Phi \left( \tilde{r}^2 \right)\tilde{s}^k\tilde{s}^ld\tilde{s}^1\cdots d\tilde{s}^d=0,\quad k\neq l,\,\,k,l=1,\cdots,d,
  \end{aligned}
\end{equation*}
and
\begin{equation*}
  m_2:=\int_{B_1 \left( 0 \right)}\Phi \left( \tilde{r}^2 \right)\left(\tilde{s}^j\right)^2d\tilde{s}^1\cdots d\tilde{s}^d,\quad j=1,\cdots,d
\end{equation*}
are constants independent of super-indices $1\leq j\leq d$. Following a direct computation,
\begin{align*}
    \int_{B_1 \left( 0 \right)}&\Phi \left( \tilde{r}^2 \right)\tilde{f} \left( 0,v \right)\left[1-\frac{\epsilon}{6}R_{kl}\left( x \right)\tilde{s}^k\tilde{s}^l+O \left( \epsilon^{\frac{3}{2}}\tilde{r}^3 \right)\right]\cdot\epsilon^{\frac{d}{2}}d\tilde{s}^1\cdots d\tilde{s}^d\\
    &=\epsilon^{\frac{d}{2}}f \left( x,v \right)\left[ m_0-\epsilon\frac{m_2}{6}\mathrm{Scal}\left( x \right)+O \left( \epsilon^2 \right) \right],\\
    \int_{B_1 \left( 0 \right)}&\Phi \left( \tilde{r}^2 \right)\epsilon^{\frac{1}{2}}\tilde{s}^j \frac{\partial\tilde{f}}{\partial s^j}\left( 0,v \right)\left[1-\frac{\epsilon}{6}R_{kl}\left( x \right)\tilde{s}^k\tilde{s}^l+O \left( \epsilon^{\frac{3}{2}}\tilde{r}^3 \right)\right]\cdot\epsilon^{\frac{d}{2}} d\tilde{s}^1\cdots d\tilde{s}^d\\
    &=\epsilon^{\frac{d}{2}}\epsilon^{\frac{1}{2}}\frac{\partial\tilde{f}}{\partial s^j}\left( 0,v \right)\left[\int_{B_1 \left( 0 \right)}\Phi \left( \tilde{r} \right)\tilde{s}^jd\tilde{s}^1\cdots d\tilde{s}^d+O \left( \epsilon^{\frac{3}{2}} \right)\right]=\epsilon^{\frac{d}{2}}\cdot O \left( \epsilon^2 \right),\\
    \int_{B_1 \left( 0 \right)}&\Phi \left( \tilde{r}^2 \right)\left[ \left( P_{y,x}v \right)^j-v^j \right]\frac{\partial\tilde{f}}{\partial v^j}\left( 0,v \right)\left[1-\frac{\epsilon}{6}R_{kl}\left( x \right)\tilde{s}^k\tilde{s}^l+O \left( \epsilon^{\frac{3}{2}}\tilde{r}^3 \right)\right]\cdot\epsilon^{\frac{d}{2}} d\tilde{s}^1\cdots d\tilde{s}^d\\
    &=\epsilon^{\frac{d}{2}}\left[ \epsilon\frac{m_2}{6}v^l\frac{\partial\tilde{f}}{\partial v^j}\left( 0,v \right)R_{lj}\left( x \right) +O \left( \epsilon^2 \right)\right],\\
    \int_{B_1 \left( 0 \right)}&\Phi \left( \tilde{r}^2 \right)\frac{\epsilon}{2}\tilde{s}^k\tilde{s}^l \frac{\partial^2\tilde{f}}{\partial s^k\partial s^l}\left( 0,v \right)  \left[1-\frac{\epsilon}{6}R_{kl}\left( x \right)\tilde{s}^k\tilde{s}^l+O \left( \epsilon^{\frac{3}{2}}\tilde{r}^3 \right)\right]\cdot\epsilon^{\frac{d}{2}} d\tilde{s}^1\cdots d\tilde{s}^d\\
    &=\epsilon^{\frac{d}{2}}\left[ \epsilon\frac{m_2}{2}\sum_{k=1}^d \frac{\partial^2\tilde{f}}{\partial \left( s^k \right)^2}\left( 0,v \right)+O \left( \epsilon^2\right) \right],\\
    \int_{B_1 \left( 0 \right)}&\Phi \left( \tilde{r}^2 \right)\cdot\frac{1}{2}\left[ \left( P_{y,x}v \right)^k-v^k \right]\left[ \left( P_{y,x}v \right)^l-v^l \right]\frac{\partial^2\tilde{f}}{\partial v^k\partial v^l}\left( 0,v \right)\left[1-\frac{\epsilon}{6}R_{kl}\left( x \right)\tilde{s}^k\tilde{s}^l+O \left( \epsilon^{\frac{3}{2}}\tilde{r}^3 \right)\right]\cdot\epsilon^{\frac{d}{2}} d\tilde{s}^1\cdots d\tilde{s}^d\\
    &=\epsilon^{\frac{d}{2}}\cdot O \left( \epsilon^2 \right),\\
    \int_{B_1 \left( 0 \right)}&\Phi \left( \tilde{r}^2 \right)\frac{\epsilon^{\frac{1}{2}}}{2}\left[ \left( P_{y,x}v \right)^j-v^j \right]\tilde{s}^m \frac{\partial^2\tilde{f}}{\partial v^js^m}\left( 0,v \right)\left[1-\frac{\epsilon}{6}R_{kl}\left( x \right)\tilde{s}^k\tilde{s}^l+O \left( \epsilon^{\frac{3}{2}}\tilde{r}^3 \right)\right]\cdot\epsilon^{\frac{d}{2}} d\tilde{s}^1\cdots d\tilde{s}^d\\
    &=\epsilon^{\frac{d}{2}}\cdot O \left( \epsilon^2 \right),
\end{align*}
where all $O \left( \epsilon^{\frac{3}{2}} \right)$ terms drop out as in the proof of Lemma~\ref{lem:basic_kernel_integral}, thanks for Remark~\ref{rem:asymptotic_expansion_parallel_transport}. Combining these computation, we have
\begin{align*}
    \int_M &\Phi_{\epsilon}\left( x,y \right)f \left( y,P_{y,x}v \right)\,d\mathrm{vol}_M \left( y \right)\\
    &=\int_{B_1 \left( 0 \right)}\Phi \left( \tilde{r} \right)\tilde{f} \left( \epsilon^{\frac{1}{2}}\tilde{s}^1,\cdots,\epsilon^{\frac{1}{2}}\tilde{s}^d,\left( P_{y,x}v \right)^1,\cdots,\left( P_{y,x}v \right)^d \right)d\mathrm{vol}_M \left( \tilde{s}^1,\cdots,\tilde{s}^d \right)\\
    &=\epsilon^{\frac{d}{2}}\Bigg\{f \left( x,v \right)\left[ m_0-\epsilon\frac{m_2}{6}\mathrm{Scal}\left( x \right) \right]\\
    &\phantom{aaaaaaaaaaaa}+\epsilon\frac{m_2}{6}v^l\frac{\partial\tilde{f}}{\partial v^j}\left( 0,v \right)R_{lj}\left( x \right)+\epsilon\frac{m_2}{2}\sum_{k=1}^d \frac{\partial^2\tilde{f}}{\partial \left( s^k \right)^2}\left( 0,v \right)+O \left( \epsilon^2\right)\Bigg\}\\
    &=\epsilon^{\frac{d}{2}}\Bigg\{f \left( x,v \right)\left[ m_0-\epsilon\frac{m_2}{6}\mathrm{Scal}\left( x \right) \right]\\
    &\phantom{aaaaaaaaaaaa}+\epsilon \frac{m_2}{2}\left[\sum_{k=1}^d \frac{\partial^2\tilde{f}}{\partial \left( s^k \right)^2}\left( 0,v \right)+\frac{1}{3}R_{lj}\left( x \right)v^l\frac{\partial\tilde{f}}{\partial v^j}\left( 0,v \right)\right]+O \left( \epsilon^2 \right) \Bigg\}\\
    &=\epsilon^{\frac{d}{2}}\left\{ m_0 f \left( x,v \right)-\epsilon \frac{m_2}{6}f \left( x,v \right)\mathrm{Scal}\left( x \right)+\epsilon \frac{m_2}{2}\Delta^Hf \left( x,v \right)+O \left( \epsilon^2 \right) \right\}\\
    &=\epsilon^{\frac{d}{2}}\left\{m_0 f \left( x,v \right)+\epsilon \frac{m_2}{2}\left[ \Delta^H f \left( x,v \right)-\frac{1}{3}\mathrm{Scal}\left( x \right)f \left( x,v \right) \right]+O \left( \epsilon^2 \right) \right\},
\end{align*}
where in the second to last equality we used the expression of $\Delta^H$ in geodesic normal coordinates from \eqref{eq:horizontal_laplace_beltrami_geodesic_normal}.
\end{proof}

The following Lemma~\ref{lem:kernel_parallel_transport_utm} is the unit tangent bundle version of Lemma~\ref{lem:kernel_parallel_transport}.
\begin{lemma}
\label{lem:kernel_parallel_transport_utm}
Following Lemma \ref{lem:basic_kernel_integral}, for any function $f\in C^{\infty}\left( UTM \right)$, as $\epsilon\rightarrow 0$,
\begin{equation}
  \label{eq:kernel_parallel_transport_utm}
  \begin{aligned}
    \int_M &\Phi_{\epsilon}\left( x,y \right)f \left( y,P_{y,x}v \right)d\mathrm{vol}_M \left( y \right)\\
    &=\epsilon^{\frac{d}{2}}\left\{m_0 f \left( x,v \right)+\epsilon \frac{m_2}{2}\left[ \Delta_S^H f \left( x,v \right)-\frac{1}{3}\mathrm{Scal}\left( x \right)f \left( x,v \right) \right]+O \left( \epsilon^2 \right) \right\},
  \end{aligned}
\end{equation}
where $\Delta_S^H$ is the horizontal Laplacian on $UTM$ as defined in \eqref{eq:horizontal_spherical_gradient_laplace_beltrami}.
\end{lemma}

\begin{proof}
  First of all, note that by the metric compatibility of the Levi-Civita connection, $P_{y,x}:T_xM\rightarrow T_yM$ is an isometry, which descends naturally to an isometry from $S_x$ to $S_y$. For any function $f\in C^{\infty}\left( UTM \right)$, let $\hat{f}$ denote its extension to the whole $TM$, as in~\eqref{eq:horizontal_spherical_gradient_laplace_beltrami}. By Lemma \ref{lem:kernel_parallel_transport},
\begin{equation*}
  \begin{aligned}
    \int_M &\Phi_{\epsilon}\left( x,y \right)f \left( y,P_{y,x}v \right)d\mathrm{vol}_M \left( y \right) = \int_M \Phi_{\epsilon}\left( x,y \right)\hat{f} \left( y,P_{y,x}v \right)d\mathrm{vol}_M \left( y \right)\\
    &=\epsilon^{\frac{d}{2}}\left\{m_0 \hat{f} \left( x,v \right)+\epsilon \frac{m_2}{2}\left[ \Delta^H \hat{f} \left( x,v \right)-\frac{1}{3}\mathrm{Scal}\left( x \right)\hat{f}\left( x,v \right) \right]+O \left( \epsilon^2 \right) \right\}\\
    &=\epsilon^{\frac{d}{2}}\left\{m_0 f\left( x,v \right)+\epsilon \frac{m_2}{2}\left[ \Delta_S^H f \left( x,v \right)-\frac{1}{3}\mathrm{Scal}\left( x \right)f\left( x,v \right) \right]+O \left( \epsilon^2 \right) \right\}.
  \end{aligned}
\end{equation*}
\end{proof}

Based on Lemma~\ref{lem:kernel_parallel_transport_utm}, we have the following Lemma~\ref{lem:asymp_sym_kernels_utm}, which carries most of the work for proving the first part of Theorem~\ref{thm:main_utm}.
\begin{lemma}
\label{lem:asymp_sym_kernels_utm}
Assume $M$ is a $d$-dimensional closed Riemannian manifold with $\mathrm{Inj}\left( M \right)>0$. For any function $g\in C^{\infty}\left( UTM \right)$ and sufficiently small $\epsilon\in \left(0,\mathrm{Inj}\left( M \right)^2  \right)$, $\delta= O \left( \epsilon \right)$,
\begin{equation}
  \label{eq:numerator_taylor_expansion_utm}
  \begin{aligned}
   &\int_{UTM}K_{\epsilon,\delta}\left( x,v; y,w \right)g \left( y,w \right)\,d\Theta \left( y,w \right)=\int_M\!\int_{S_y}K_{\epsilon,\delta}\left( x,v; y,w \right)g \left( y,w \right)\,d\sigma_y\left(w\right)d\mathrm{vol}_M\left(y\right)\\
    &=\epsilon^{\frac{d}{2}}\delta^{\frac{d-1}{2}}\Bigg\{m_0g \left( x,v \right)+\epsilon\frac{m_{21}}{2}\left( \Delta_S^Hg \left( x,v \right)-\frac{1}{3}\mathrm{Scal}\left( x \right)g \left( x,v \right) \right)\\
    &\phantom{aaaaaaaaaaaaaaaaa}+\delta\frac{m_{22}}{2}\left(\Delta_S^Vg \left( x,v \right)-\frac{\left( d-1 \right)\left( d-2 \right)}{3}g \left( x,v \right)  \right)+O \left( \epsilon^2+\delta^2 \right) \Bigg\}
  \end{aligned}
\end{equation}
where $m_0,m_{21},m_{22}$ are positive constants, $d\sigma_y$ is the volume element on $S_y$, and $\Delta_S^H$, $\Delta_S^V$ are the horizontal and vertical spherical Laplace-Beltrami operators on $UTM$ as defined in~\eqref{eq:vertical_spherical_gradient_laplace_beltrami}\eqref{eq:horizontal_spherical_gradient_laplace_beltrami}.
\end{lemma}

\begin{proof}
By definition,
\begin{equation}
\label{eq:preparation}
  \begin{aligned}
    \int_M\!\int_{S_y}K_{\epsilon,\delta}\left( x,v; y,w \right)&\,g \left( y,w \right)\,d\sigma_y\left(w\right)d\mathrm{vol}_M\left(y\right)\\
    &=\int_M\!\int_{S_y}K\left( \frac{d^2_M \left( x,y \right)}{\epsilon}, \frac{d^2_{S_y}\left(P_{y,x}v,w  \right)}{\delta} \right)g \left( y,w \right)\,d\sigma_y\left(w\right)d\mathrm{vol}_M\left(y\right).
  \end{aligned}
\end{equation}
Since $\delta=O \left( \epsilon \right)$, $\delta\rightarrow0$ as $\epsilon\rightarrow 0$. Applying Lemma \ref{lem:basic_kernel_integral},
\begin{equation*}
  \begin{aligned}
    &\int_{S_y}K\left( \frac{d^2_M \left( x,y \right)}{\epsilon}, \frac{d^2_{S_y}\left(P_{y,x}v,w  \right)}{\delta} \right)g \left( y,w \right)\,d\sigma_y\left(w\right)\\
    &=\delta^{\frac{d-1}{2}}\Bigg\{ M_0 \left( \frac{d^2_M \left( x,y \right)}{\epsilon} \right)g \left( y,P_{y,x}v \right)\\
    &\qquad+\frac{\delta}{2}M_2 \left( \frac{d^2_M \left( x,y \right)}{\epsilon} \right)\left[ \Delta_S^Vg \left( y,P_{y,x}v \right)-\frac{1}{3}\mathrm{Scal}^{S_y} \left( P_{y,x}v \right)g \left( y,P_{y,x}v \right) \right]+O \left( \delta^2 \right) \Bigg\},
  \end{aligned}
\end{equation*}
where $\mathrm{Scal}^{S_y} \left( \cdot \right)$ is the scalar curvature of $S_y$, and $M_0 \left( r \right)$, $M_2 \left( r \right)$ are functions of a single variable determined by the following integrals over the unit ball in $\mathbb{R}^d$
\begin{equation*}
  \begin{aligned}
    M_0 \left( r^2 \right)&=\int_{B_1^{d-1} \left( 0 \right)}K \left( r^2,\rho^2 \right)d\theta^1\cdots d\theta^{d-1},\\
    M_2\left( r^2 \right)&=\int_{B_1^{d-1} \left( 0 \right)}\left(\theta^1\right)^2K \left( r^2,\rho^2 \right)d\theta^1\cdots d\theta^{d-1}
  \end{aligned}
\end{equation*}
with
\begin{equation*}
  \rho^2 = \left( \theta^1 \right)^2+\cdots\left(\theta^{d-1}\right)^2.
\end{equation*}
Moreover, for any $y\in M$, since $\left(T_yM,\|\cdot\|_y\right)$ is isometric to the standard $d$-Euclidean space (this can be seen by simply fixing an orthonormal basis in $T_yM$), $S_y$ equipped with the induced metric from $\|\cdot\|_y$ is also isometric to the standard unit $\left( d-1 \right)$-ball in $\mathbb{R}^d$. As a result,
\begin{equation*}
  \mathrm{Scal}^{S_y} \left( w \right)=\left( d-1 \right)\left( d-2 \right)\quad\textrm{for all $y\in M$, $w\in S_y$.}
\end{equation*}
Thus
\begin{equation}
\label{eq:inner_integral}
  \begin{aligned}
    &\int_{S_y}K\left( \frac{d^2_M \left( x,y \right)}{\epsilon}, \frac{d^2_{S_y}\left(P_{y,x}v,w  \right)}{\delta} \right)g \left( y,w \right)\,d\sigma_y\left(w\right)\\
    &=\delta^{\frac{d-1}{2}}\Bigg\{ M_0 \left( \frac{d^2_M \left( x,y \right)}{\epsilon} \right)g \left( y,P_{y,x}v \right)\\
    &\qquad+\frac{\delta}{2}M_2 \left( \frac{d^2_M \left( x,y \right)}{\epsilon} \right)\left[ \Delta_S^Vg \left( y,P_{y,x}v \right)-\frac{\left( d-1 \right)\left( d-2 \right)}{3} g \left( y,P_{y,x}v \right) \right]+O \left( \delta^2 \right) \Bigg\}.
  \end{aligned}
\end{equation}
It remains to apply Lemma \ref{lem:kernel_parallel_transport_utm} multiple times to \eqref{eq:inner_integral} to obtain
\begin{equation*}
  \begin{aligned}
    &\int_M\!\int_{S_y}K_{\epsilon,\delta}\left( x,v; y,w \right)\,g \left( y,w \right)\,d\sigma_y\left(w\right)d\mathrm{vol}_M\left(y\right)\\
    &=\int_M\left[\int_{S_y}K\left( \frac{d^2_M \left( x,y \right)}{\epsilon}, \frac{d^2_{S_y}\left(P_{y,x}v,w  \right)}{\delta} \right)g \left( y,w \right)d\sigma_y\left(w\right)\right]d\mathrm{vol}_M\left(y\right)\\
    &=\delta^{\frac{d-1}{2}}\Bigg\{\int_MM_0 \left( \frac{d^2_M \left( x,y \right)}{\epsilon} \right)g \left( y,P_{y,x}v \right)d\mathrm{vol}_M\left(y\right)\\
    &+\int_M \frac{\delta}{2}M_2 \left( \frac{d^2_M \left( x,y \right)}{\epsilon} \right)\left[ \Delta_S^Vg \left( y,P_{y,x}v \right)-\frac{1}{3}\left( d-1 \right)\left( d-2 \right) g \left( y,P_{y,x}v \right) \right] d\mathrm{vol}_M\left(y\right)+O \left( \delta^2 \right)\Bigg\}\\
    &=\delta^{\frac{d-1}{2}}\Bigg\{\epsilon^{\frac{d}{2}}\left[ m_0g \left( x,v \right)+\epsilon\frac{m_{21}}{2}\left( \Delta_S^Hg \left( x,v \right)-\frac{1}{3}\mathrm{Scal}\left( x \right)g \left( x,v \right) \right)+O \left( \epsilon^2 \right) \right]\\
    &\qquad+\epsilon^{\frac{d}{2}}\cdot\delta\frac{m_{22}}{2}\left(\Delta_S^Vg \left( x,v \right)-\frac{\left( d-1 \right)\left( d-2 \right)}{3}g \left( x,v \right)  \right)+\epsilon^{\frac{d}{2}}\cdot O \left( \epsilon^2+\delta^2 \right)\Bigg\}\\
    &=\epsilon^{\frac{d}{2}}\delta^{\frac{d-1}{2}}\Bigg\{m_0g \left( x,v \right)+\epsilon\frac{m_{21}}{2}\left( \Delta_S^Hg \left( x,v \right)-\frac{1}{3}\mathrm{Scal}\left( x \right)g \left( x,v \right) \right)\\
    &\phantom{aaaaaaaaaaaaaaaaa}+\delta\frac{m_{22}}{2}\left(\Delta_S^Vg \left( x,v \right)-\frac{\left( d-1 \right)\left( d-2 \right)}{3}g \left( x,v \right)  \right)+O \left( \epsilon^2+\delta^2 \right) \Bigg\},
  \end{aligned}
\end{equation*}
where $m_0$, $m_1$, $m_2$ are constants determined by the following integrals of $M_0 \left( r^2 \right)$ or $M_2 \left( r^2 \right)$ over the unit ball in $\mathbb{R}^d$
\begin{equation*}
  \begin{aligned}
   &m_0=\int_{B_1^d \left( 0 \right)}M_0 \left( r^2 \right)ds^1\cdots ds^d,\\
   &m_{21}=\int_{B_1^d \left( 0 \right)}M_0 \left( r^2 \right)\left( s^1 \right)^2ds^1\cdots ds^d,\\
   &m_{22}=\int_{B_1^d \left( 0 \right)}M_2 \left( r^2 \right)ds^1\cdots ds^d
  \end{aligned}
\end{equation*}
where
\begin{equation*}
  r^2=\left( s^1 \right)^2+\cdots+\left( s^d \right)^2.
\end{equation*}
\end{proof}

\begin{proof}[Proof of Theorem~\ref{thm:main_utm}]
As $\delta=O \left( \epsilon \right)$, direct application of Lemma~\ref{lem:asymp_sym_kernels_utm} gives
\begin{equation}
  \label{eq:empirical_density_function_asymp_utm}
  \begin{aligned}
  p_{\epsilon,\delta}\left( x,v \right)&=\int_{UTM} K_{\epsilon,\delta}\left( x,v; y,w \right)p \left( y,w \right)d\Theta \left( y,w \right)\\
  &=\epsilon^{\frac{d}{2}}\delta^{\frac{d-1}{2}}\Bigg\{m_0p \left( x,v \right)+\epsilon\frac{m_{21}}{2}\left( \Delta_S^Hp \left( x,v \right)-\frac{1}{3}\mathrm{Scal}\left( x \right)p \left( x,v \right) \right)\\
  &\quad+\delta\frac{m_{22}}{2}\left(\Delta_S^Vp \left( x,v \right)-\frac{\left( d-1 \right)\left( d-2 \right)}{3}p\left( x,v \right)  \right)+O \left( \epsilon^2+\delta^2 \right) \Bigg\}.
\end{aligned}
\end{equation}
In order to expand the denominator of \eqref{eq:hdo_utm}, note that
\begin{equation*}
  \begin{aligned}
  p_{\epsilon,\delta}^{-\alpha}\left( x,v \right)&=\epsilon^{-\frac{\alpha d}{2}}\delta^{-\frac{\alpha\left(d-1\right)}{2}}\Bigg\{m_0p \left( x,v \right)+\epsilon\frac{m_{21}}{2}\left( \Delta_S^Hp\left( x,v \right)-\frac{1}{3}\mathrm{Scal}\left( x \right)p\left( x,v \right) \right)\\
  &\quad+\delta\frac{m_{22}}{2}\left(\Delta_S^Vp\left( x,v \right)-\frac{\left( d-1 \right)\left( d-2 \right)}{3}p\left( x,v \right)  \right)+O \left(\epsilon^2+\delta^2 \right) \Bigg\}^{-\alpha}\\
  &=\epsilon^{-\frac{\alpha d}{2}}\delta^{-\frac{\alpha\left(d-1\right)}{2}}m_0^{-\alpha}p^{-\alpha}\left( x,v \right)\Bigg\{1-\epsilon\frac{\alpha m_{21}}{2m_0}\left( \frac{\Delta_S^Hp\left( x,v \right)}{p \left( x,v \right)}-\frac{1}{3}\mathrm{Scal}\left( x \right)\right)\\
  &\quad-\delta\frac{\alpha m_{22}}{2m_0}\left(\frac{\Delta_S^Vp\left( x,v \right)}{p \left( x,v \right)}-\frac{\left( d-1 \right)\left( d-2 \right)}{3} \right)+O \left(\epsilon^2+\delta^2 \right) \Bigg\},
  \end{aligned}
\end{equation*}
and hence by Lemma \ref{lem:asymp_sym_kernels_utm}
\begin{align*}
    &\int_{UTM}K_{\epsilon,\delta}^{\alpha} \left( x,v; y,w \right)p \left( y,w \right)d\Theta \left( y,w \right)\\
    &=\int_M\!\int_{S_y}K_{\epsilon,\delta}\left( x,v; y,w \right)p_{\epsilon,\delta}^{-\alpha}\left( x,v \right)p_{\epsilon,\delta}^{-\alpha}\left( y,w \right) p\left( y,w \right)d\sigma_y\left(w\right)d\mathrm{vol}_M\left(y\right)\\
    &=\epsilon^{-\frac{\alpha d}{2}}\delta^{-\frac{\alpha\left(d-1\right)}{2}}m_0^{-\alpha}p_{\epsilon,\delta}^{-\alpha}\left( x,v \right)\times\\
    &\int_M\!\int_{S_y}K_{\epsilon,\delta}\left( x,y; v,w \right)p^{1-\alpha}\left( y,w \right)\Bigg[1-\epsilon\frac{\alpha m_{21}}{2m_0}\left( \frac{\Delta_S^Hp\left( y,w \right)}{p \left( y,w \right)}-\frac{1}{3}\mathrm{Scal}\left( y \right)\right)\\
    &\phantom{aaaaaaaaaaaaaaaaaaaaaaaaaaaaa}-\delta\frac{\alpha m_{22}}{2m_0}\left(\frac{\Delta_S^Vp\left( y,w \right)}{p \left( y,w \right)}-\frac{\left( d-1 \right)\left( d-2 \right)}{3} \right)+O \left( \epsilon^2+\delta^2 \right) \Bigg]\\
    &=\epsilon^{\frac{\left(1-\alpha\right)d}{2}}\delta^{\frac{\left(1-\alpha\right)\left(d-1\right)}{2}}m_0^{-\alpha}p_{\epsilon,\delta}^{-\alpha}\left( x,v \right)\times\\
    &\Bigg\{m_0p^{1-\alpha}\left( x,v \right)\Bigg[1-\epsilon\frac{\alpha m_{21}}{2m_0}\left( \frac{\Delta_S^Hp\left( x,v \right)}{p \left(x,v \right)}-\frac{1}{3}\mathrm{Scal}\left( x \right)\right)-\delta\frac{\alpha m_{22}}{2m_0}\left(\frac{\Delta_S^Vp\left( x,v \right)}{p \left( x,v \right)}-\frac{\left( d-1 \right)\left( d-2 \right)}{3} \right)\Bigg]\\
    &+\epsilon \frac{m_{21}}{2}\left[ \Delta_S^Hp^{1-\alpha}\left( x,v \right)-\frac{1}{3}\mathrm{Scal}\left( x \right)p^{1-\alpha}\left( x,v \right) \right]+\delta \frac{m_{22}}{2}\left[ \Delta_S^Vp^{1-\alpha}\left( x,v \right)-\frac{\left( d-1 \right)\left( d-2 \right)}{3}p^{1-\alpha}\left( x,v \right) \right]\\
    &+O \left(\epsilon^2+\delta^2 \right) \Bigg\}\\
    &=\epsilon^{\frac{\left(1-\alpha\right)d}{2}}\delta^{\frac{\left(1-\alpha\right)\left(d-1\right)}{2}}m_0^{1-\alpha}p_{\epsilon,\delta}^{-\alpha}\left( x,v \right)p^{1-\alpha}\left( x,v \right)\times\\
    &\Bigg\{1-\epsilon\frac{\alpha m_{21}}{2m_0}\left( \frac{\Delta_S^Hp\left( x,v \right)}{p \left(x,v \right)}-\frac{1}{3}\mathrm{Scal}\left( x \right)\right)-\delta\frac{\alpha m_{22}}{2m_0}\left(\frac{\Delta_S^Vp\left( x,v \right)}{p \left( x,v \right)}-\frac{\left( d-1 \right)\left( d-2 \right)}{3} \right)\\
    &\quad+\epsilon \frac{m_{21}}{2m_0}\left( \frac{\Delta_S^Hp^{1-\alpha}\left( x,v \right)}{p^{1-\alpha}\left( x,v \right)}-\frac{1}{3}\mathrm{Scal}\left( x \right) \right)+\delta \frac{m_{22}}{2m_0}\left( \frac{\Delta_S^Vp^{1-\alpha}\left( x,v \right)}{p^{1-\alpha}\left( x,v \right)}-\frac{\left( d-1 \right)\left( d-2 \right)}{3} \right)\\
    &+O \left(\epsilon^2+\delta^2\right) \Bigg\}\\
    &=\epsilon^{\frac{\left(1-\alpha\right)d}{2}}\delta^{\frac{\left(1-\alpha\right)\left(d-1\right)}{2}}m_0^{1-\alpha}p_{\epsilon,\delta}^{-\alpha}\left( x,v \right)p^{1-\alpha}\left( x,v \right)\times\\
    &\Bigg\{1+\epsilon \frac{m_{21}}{2m_0}\left[\frac{\Delta_S^Hp^{1-\alpha}\left( x,v \right)}{p^{1-\alpha}\left( x,v \right)}-\alpha \frac{\Delta_S^Hp\left( x,v \right)}{p \left(x,v \right)}-\left( 1-\alpha \right)\frac{1}{3}\mathrm{Scal}\left( x \right)  \right]\\
    &\quad+\delta \frac{m_{22}}{2m_0}\left[\frac{\Delta_S^Vp^{1-\alpha}\left( x,v \right)}{p^{1-\alpha}\left( x,v \right)}-\alpha\frac{\Delta_S^Vp\left( x,v \right)}{p \left( x,v \right)}-\left( 1-\alpha \right)\frac{\left( d-1 \right)\left( d-2 \right)}{3}  \right]+O \left(\epsilon^2+\delta^2 \right) \Bigg\}.
\end{align*}
A similar computation expands the numerator of \eqref{eq:hdo_utm}:
\begin{align*}
  &\int_{UTM} K_{\epsilon,\delta}^{\alpha} \left( x,y; v,w \right)f \left( y,w \right)p \left( y,w \right)\,\mathrm{d}\sigma_y\left(w\right)d\mathrm{vol}_M\left(y\right)\\
  &=\int_M\!\int_{S_y}K_{\epsilon,\delta}\left( x,y; v,w \right)p_{\epsilon,\delta}^{-\alpha}\left( x,v \right)p_{\epsilon,\delta}^{-\alpha}\left( y,w \right)f \left( y,w \right) p\left( y,w \right)d\sigma_y\left(w\right)d\mathrm{vol}_M\left(y\right)\\
  &=\epsilon^{-\frac{\alpha d}{2}}\delta^{-\frac{\alpha\left(d-1\right)}{2}}m_0^{-\alpha}p_{\epsilon,\delta}^{-\alpha}\left( x,v \right)\times\\
  &\int_M\!\int_{S_y}K_{\epsilon,\delta}\left( x,y; v,w \right)f \left( y,w \right)p^{1-\alpha}\left( y,w \right)\Bigg[1-\epsilon\frac{\alpha m_{21}}{2m_0}\left( \frac{\Delta_S^Hp\left( y,w \right)}{p \left( y,w \right)}-\frac{1}{3}\mathrm{Scal}\left( y \right)\right)\\
    &\phantom{aaaaaaaaaaaaaaaaaaaaaaaaaaaaa}-\delta\frac{\alpha m_{22}}{2m_0}\left(\frac{\Delta_S^Vp\left( y,w \right)}{p \left( y,w \right)}-\frac{\left( d-1 \right)\left( d-2 \right)}{3} \right)+O \left(\epsilon^2+\delta^2 \right) \Bigg]\\
    &=\epsilon^{\frac{\left(1-\alpha\right)d}{2}}\delta^{\frac{\left(1-\alpha\right)\left(d-1\right)}{2}}m_0^{-\alpha}p_{\epsilon,\delta}^{-\alpha}\left( x,v \right)\times\\
    &\Bigg\{m_0 f \left( x,v \right) p^{1-\alpha}\left( x,v \right)\Bigg[1-\epsilon\frac{\alpha m_{21}}{2m_0}\left( \frac{\Delta_S^Hp\left( x,v \right)}{p \left(x,v \right)}-\frac{1}{3}\mathrm{Scal}\left( x \right)\right)-\delta\frac{\alpha m_{22}}{2m_0}\left(\frac{\Delta_S^Vp\left( x,v \right)}{p \left( x,v \right)}-\frac{\left( d-1 \right)\left( d-2 \right)}{3} \right)\Bigg]\\
    &+\epsilon \frac{m_1}{2}\left[ \Delta_S^H\left[fp^{1-\alpha}\right]\left( x,v \right)-\frac{1}{3}\mathrm{Scal}\left( x \right)\left[fp^{1-\alpha}\right]\left( x,v \right) \right]\\
    &+\delta \frac{m_2}{2}\left[ \Delta_S^V\left[fp^{1-\alpha}\right]\left( x,v \right)-\frac{\left( d-1 \right)\left( d-2 \right)}{3}\left[fp^{1-\alpha}\right]\left( x,v \right) \right]+O \left(\epsilon^2+\delta^2 \right) \Bigg\}\\
    &=\epsilon^{\frac{\left(1-\alpha\right)d}{2}}\delta^{\frac{\left(1-\alpha\right)\left(d-1\right)}{2}}m_0^{1-\alpha}p_{\epsilon,\delta}^{-\alpha}\left( x,v \right)p^{1-\alpha}\left( x,v \right)\times\\
    &\Bigg\{f \left( x,v \right)-\epsilon\frac{\alpha m_{21}}{2m_0}f \left( x,v \right)\left( \frac{\Delta_S^Hp\left( x,v \right)}{p \left(x,v \right)}-\frac{1}{3}\mathrm{Scal}\left( x \right)\right)-\delta\frac{\alpha m_{22}}{2m_0}f \left( x,v \right)\left(\frac{\Delta_S^Vp\left( x,v \right)}{p \left( x,v \right)}-\frac{\left( d-1 \right)\left( d-2 \right)}{3} \right)\\
    &\quad+\epsilon \frac{m_{21}}{2m_0}f \left( x,v \right)\left( \frac{\Delta_S^H\left[fp^{1-\alpha}\right]\left( x,v \right)}{\left[fp^{1-\alpha}\right]\left( x,v \right)}-\frac{1}{3}\mathrm{Scal}\left( x \right) \right)+\delta \frac{m_{22}}{2m_0}f \left( x,v \right)\left( \frac{\Delta_S^V\left[fp^{1-\alpha}\right]\left( x,v \right)}{\left[fp^{1-\alpha}\right]\left( x,v \right)}-\frac{\left( d-1 \right)\left( d-2 \right)}{3} \right)\\
    &+O \left( \epsilon^2+\delta^2 \right) \Bigg\}\\
    &=\epsilon^{\frac{\left(1-\alpha\right) d}{2}}\delta^{\frac{\left(1-\alpha\right)\left(d-1\right)}{2}}m_0^{1-\alpha}p_{\epsilon,\delta}^{-\alpha}\left( x,v \right)p^{1-\alpha}\left( x,v \right)\times\\
    &\quad\Bigg\{f \left( x,v \right)+\epsilon \frac{m_{21}}{2 m_0}f \left( x,v \right)\Bigg[ \frac{\Delta_S^H\left[fp^{1-\alpha}\right]\left( x,v \right)}{\left[fp^{1-\alpha}\right]\left( x,v \right)}-\alpha \frac{\Delta_S^H p \left( x,v \right)}{p \left( x,v \right)}-\frac{1}{3}\left( 1-\alpha \right)\mathrm{Scal}\left( x \right) \Bigg]\\
    &\qquad+\frac{m_{22}}{2m_0}f \left( x,v \right)\left[\frac{\Delta_S^V\left[fp^{1-\alpha}\right]\left( x,v \right)}{\left[fp^{1-\alpha}\right]\left( x,v \right)}-\alpha \frac{\Delta_S^V p \left( x,v \right)}{p \left( x,v \right)}-\left( 1-\alpha \right)\frac{\left( d-1 \right)\left( d-2 \right)}{3}\right]\\
    &\qquad+O \left(\epsilon^2+\delta^2 \right) \Bigg\}.
\end{align*}
Combining expansions for denominator and numerator, we conclude that
\begin{align*}
  &H_{\epsilon,\delta}^{\alpha}f \left( x,v \right)=\frac{\int_{UTM}K_{\epsilon,\delta}^{\alpha}\left( x,v;y,w \right)f \left( y,w \right)p \left( y,w \right)\,d\Theta \left( y,w \right)}{\int_{UTM}K_{\epsilon,\delta}^{\alpha}\left( x,v;y,w \right)p \left( y,w \right)\,d\Theta \left( y,w \right)}\\
  &=\Bigg\{f \left( x,v \right)+\epsilon \frac{m_{21}}{2 m_0}f \left( x,v \right)\left[ \frac{\Delta_S^H\left[fp^{1-\alpha}\right]\left( x,v \right)}{\left[fp^{1-\alpha}\right]\left( x,v \right)}-\alpha \frac{\Delta_S^H p \left( x,v \right)}{p \left( x,v \right)}-\frac{1}{3}\left( 1-\alpha \right)\mathrm{Scal}\left( x \right) \right]\\
  &\qquad+\delta\frac{m_{22}}{2m_0}f \left( x,v \right)\left[\frac{\Delta_S^V\left[fp^{1-\alpha}\right]\left( x,v \right)}{\left[fp^{1-\alpha}\right]\left( x,v \right)}-\alpha \frac{\Delta_S^V p \left( x,v \right)}{p \left( x,v \right)}-\left( 1-\alpha \right)\frac{\left( d-1 \right)\left( d-2 \right)}{3}\right]+O \left( \epsilon^2+\delta^2 \right) \Bigg\}\\
  &\cdot\Bigg\{ 1+\epsilon \frac{m_{21}}{2 m_0}\left[ \frac{\Delta_S^Hp^{1-\alpha}\left( x,v \right)}{p^{1-\alpha}\left( x,v \right)}-\alpha \frac{\Delta_S^H p \left( x,v \right)}{p \left( x,v \right)}-\frac{1}{3}\left( 1-\alpha \right)\mathrm{Scal}\left( x \right) \right]\\
  &\qquad+\delta \frac{m_{22}}{2m_0}\left[\frac{\Delta_S^V p^{1-\alpha}\left( x,v \right)}{p^{1-\alpha}\left( x,v \right)}-\alpha \frac{\Delta_S^V p \left( x,v \right)}{p \left( x,v \right)}-\left( 1-\alpha \right)\frac{\left( d-1 \right)\left( d-2 \right)}{3}  \right]+O \left( \epsilon^2+\delta^2 \right) \Bigg\}^{-1}\\
  &=f \left( x,v \right)\Bigg\{1+\epsilon \frac{m_{21}}{2 m_0}\left[ \frac{\Delta_S^H\left[fp^{1-\alpha}\right]\left( x,v \right)}{\left[fp^{1-\alpha}\right]\left( x,v \right)}-\alpha \frac{\Delta_S^H p \left( x,v \right)}{p \left( x,v \right)}-\frac{1}{3}\left( 1-\alpha \right)\mathrm{Scal}\left( x \right) \right]\\
  &\qquad+\delta \frac{m_{22}}{2m_0}\left[\frac{\Delta_S^V\left[fp^{1-\alpha}\right]\left( x,v \right)}{\left[fp^{1-\alpha}\right]\left( x,v \right)}-\alpha \frac{\Delta_S^V p \left( x,v \right)}{p \left( x,v \right)}-\left( 1-\alpha \right)\frac{\left( d-1 \right)\left( d-2 \right)}{3}\right]\\
  &\qquad-\epsilon \frac{m_{21}}{2 m_0}\left[ \frac{\Delta_S^Hp^{1-\alpha}\left( x,v \right)}{p^{1-\alpha}\left( x,v \right)}-\alpha \frac{\Delta_S^H p \left( x,v \right)}{p \left( x,v \right)}-\frac{1}{3}\left( 1-\alpha \right)\mathrm{Scal}\left( x \right) \right]\\
  &\qquad-\delta \frac{m_{22}}{2m_0}\left[\frac{\Delta_S^V p^{1-\alpha}\left( x,v \right)}{p^{1-\alpha}\left( x,v \right)}-\alpha \frac{\Delta_S^V p \left( x,v \right)}{p \left( x,v \right)}-\left( 1-\alpha \right)\frac{\left( d-1 \right)\left( d-2 \right)}{3}  \right]\Bigg\}+O \left( \epsilon^2+\delta^2 \right)\\
  &=f \left( x,v \right)\Bigg\{1+\epsilon \frac{m_{21}}{2 m_0}\Bigg[ \frac{\Delta_S^H\left[fp^{1-\alpha}\right]\left( x,v \right)}{\left[fp^{1-\alpha}\right]\left( x,v \right)}-\frac{\Delta_S^Hp^{1-\alpha}\left( x,v \right)}{p^{1-\alpha}\left( x,v \right)}\Bigg]\\
  &\qquad+\delta \frac{m_{22}}{2m_0}\left[\frac{\Delta_S^V\left[fp^{1-\alpha}\right]\left( x,v \right)}{\left[fp^{1-\alpha}\right]\left( x,v \right)}-\frac{\Delta_S^V p^{1-\alpha}\left( x,v \right)}{p^{1-\alpha}\left( x,v \right)}\right]\Bigg\}+O \left(\epsilon^2+\delta^2 \right)\\
  &= f \left( x,v \right)+\epsilon \frac{m_{21}}{2m_0}\left[\frac{\Delta_S^H\left[fp^{1-\alpha}\right]\left( x,v \right)}{p^{1-\alpha}\left( x,v \right)}-f \left( x,v \right) \frac{\Delta_S^Hp^{1-\alpha}\left( x,v \right)}{p^{1-\alpha}\left( x,v \right)}  \right]\\
  &\phantom{f \left( x,v \right)}\quad\,+\delta \frac{m_{22}}{2m_0}\left[\frac{\Delta_S^V\left[fp^{1-\alpha}\right]\left( x,v \right)}{p^{1-\alpha}\left( x,v \right)}-f \left( x,v \right) \frac{\Delta_S^Vp^{1-\alpha}\left( x,v \right)}{p^{1-\alpha}\left( x,v \right)}  \right]+O\left(\epsilon^2+\delta^2  \right).
\end{align*}
\end{proof}
A proof of Theorem~\ref{thm:main_tm} can be composed with a similar direct computation. The only prerequisite is to establish a tangent bundle version of Lemma~\ref{lem:asymp_sym_kernels_utm}, using Lemma~\ref{lem:kernel_parallel_transport} instead of Lemma~\ref{lem:kernel_parallel_transport_utm}. Similarly, a proof of Proposition~\ref{prop:gamma_infty} can be derived from the following proof of Proposition~\ref{prop:gamma_infty_utm}.

\begin{proof}[Proof of Proposition~\ref{prop:gamma_infty_utm}]
  To establish~\eqref{eq:main_utm_ii_base}, first note that
\begin{align*}
    \lim_{\gamma\rightarrow\infty}p_{\epsilon,\gamma\epsilon}\left( x,v \right)&=\lim_{\gamma\rightarrow\infty}\int_M\!\int_{S_y} K_{\epsilon,\gamma\epsilon}\left( x,v; y,w \right)p \left( y,w \right)d\sigma_y\left(w\right)d\mathrm{vol}_M\left(y\right)\\
    &=\lim_{\gamma\rightarrow\infty}\int_M\!\int_{S_y}K \left( \frac{d^2_M \left( x,y \right)}{\epsilon}, \frac{d^2_{S_y}\left(P_{y,x}v,w  \right)}{\gamma\epsilon} \right) p \left( y,w \right)d\sigma_y\left(w\right)d\mathrm{vol}_M\left(y\right)\\
    &=\int_M\!\int_{S_y}\lim_{\gamma\rightarrow\infty}K \left( \frac{d^2_M \left( x,y \right)}{\epsilon}, \frac{d^2_{S_y}\left(P_{y,x}v,w  \right)}{\gamma\epsilon} \right) p \left( y,w \right)d\sigma_y\left(w\right)d\mathrm{vol}_M\left(y\right)\\
    &=\int_M\!\int_{S_y}K \left( \frac{d^2_M \left( x,y \right)}{\epsilon},0 \right) p \left( y,w \right)d\sigma_y\left(w\right)d\mathrm{vol}_M\left(y\right)\\
    &=\int_M\!K \left( \frac{d^2_M \left( x,y \right)}{\epsilon},0 \right) \Bigg[\int_{S_y}p \left( y,w \right)d\sigma_y\left(w\right)\Bigg]d\mathrm{vol}_M\left(y\right),
\end{align*}
since $P_{y,x}v$ does not depend on $\gamma$. Recall from~\eqref{eq:density_on_base_utm} that
\begin{equation*}
  \overline{p}\left( x \right)=\int_{S_x}p \left( x,v \right)\,dV_x \left( v \right),
\end{equation*}
and define
\begin{equation*}
  \overline{K}_{\epsilon}\left( x,y \right)=\overline{K}\left( \frac{d^2_M \left( x,y \right)}{\epsilon} \right)=K \left( \frac{d^2_M \left( x,y \right)}{\epsilon},0 \right),
\end{equation*}
\begin{equation*}
  \overline{K}_{\epsilon}^{\alpha}\left( x,y \right)=\frac{\overline{K}_{\epsilon}\left( x,y \right)}{\overline{p}_{\epsilon}^{\alpha}\left( x \right)\overline{p}_{\epsilon}^{\alpha}\left( y \right)}.
\end{equation*}
Then
\begin{equation}
\label{eq:homogenized_H_alpha}
  \begin{aligned}
    \lim_{\gamma\rightarrow\infty}H_{\epsilon,\gamma\epsilon}^{\alpha}f \left( x,v \right)&=\frac{\displaystyle\int_M\overline{K}_{\epsilon}^{\alpha}\left( x,y \right)\Bigg[\int_{S_y}f \left( y,w \right)p \left( y,w \right)d\sigma_y\left(w\right)\Bigg]d\mathrm{vol}_M\left(y\right)}{\displaystyle\int_M\overline{K}_{\epsilon}^{\alpha}\left( x,y \right)\Bigg[\int_{S_y}p \left( y,w \right)d\sigma_y\left(w\right)\Bigg]d\mathrm{vol}_M\left(y\right)}\\
    &=\frac{\displaystyle\int_M\overline{K}_{\epsilon}^{\alpha}\left( x,y \right)\Bigg[\int_{S_y}f \left( y,w \right)\frac{p \left( y,w \right)}{\overline{p}\left( y \right)}d\sigma_y\left(w\right)\Bigg]\overline{p}\left( y \right)d\mathrm{vol}_M\left(y\right)}{\displaystyle\int_M\overline{K}_{\epsilon}^{\alpha}\left( x,y \right)\overline{p}\left( y \right)d\mathrm{vol}_M\left(y\right)}\\
    &=\frac{\displaystyle\int_M\overline{K}_{\epsilon}^{\alpha}\left( x,y \right)\overline{f}\left( y \right)\overline{p}\left( y \right)d\mathrm{vol}_M\left(y\right)}{\displaystyle\int_M\overline{K}_{\epsilon}^{\alpha}\left( x,y \right)\overline{p}\left( y \right)d\mathrm{vol}_M\left(y\right)}.
  \end{aligned}
\end{equation}
By \cite[Theorem 2]{CoifmanLafon2006}, as $\epsilon\rightarrow0$
\begin{equation}
\label{eq:homogenized_limit}
\begin{aligned}
  \lim_{\gamma\rightarrow\infty}&H^{\alpha}_{\epsilon,\gamma\epsilon}f \left( x,v \right)\\
  &=\frac{\displaystyle\int_M\overline{K}_{\epsilon}^{\alpha}\left( x,y \right)\overline{f}\left( y \right)\overline{p}\left( y \right)d\mathrm{vol}_M\left(y\right)}{\displaystyle\int_M\overline{K}_{\epsilon}^{\alpha}\left( x,y \right)\overline{p}\left( y \right)d\mathrm{vol}_M\left(y\right)}\\
  &=\overline{f}\left( x \right)+\epsilon \frac{m'_2}{2m'_0} \left[\frac{\Delta_M\left[\overline{f}\overline{p}^{1-\alpha}\right]\left( x \right)}{\overline{p}^{1-\alpha}\left( x \right)}-\overline{f}\left( x \right)\frac{\Delta_M\overline{p}^{1-\alpha}\left( x \right)}{\overline{p}^{1-\alpha}\left( x \right)} \right]+O \left( \epsilon^2 \right),
\end{aligned}
\end{equation}
where
\begin{equation*}
  \begin{aligned}
    m'_0&=\int_{B_1^d \left( 0 \right)}\overline{K}\left( r^2 \right)ds^1\cdots ds^d,\\
    m'_2&=\int_{B_1^d \left( 0 \right)}\left(s^1\right)^2\overline{K}\left( r^2 \right)ds^1\cdots ds^d,
  \end{aligned}
\end{equation*}
and we again dropped the higher order error term from $O \left( \epsilon^{\frac{3}{2}} \right)$ to $O \left( \epsilon^2 \right)$, as argued in~\cite[\S 2]{Singer2006ConvergenceRate}.
\end{proof}

\subsection{Proofs of Theorem~\ref{thm:utm_finite_sampling_noiseless} and Theorem~\ref{thm:utm_finite_sampling_noise}}
\label{sec:proof-finite-sampling-theorem}

To prove the two finite sampling theorems, we'll follow the path paved by \cite{BelkinNiyogi2005,HeinAudibertVonLuxburg2007,Singer2006ConvergenceRate,SingerWu2012VDM}.

\subsubsection{Sampling without Noise}
\label{sec:proof-theor-noiseless}

The following lemma builds the bridge between the geodesic distance on the manifold and the Euclidean distance in the ambient space.

\begin{lemma}
\label{lem:distance_expansion}
Let $\iota:M\hookrightarrow \mathbb{R}^D$ be an isometric embedding of the smooth $d$-dimensional closed Riemannian manifold $M$ into $\mathbb{R}^D$. For any $x,y\in M$ such that $d_M \left( x,y \right)<\mathrm{Inj}\left( M \right)$, we have
\begin{equation}
  \label{eq:taylor_geodesic_distance}
  d_M^2 \left( x,y \right)=\left\|\iota \left( x \right)-\iota \left( y \right)  \right\|^2+\frac{1}{12}d_M^4 \left( x,y \right)\left\|\Pi \left( \theta,\theta \right) \right\|^2+O \left( d_M^5 \left( x,y \right)\right),
\end{equation}
where $\theta\in T_xM$, $\left\|\theta\right\|_x=1$ comes from the geodesic polar coordinates of $y$ in a geodesic normal neighborhood of $x$:
\begin{equation*}
  y=\mathrm{exp}_xr\theta,\quad r=d_M \left( x,y \right).
\end{equation*}
\end{lemma}
\begin{proof}
  See \cite[Proposition 6]{SWW2007}.
\end{proof}

The reason we need Lemma~\ref{lem:distance_expansion} is due to the fact that the hypoelliptic diffusion operator in \eqref{eq:kernel_UTM} is constructed using geodesic distances on the manifolds, whereas in practice only the Euclidean distance in the ambient space is observed. In order to prove Theorem~\ref{thm:utm_finite_sampling_noiseless}, it is convenient to introduce the ``Euclidean distance version'' of the hypoelliptic diffusion operators. Note that in Definition~\ref{defn:utm_finite_sampling_noiseless} the hat ``$\hat{\phantom{a}}$'' is used for empirical quantities; for the remainder of this appendix, the tilde ``$\tilde{\phantom{a}}$'' will be used for quantities in hypoelliptic diffusion operators that replace the geodesic distance with the Euclidean distance. These quantities include\footnote{Note that in this subsection $\hat{K}_{\epsilon,\delta}$ is not much different from $\tilde{K}_{\epsilon,\delta}^{\alpha}$, since they are both constructed from Euclidean distance and exact parallel-transports. They will represent quite different quantities in next subsection, where $\hat{K}_{\epsilon,\delta}^{\alpha}$ is constructed from estimated parallel-transports.}
\begin{equation*}
  \begin{aligned}
    \tilde{K}_{\epsilon,\delta}\left( x,v;y,w \right)&=K \left( \frac{\left\|x-y\right\|^2}{\epsilon},\frac{\left\| P_{y,x}v-w \right\|^2_y}{\delta} \right),\\
    \tilde{p}_{\epsilon,\delta}\left( x,v \right)&=\int_{UTM}\tilde{K}_{\epsilon,\delta}\left( x,v;y,w \right)p \left( y,w \right)d\Theta \left( y,w \right),\\
    \tilde{K}_{\epsilon,\delta}^{\alpha}\left( x,v;y,w \right)&=\frac{\tilde{K}_{\epsilon,\delta}\left( x,v;y,w \right)}{\tilde{p}_{\epsilon,\delta}^{\alpha}\left( x,v \right)\tilde{p}_{\epsilon,\delta}^{\alpha}\left( y,w \right)},
  \end{aligned}
\end{equation*}
and eventually
\begin{equation*}
  \begin{aligned}
    \tilde{H}_{\epsilon,\delta}^{\alpha}f \left( x,v \right)=\frac{\displaystyle\int_{UTM}\tilde{K}_{\epsilon,\delta}^{\alpha}\left( x,v;y,w \right)f \left( y,w \right)p \left( y,w \right)\,d\Theta \left( y,w \right)}{\displaystyle\int_{UTM}\tilde{K}_{\epsilon,\delta}^{\alpha}\left( x,v;y,w \right)p \left( y,w \right)\,d\Theta \left( y,w \right)}.
  \end{aligned}
\end{equation*}

The next step is to establish an asymptotic expansion of type \eqref{eq:main_utm_i} for $\tilde{H}_{\epsilon,\delta}^{\alpha}$. We deduce the following Lemma~\ref{lem:euclidean_kernel_integral}, the ``Euclidean distance version'' of Lemma~\ref{lem:basic_kernel_integral}, from Lemma~\ref{lem:distance_expansion} and Lemma~\ref{lem:basic_kernel_integral} itself.
\begin{lemma}
\label{lem:euclidean_kernel_integral}
  Let $\Phi:\mathbb{R}\rightarrow\mathbb{R}$ be a smooth function compactly supported in $\left[ 0,1 \right]$. Assume $M$ is a $d$-dimensional closed Riemannian manifold isometrically embedded in $\mathbb{R}^D$, with injectivity radius $\mathrm{Inj}\left( M \right)>0$. For any $\epsilon>0$, define kernel function
  \begin{equation}
    \hat{\Phi}_{\epsilon}\left( x,y \right)=\Phi \left( \frac{\left\| x-y\right\|^2}{\epsilon} \right)
  \end{equation}
on $M\times M$, where $\left\| \cdot \right\|$ is the Euclidean distance on $\mathbb{R}^D$. If the parameter $\epsilon$ is sufficiently small such that $0\leq\epsilon\leq\sqrt{\mathrm{Inj}\left( M \right)}$, then the integral operator associated with kernel $\Phi_{\epsilon}$
\begin{equation}
  \left(\hat{\Phi}_{\epsilon}\,g\right)\left( x \right):=\int_M \Phi_{\epsilon}\left( x,y \right)g \left( y \right)d\mathrm{vol}_M \left( y \right)
\end{equation}
has the following asymptotic expansion as $\epsilon\rightarrow 0$
\begin{equation}
  \left(\hat{\Phi}_{\epsilon}\,g\right)\left( x \right) = \epsilon^{\frac{d}{2}}\left[ m_0 g \left( x \right)+\epsilon \frac{m_2}{2}\left( \Delta_M g \left( x \right)+E\left( x \right)g \left( x \right) \right)+O \left( \epsilon^2 \right) \right],
\end{equation}
with
\begin{equation*}
  E \left( x \right) = -\frac{1}{3}\mathrm{Scal}\left( x \right)+\frac{d \left( d+2 \right)}{12}A \left( x \right)
\end{equation*}
where $m_0,m_2$ are constants that depend on the moments of $\Phi$ and the dimension $d$ of the Riemannian manifold $M$, $\Delta_M$ is the Laplace-Beltrami operator on $M$, $\mathrm{Scal}\left( x \right)$ is the scalar curvature of $M$ at $x$, and $A \left( x \right)$ is a scalar function on $M$ that only depends on the intrinsic dimension $d$ and the second fundamental form of the isometric embedding $\iota: M\hookrightarrow\mathbb{R}^D$.
\end{lemma}

\begin{proof}
  From Lemma~\ref{lem:basic_kernel_integral},
\begin{equation*}
  \begin{aligned}
  \int_M & \Phi \left( \frac{d^2_M \left( x,y \right)}{\epsilon} \right)g \left( y \right)d\mathrm{vol}_M \left( y \right)\\
  & = \epsilon^{\frac{d}{2}}\left[ m_0 g \left( x \right)+\epsilon \frac{m_2}{2}\left( \Delta_M g \left( x \right)-\frac{1}{3}\mathrm{Scal}\left( x \right)g \left( x \right) \right)+O \left( \epsilon^2 \right) \right],
  \end{aligned}
\end{equation*}
thus we only need to expand
\begin{equation}
\label{eq:distance_approximation_error}
  \int_M \left[\Phi \left( \frac{\left\| x-y\right\|^2}{\epsilon} \right) - \Phi \left( \frac{d^2_M \left( x,y \right)}{\epsilon} \right)\right] g \left( y \right)d\mathrm{vol}_M \left( y \right).
\end{equation}
Put $y$ in geodesic polar coordinates in a geodesic normal neighborhood of $x\in M$,
\begin{equation*}
  y=\mathrm{exp}_xr\theta, \quad r=d_M \left( x,y \right), \theta\in T_xM, \left\|\theta\right\|_x=1,
\end{equation*}
and denote the geodesic normal coordinates around $x$ as $\left( s^1,\cdots,s^d \right)$. By Lemma~\ref{lem:distance_expansion},
\begin{equation*}
  \left\|x-y\right\|^2-d_M^2 \left( x,y \right)=-\frac{1}{12}d_M^4 \left( x,y \right)\left\|\Pi \left( \theta,\theta \right)\right\|^2+O \left( d_M^5 \left( x,y \right) \right)
\end{equation*}
thus
\begin{equation}
\label{eq:taylor_kernel_function}
  \begin{aligned}
    &\Phi \left( \frac{\left\| x-y\right\|^2}{\epsilon} \right) - \Phi \left( \frac{d^2_M \left( x,y \right)}{\epsilon} \right)\\
    &=\Phi' \left( \frac{d^2_M \left( x,y \right)}{\epsilon} \right)\cdot \left[ \frac{\left\|x-y\right\|^2}{\epsilon}-\frac{d_M^2 \left( x,y \right)}{\epsilon}  \right]+O \left(  \left[ \frac{\left\|x-y\right\|^2}{\epsilon}-\frac{d_M^2 \left( x,y \right)}{\epsilon}  \right]^2 \right)\\
    &=\Phi' \left(\frac{d_M^2 \left( x,y \right)}{\epsilon}  \right)\cdot \left( -\frac{1}{12\epsilon}d_M^4 \left( x,y \right)\left\|\Pi \left( \theta,\theta \right)\right\|^2 \right)+O \left( \frac{d_M^8 \left( x,y \right)}{\epsilon^2} \right).
  \end{aligned}
\end{equation}
Recall that $\Phi$ is supported on the unit interval, which implies that in \eqref{eq:distance_approximation_error} only those $y\in M$ satisfying $\left\|x-y \right\|\leq \sqrt{\epsilon}$ or $d_M \left( x,y \right)\leq \sqrt{\epsilon}$ are involved. According to Lemma~\ref{lem:distance_expansion}, for sufficiently small $\epsilon>0$, $\left\|x-y\right\|\leq \sqrt{\epsilon}$ implies $d_M \left( x,y \right)<2\sqrt{\epsilon}$, which means that the higher order error in \eqref{eq:taylor_kernel_function} is indeed
\begin{equation*}
  O \left( \frac{d_M^8 \left( x,y \right)}{\epsilon^2} \right)=O \left(\frac{\left(\sqrt{\epsilon}\right)^8}{\epsilon^2} \right)=O \left( \epsilon^2 \right).
\end{equation*}
Therefore,
\begin{equation}
\label{eq:error_involves_second_fundamental_form}
  \begin{aligned}
    \int_M & \left[\Phi \left( \frac{\left\| x-y\right\|^2}{\epsilon} \right) - \Phi \left( \frac{d^2_M \left( x,y \right)}{\epsilon} \right)\right] g \left( y \right)d\mathrm{vol}_M \left( y \right)\\
    &=-\frac{1}{12\epsilon}\int_M\Phi' \left(\frac{d_M^2 \left( x,y \right)}{\epsilon}  \right)d_M^4 \left( x,y \right)\left\|\Pi \left( \theta,\theta \right)\right\|^2g \left( y \right)d\mathrm{vol}_M \left( y \right)+\epsilon^{\frac{d}{2}}\cdot O \left( \epsilon^2 \right)\\
    &=-\frac{1}{12\epsilon}\int_M\Phi' \left(\frac{r^2}{\epsilon}  \right)r^4\left\|\Pi \left( \theta,\theta \right)\right\|^2g \left( y \right)d\mathrm{vol}_M \left( y \right)+\epsilon^{\frac{d}{2}}\cdot O \left( \epsilon^2 \right).
  \end{aligned}
\end{equation}
In geodesic normal coordinates $\left( s^1,\cdots,s^d \right)$,
\begin{equation}
\label{eq:put_geodesic_coordinates}
  \begin{aligned}
    \int_M& \Phi' \left(\frac{r^2}{\epsilon}  \right)r^4\left\|\Pi \left(\theta,\theta \right)\right\|^2g \left( y \right)d\mathrm{vol}_M \left( y \right)\\
    &=\int_{B_{\sqrt{\epsilon}}\left( 0 \right)}\Phi' \left(\frac{r^2}{\epsilon}  \right)r^4\left\|\Pi \left( \theta,\theta \right)\right\|^2\tilde{g}\left( s \right)\left[ 1-\frac{1}{6}R_{kl}\left( x \right)s^ks^l+O \left( r^3 \right) \right]ds^1\cdots ds^d.
  \end{aligned}
\end{equation}
As in Lemma~\ref{lem:basic_kernel_integral}, we Taylor expand $\tilde{g} \left( s \right)$ around $s=0$
\begin{equation*}
  \tilde{g}\left( s^1,\cdots,s^d \right)=\tilde{g}\left( 0 \right)+\frac{\partial\tilde{g}}{\partial s^j}\left( 0 \right)s^j+ \frac{1}{2}\frac{\partial^2\tilde{g}}{\partial s^k\partial s^l}\left( 0 \right)s^ks^l+O \left( r^3 \right),
\end{equation*}
and note that by symmetry
\begin{equation*}
  \int_{B_{\sqrt{\epsilon}}\left( 0 \right)}\Phi' \left(\frac{r^2}{\epsilon}  \right)r^4\left\|\Pi \left( \theta,\theta \right)\right\|^2 s^jds^1\cdots ds^d=0, \quad j=1,\cdots,d,
\end{equation*}
thus \eqref{eq:put_geodesic_coordinates} reduces to
\begin{equation*}
  \begin{aligned}
    \int_{B_{\sqrt{\epsilon}}\left( 0 \right)} & \Phi' \left(\frac{r^2}{\epsilon}  \right)r^4\left\|\Pi \left( \theta,\theta \right)\right\|^2 g \left( x \right)ds^1\cdots ds^d\\
    &+\int_{B_{\sqrt{\epsilon}}\left( 0 \right)}\Phi' \left(\frac{r^2}{\epsilon}  \right)r^4\left\|\Pi \left( \theta,\theta \right)\right\|^2 O \left( r^2 \right)ds^1\cdots ds^d
  \end{aligned}
\end{equation*}
where
\begin{equation}
\label{eq:error_leading_term}
  \begin{aligned}
    \int_{B_{\sqrt{\epsilon}}\left( 0 \right)} &\Phi' \left(\frac{r^2}{\epsilon}  \right)r^4\left\|\Pi \left( \theta,\theta \right)\right\|^2 g \left( x \right)ds^1\cdots ds^d\\
    &=g \left( x \right)\int_{B_1 \left( 0 \right)}\Phi' \left( \tilde{r}^2 \right)\cdot\epsilon^2\tilde{r}^4\left\|\Pi \left( \theta,\theta \right)\right\|^2\cdot \epsilon^{\frac{d}{2}}d\tilde{s}^1\cdots\tilde{s}^d\\
    &=\epsilon^{\frac{d}{2}}\cdot\epsilon^2g \left( x \right)\int_{S_1 \left( 0 \right)}\left\|\Pi \left( \theta,\theta \right)\right\|^2d\theta\int_0^1\Phi' \left(\tilde{r}^2\right)\tilde{r}^{4+\left(d-1\right)}d\tilde{r}\\
    &=\epsilon^{\frac{d}{2}}\cdot\epsilon^2g \left( x \right)\int_{S_1 \left( 0 \right)}\left\|\Pi \left( \theta,\theta \right)\right\|^2d\theta\int_0^1\Phi' \left(\tilde{r}^2\right)\tilde{r}^{3+d}d\tilde{r},
  \end{aligned}
\end{equation}
and
\begin{equation}
\label{eq:error_higher_order_term}
  \begin{aligned}
    \int_{B_{\sqrt{\epsilon}}\left( 0 \right)} &\Phi' \left(\frac{r^2}{\epsilon}  \right)r^4\left\|\Pi \left( \theta,\theta \right)\right\|^2 O \left( r^2 \right)ds^1\cdots ds^d\\
    &=\int_{B_1 \left( 0 \right)}\Phi' \left( \tilde{r}^2 \right)\cdot\epsilon^2\tilde{r}^4\left\|\Pi \left( \theta,\theta \right)\right\|^2\cdot\epsilon O \left( \tilde{r}^2 \right) \cdot \epsilon^{\frac{d}{2}}d\tilde{s}^1\cdots\tilde{s}^d\\
    &=\epsilon^{\frac{d}{2}}\cdot O \left( \epsilon^3 \right).
  \end{aligned}
\end{equation}
Recall from the proof of Lemma~\ref{lem:kernel_parallel_transport}, we adopted notation
\begin{equation*}
  m_2=\int_{B_1 \left( 0 \right)}\Phi \left( \tilde{r}^2 \right)\left(\tilde{s}^j\right)^2d\tilde{s}^1\cdots d\tilde{s}^d,\quad j=1,\cdots,d,
\end{equation*}
thus
\begin{equation*}
  m_2d=\int_{B_1 \left( 0 \right)}\Phi \left( \tilde{r}^2 \right)\tilde{r}^2d\tilde{s}^1\cdots d\tilde{s}^d=\omega_{d-1}\int_0^1\Phi \left( \tilde{r}^2 \right)\tilde{r}^{d+1}dr
\end{equation*}
where $\omega_{d-1}$ is the volume of the standard unit sphere of dimension $\left( d-1 \right)$. Denoting
\begin{equation*}
  A \left( x \right)= \frac{1}{\omega_{d-1}}\int_{S_1 \left( 0 \right)}\left\|\Pi \left( \theta,\theta \right)\right\|^2d\theta
\end{equation*}
as the average of the length of the second fundamental form over the standard unit sphere, we can write \eqref{eq:error_leading_term} as
\begin{equation}
\label{eq:last_piece}
  \begin{aligned}
    \epsilon^{\frac{d}{2}}&\cdot\epsilon^2g \left( x \right)\int_{S_1 \left( 0 \right)}\left\|\Pi \left( \theta,\theta \right)\right\|^2d\theta\int_0^1\Phi' \left(\tilde{r}^2\right)\tilde{r}^{3+d}d\tilde{r}\\
    &=\epsilon^{\frac{d}{2}}\cdot\epsilon^2g \left( x \right)A \left( x \right)\omega_{d-1}\cdot \left( -\frac{m_2}{2\omega_{d-1}}d \left( d+2 \right) \right) =-\epsilon^{\frac{d}{2}}\cdot \epsilon^2 \frac{m_2}{2}d \left( d+2 \right)g \left( x \right)A \left( x \right),
  \end{aligned}
\end{equation}
where we integrated by parts
\begin{equation*}
  \begin{aligned}
    \int_0^1\Phi' \left(\tilde{r}^2\right)\tilde{r}^{3+d}d\tilde{r}&\overset{\xi=\tilde{r}^2}{=\joinrel=\joinrel=}\frac{1}{2}\int_0^1\Phi' \left( \xi \right)\xi^{\frac{3+d}{2}}\cdot\xi^{-\frac{1}{2}}d\xi=\frac{1}{2}\int_0^1\Phi' \left( \xi \right)\xi^{1+\frac{d}{2}}d\xi\\
    &=\frac{1}{2}\left[ \Phi \left( \xi \right)\xi^{1+\frac{d}{2}}\Big|_{\xi=0}^{\xi=1}- \left( 1+\frac{d}{2} \right) \int_0^1\Phi \left( \xi \right)\xi^{\frac{d}{2}}d\xi \right]\\
    &=-\frac{d+2}{2}\int_0^1\Phi \left( \tilde{r}^2 \right)\tilde{r}^{d+1}d\tilde{r}=-\frac{m_2}{2\omega_{d-1}}d \left( d+2 \right).
  \end{aligned}
\end{equation*}
Combining \eqref{eq:last_piece}, \eqref{eq:error_leading_term}, \eqref{eq:error_higher_order_term} with \eqref{eq:put_geodesic_coordinates}, we conclude that
\begin{equation*}
  \begin{aligned}
  \int_M &\left[\Phi \left( \frac{\left\| x-y\right\|^2}{\epsilon} \right) - \Phi \left( \frac{d^2_M \left( x,y \right)}{\epsilon} \right)\right] g \left( y \right)d\mathrm{vol}_M \left( y \right)\\
  &=\epsilon^{\frac{d}{2}}\left[\frac{1}{12\epsilon}\cdot\epsilon^2 \frac{m_2}{2}d \left( d+2 \right)g \left( x \right)A \left( x \right)+O \left( \epsilon^2 \right)  \right]\\
  &=\epsilon^{\frac{d}{2}}\left[ \epsilon \frac{m_2}{24}d \left( d+2 \right)A \left( x \right)g \left( x \right)+O \left( \epsilon^2 \right) \right],
  \end{aligned}
\end{equation*}
which establishes
\begin{equation*}
  \begin{aligned}
    &\left(\hat{\Phi}_{\epsilon}\,g\right)\left( x \right)=\int_M \hat{\Phi}_{\epsilon}\left( x,y \right)g \left( y \right)d\mathrm{vol}_M \left( y \right)\\
    &=\int_M \Phi \left( \frac{d^2_M \left( x,y \right)}{\epsilon} \right)g \left( y \right)d\mathrm{vol}_M \left( y \right)+\int_M\left[\Phi \left( \frac{\left\| x-y\right\|^2}{\epsilon} \right) - \Phi \left( \frac{d^2_M \left( x,y \right)}{\epsilon} \right)\right] g \left( y \right)d\mathrm{vol}_M \left( y \right)\\
    &=\epsilon^{\frac{d}{2}}\left[ m_0 g \left( x \right)+\epsilon \frac{m_2}{2}\left( \Delta_M g \left( x \right)-\frac{1}{3}\mathrm{Scal}\left( x \right)g \left( x \right)+\frac{1}{12}d \left( d+2 \right)A \left( x \right)g \left( x \right) \right) +O \left( \epsilon^2 \right) \right]\\
    &=\epsilon^{\frac{d}{2}}\left[ m_0 g \left( x \right)+\epsilon \frac{m_2}{2}\left( \Delta_M g \left( x \right)+E \left( x \right)g \left( x \right)\right) +O \left( \epsilon^2 \right) \right]
  \end{aligned}
\end{equation*}
with
\begin{equation*}
  E \left( x \right) := -\frac{1}{3}\mathrm{Scal}\left( x \right)+\frac{1}{12}d \left( d+2 \right)A \left( x \right).
\end{equation*}
\end{proof}

\begin{remark}
\label{rem:lazy_remark}
  The only difference between the conclusions in Lemma~\ref{lem:euclidean_kernel_integral} and Lemma~\ref{lem:basic_kernel_integral} is that the scalar function $E \left( x \right)$ takes the place of the scalar curvature $\mathrm{Scal}\left( x \right)$; one can check, essentially by going through the proof of Theorem~\ref{thm:main_utm}, that this change does not affect the conclusion of Theorem~\ref{thm:main_utm}. Specifically, in that proof the same $\mathrm{Scal} \left( x \right)$ from the numerator and denominator cancel out with each other in the asymptotic expansion, and this cancellation still occurs if one replaces $\mathrm{Scal}\left( x \right)$ with $E \left( x \right)$. In fact, by applying Lemma~\ref{lem:euclidean_kernel_integral} repeatedly we have the following expansions for $f,g\in C^{\infty}\left( UTM \right)$
\begin{equation}
  \label{eq:kernel_parallel_transport_utm_hat}
  \begin{aligned}
    \int_M &\hat{\Phi}_{\epsilon}\left( x,y \right)f \left( y,P_{y,x}v \right)d\mathrm{vol}_M \left( y \right)\\
    &=\epsilon^{\frac{d}{2}}\left\{m_0 f \left( x,v \right)+\epsilon \frac{m_2}{2}\left[ \Delta_S^H f \left( x,v \right)+E_1 \left( x \right) f \left( x,v \right) \right]+O \left( \epsilon^2 \right) \right\},
  \end{aligned}
\end{equation}
and
\begin{equation}
  \label{eq:numerator_taylor_expansion_utm_hat}
  \begin{aligned}
   &\int_{UTM}\tilde{K}_{\epsilon,\delta}\left( x,v; y,w \right)g \left( y,w \right)\,d\Theta \left( y,w \right)\\
    &=\epsilon^{\frac{d}{2}}\delta^{\frac{d-1}{2}}\Bigg\{m_0g \left( x,v \right)+\epsilon\frac{m_{21}}{2}\left[ \Delta_S^Hg \left( x,v \right)+E_1\left( x \right)g \left( x,v \right) \right]\\
    &\phantom{aaaaaaaaaaaaaaaaa}+\delta\frac{m_{22}}{2}\left[\Delta_S^Vg \left( x,v \right)+E_2\cdot g \left( x,v \right)  \right]+O \left( \epsilon^2+\delta^2 \right) \Bigg\},
  \end{aligned}
\end{equation}
where
\begin{equation*}
  E_1 \left( \xi_i \right)=-\frac{1}{3}\mathrm{Scal}_M\left( \xi_i \right)+\frac{d \left( d+2 \right)}{12}\cdot \frac{1}{\omega_{d-1}}\int_{S_1 \left( 0 \right)}\left\|\Pi_M \left( \theta,\theta \right)\right\|^2d\theta
\end{equation*}
only depends on the scalar curvature $\mathrm{Scal}_M$ and the second fundamental form $\Pi_M$ of the base manifold $M $at $\xi$, and
\begin{equation*}
  E_2=-\frac{1}{3}\mathrm{Scal}_S+\frac{\left( d-1 \right)\left( d+1 \right)}{12}\cdot \frac{1}{\omega_{d-2}}\int_{S_1 \left( 0 \right)}\left\|\Pi_S \left( \theta,\theta \right)\right\|^2d\theta
\end{equation*}
is a constant because
\begin{equation*}
    \mathrm{Scal}_S \equiv \left( d-1 \right)\left( d-2 \right),\quad \left\|\Pi_S \left( \theta,\theta \right)\right\|^2 \equiv1\quad\textrm{for any unit tangent vector $\theta$}.
\end{equation*}
These expansions are essentially the equivalents of Lemma~\ref{lem:kernel_parallel_transport_utm} and Lemma~\ref{lem:asymp_sym_kernels_utm} for $\tilde{K}_{\epsilon,\delta}$. Using \eqref{eq:kernel_parallel_transport_utm_hat} and \eqref{eq:numerator_taylor_expansion_utm_hat}, and picking $\delta=O \left( \epsilon \right)$ as $\epsilon\rightarrow0$, a version of Theorem~\ref{thm:main_utm} holds true when $K_{\epsilon,\delta}^{\alpha}$ is replaced with $\tilde{K}_{\epsilon,\delta}^{\alpha}$, i.e., as $\epsilon\rightarrow0$ (and thus $\delta\rightarrow0$),
\begin{equation}
\label{eq:bias_error_term}
  \begin{aligned}
    \tilde{H}_{\epsilon,\delta}^{\alpha}f\left( x,v \right) &= f \left( x,v \right)+\epsilon \frac{m_{21}}{2m_0}\left[\frac{\Delta_S^H\left[fp^{1-\alpha}\right]\left( x,v \right)}{p^{1-\alpha}\left( x,v \right)}-f \left( x,v \right) \frac{\Delta_S^Hp^{1-\alpha}\left( x,v \right)}{p^{1-\alpha}\left( x,v \right)}  \right]\\
    &+\delta \frac{m_{22}}{2m_0}\left[\frac{\Delta_S^V\left[fp^{1-\alpha}\right]\left( x,v \right)}{p^{1-\alpha}\left( x,v \right)}-f \left( x,v \right) \frac{\Delta_S^Vp^{1-\alpha}\left( x,v \right)}{p^{1-\alpha}\left( x,v \right)}  \right]+O\left(\epsilon^2+\delta^2 \right).\\
\end{aligned}
\end{equation}
As we shall see below, this observation is the key to establish estimates for the bias error in the proof of Theorem~\ref{thm:utm_finite_sampling_noiseless}.
\end{remark}

The last missing piece for the proof of Theorem~\ref{thm:utm_finite_sampling_noiseless} is a large deviation bound for our two-step sampling strategy. Recall from Assumption~\ref{assum:utm_finite_sampling_noiseless} that we first sample $N_B$ points $\xi_1,\cdots,\xi_{N_B}$ i.i.d. with respect to $\overline{p}$ on the base manifold $M$, and then sample $N_F$ points on each fibre $S_{\xi_j}$ i.i.d. with respect to $p \left( \cdot\mid\xi_j \right)$. The resulting $N_B\times N_F$ points on $UTM$
\begin{equation*}
  \begin{matrix}
    x_{1,1}, &x_{1,2}, &\cdots, &x_{1,N_F}\\
    x_{2,1}, &x_{2,2}, &\cdots, &x_{2,N_F}\\
    \vdots & \vdots &\cdots &\vdots\\
    x_{N_B,1}, &x_{N_B,2}, &\cdots, &x_{N_B,N_F}
  \end{matrix}
\end{equation*}
are generally not i.i.d. sampled from $UTM$. This forbids applying the \emph{Law of Large Numbers} directly to quantities that take the form of an average over the entire unit tangent bundle, such as
\begin{equation*}
  \frac{1}{N_BN_F}\sum_{j=1}^{N_B}\sum_{s=1}^{N_F}\hat{K}_{\epsilon,\delta}\left( x_{i,r},x_{j,s} \right)f \left( x_{j,s} \right).
\end{equation*}
However, due to the conditional i.i.d. fibrewise sampling, it makes sense to apply the law of large numbers to average quantities on a fixed fibre, e.g.,
\begin{equation*}
  \frac{1}{N_F}\sum_{s=1}^{N_F}\hat{K}_{\epsilon,\delta}\left( x_{i,r},x_{j,s} \right)f \left( x_{j,s} \right)\longrightarrow \mathbb{E}_Z \left[ \tilde{K}_{\epsilon,\delta}\left( x_{i,r},\left( \xi_j,Z \right) \right)f \left( \xi_j,Z \right) \right],
\end{equation*}
where $\mathbb{E}_Z$ stands for the expectation with respect to the ``fibre component'' of the coordinates of the points on $S_{\xi_j}$. Explicitly,
\begin{equation*}
  \mathbb{E}_Z \left[ \tilde{K}_{\epsilon,\delta}\left( x_{i,r},\left( \xi_j,\cdot \right) \right)f \left( \xi_j,\cdot \right) \right]=\int_{S_{\xi_j}}\tilde{K}_{\epsilon,\delta}\left( x_{i,r},\left( \xi_j,w \right) \right)f \left( \xi_j,w \right)p \left( w\mid \xi_j \right)d\sigma_{\xi_j} \left( w \right).
\end{equation*}
Next, note that $\xi_1,\cdots,\xi_{N_B}$ are i.i.d. sampled from the base manifold $M$, the partial expectations
\begin{equation*}
  \left\{ \mathbb{E}_Z \left[ \tilde{K}_{\epsilon,\delta}\left( x_{i,r},\left( \xi_j,Z \right) \right)f \left( \xi_j,Z \right) \right] \right\}_{j=1}^{N_B}
\end{equation*}
are i.i.d. random variables on $M$ with respect to $\overline{p}$. Thus
\begin{equation*}
  \frac{1}{N_B}\sum_{j=1}^{N_B}\mathbb{E}_Z \left[ \tilde{K}_{\epsilon,\delta}\left( x_{i,r},\left( \xi_j,Z \right) \right)f \left( \xi_j,Z \right) \right]\longrightarrow \mathbb{E}_Y\left[ \mathbb{E}_Z \left[\tilde{K}_{\epsilon,\delta}\left( x_{i,r},\left( Y,Z \right) \right)f \left(Y,Z \right) \right] \right].
\end{equation*}
Explicitly,
\begin{equation*}
  \begin{aligned}
    &\mathbb{E}_Y\left[ \mathbb{E}_Z \left[\tilde{K}_{\epsilon,\delta}\left( x_{i,r},\left( Y,Z \right) \right)f \left(Y,Z \right) \right] \right]\\
   =&\int_M\overline{p}\left( y \right)\int_{S_{\xi_j}}\tilde{K}_{\epsilon,\delta}\left( x_{i,r},\left( \xi_j,w \right) \right)f \left( y,w \right)p \left( w\mid y \right)d\sigma_{y} \left( w \right)d\mathrm{vol}_M \left( y \right)\\
   =&\int_M\!\!\int_{S_{\xi_j}}\tilde{K}_{\epsilon,\delta}\left( x_{i,r},\left( \xi_j,w \right) \right)f \left( y,w \right)\left[\overline{p}\left( y \right)p \left( w\mid y \right)\right]d\sigma_{y} \left( w \right)d\mathrm{vol}_M \left( y \right)\\
   =&\int_M\!\int_{S_{\xi_j}}\tilde{K}_{\epsilon,\delta}\left( x_{i,r},\left( \xi_j,w \right) \right)f \left( y,w \right)p \left( y,w \right)d\sigma_{y} \left( w \right)d\mathrm{vol}_M \left( y \right).
  \end{aligned}
\end{equation*}
This observation suggests the following iterated limit process
\begin{equation*}
  \begin{aligned}
    \lim_{N_B\rightarrow\infty}&\lim_{N_F\rightarrow\infty}\frac{1}{N_BN_F}\sum_{j=1}^{N_B}\sum_{s=1}^{N_F}\hat{K}_{\epsilon,\delta}\left( x_{i,r},x_{j,s} \right)f \left( x_{j,s} \right)\\
    &=\int_M\!\int_{S_{\xi_j}}\tilde{K}_{\epsilon,\delta}\left( x_{i,r},\left( \xi_j,w \right) \right)f \left( y,w \right)p \left( y,w \right)d\sigma_{y} \left( w \right)d\mathrm{vol}_M \left( y \right),
  \end{aligned}
\end{equation*}
and the two limits on the left hand side generally do not commute.

For this reason, it is natural for us to consider iterated partial expectations rather than total expectation on the entire $UTM$. For simplicity of notation, let us denote $\mathbb{E}_Y,\mathbb{E}_Z$ as $\mathbb{E}_1,\mathbb{E}_2$ respectively, see the following definition.

\begin{definition}
\label{defn:special_sampling}
Let $p$ be a probability density function on $UTM$, and
\begin{equation*}
  \overline{p}\left( x \right)=\int_{S_x}p \left( x,w \right)d\sigma_x \left( w \right),\quad p \left( v\mid x \right)=\frac{p \left( x,v \right)}{\overline{p}\left( x \right)}
\end{equation*}
as defined in \eqref{eq:density_on_base}\eqref{eq:conditional_density_on_fibre}. For any function $f\in C^{\infty}\left( M \right)$, define
\begin{equation*}
  \mathbb{E}_1 f:= \int_Mf \left( y \right)p \left( y \right)d\mathrm{vol}_M \left( y \right).
\end{equation*}
For any function $g\in C^{\infty}\left( S_{\xi} \right)$ for $\xi\in M$, define
\begin{equation*}
  \mathbb{E}_2^{\xi}g := \int_{S_{\xi}}g \left( \xi,w \right)p \left( w\mid \xi \right)d\sigma_{\xi}\left( w \right).
\end{equation*}
\end{definition}

\begin{definition}
\label{defn:procrustesian}
Let $p$ be a probability density function on $UTM$. We call a collection of $N_B\times N_F$ real-valued random functions
\begin{equation*}
  \left\{ X_{j,s}\mid 1\leq j\leq N_B,1\leq s\leq N_F \right\}
\end{equation*}
\emph{Procrustean with respect to $p$ on $UTM$}, if
\begin{enumerate}[(i)]
\item\label{item:37} For each $1\leq j\leq N_B$, the subcollection $\left\{ X_{j,s}\mid 1\leq s\leq N_F \right\}$ are i.i.d. on $S_{\xi_j}$ for some $\xi_j\in M$, with respect to the conditional probability density $p \left( \cdot\mid \xi_j \right)$;
\item\label{item:38} The points $\left\{ \xi_j\mid 1\leq j\leq N_B \right\}$ are i.i.d. on $M$ with respect to the projected probability density $\overline{p} \left( \cdot \right)$.
\end{enumerate}
Due to \eqref{item:37}, we denote for simplicity of notation
\begin{equation*}
  \mathbb{E}_2^{\xi_j}X_j:=\mathbb{E}_2^{\xi_j}X_{j,s},\quad \mathbb{E}_2^{\xi_j}X_j^2:=\mathbb{E}_2^{\xi_j}X_{j,s}^2,
\end{equation*}
and for the same purpose, due to \eqref{item:38},
\begin{equation*}
  \mathbb{E}_1\mathbb{E}_2X:=\mathbb{E}_1\mathbb{E}_2^{\xi_j}X_j,\quad \mathbb{E}_1 \left( \mathbb{E}_2X \right)^2:=\mathbb{E}_1\left(\mathbb{E}_2^{\xi_j}X_j\right).
\end{equation*}

\end{definition}

\begin{lemma}
\label{lem:large_deviation_bound}
Let $\left\{ X_{j,s}\mid 1\leq j\leq N_B,1\leq s\leq N_F \right\}$ be a collection of Procrustean random functions with respect to some density function $p$ on $UTM$. If
\begin{equation*}
  \left| X_{j,s} \right|\leq M_0,\,\,\left| \mathbb{E}_2^{\xi_j}X_j \right|\leq M_1,\,\,\left| \mathbb{E}_1\mathbb{E}_2X \right|\leq M_2\quad\textrm{a.s. for all }1\leq j\leq N_B,1\leq s\leq N_F,
\end{equation*}
then for any $t>0$ and $0<\theta<1$,
\begin{equation*}
  \begin{aligned}
    &\mathbb{P}\left\{ \frac{1}{N_BN_F}\sum_{j=1}^{N_B}\sum_{s=1}^{N_F}X_{j,s}-\mathbb{E}_1\mathbb{E}_2X>t \right\}\\
    &\leq \sum_{j=1}^{N_B}\exp \left\{-\frac{\displaystyle \frac{1}{2}\left(1-\theta\right)^2N_Ft^2}{\displaystyle \left[\mathbb{E}_2^{\xi_j}X_j^2-\left(\mathbb{E}_2^{\xi_j}X_j \right)^2\right]+ \frac{1}{3}\left(M_0+M_1\right)\left(1-\theta\right) t}  \right\}\\
    &+\exp \left\{ -\frac{\displaystyle \frac{1}{2}\theta^2N_Bt^2}{\displaystyle \left[\mathbb{E}_1 \left( \mathbb{E}_2X\right)^2-\left(\mathbb{E}_1\mathbb{E}_2X \right)^2\right]+\frac{1}{3}\left(M_1+M_2\right)\theta t} \right\}.
  \end{aligned}
\end{equation*}
\end{lemma}
\begin{proof}
  Note that
  \begin{equation*}
    \begin{aligned}
      &\mathbb{P}\left\{ \frac{1}{N_BN_F}\sum_{j=1}^{N_B}\sum_{s=1}^{N_F}X_{j,s}-\mathbb{E}_1\mathbb{E}_2X>t \right\}\\
      &=\mathbb{P}\left\{ \left(\frac{1}{N_BN_F}\sum_{j=1}^{N_B}\sum_{s=1}^{N_F}X_{j,s}-\frac{1}{N_B}\sum_{j=1}^{N_B}\mathbb{E}_2^{\xi_j}X_j\right)+\left(\frac{1}{N_B}\sum_{j=1}^{N_B}\mathbb{E}_2^{\xi_j}X_j-\mathbb{E}_1\mathbb{E}_2X\right)>t \right\}\\
      &\leq \mathbb{P}\left( \left\{\frac{1}{N_BN_F}\sum_{j=1}^{N_B}\sum_{s=1}^{N_F}X_{j,s}-\frac{1}{N_B}\sum_{j=1}^{N_B}\mathbb{E}_2^{\xi_j}X_j>\left(1-\theta\right) t \right\}\bigcup \left\{\frac{1}{N_B}\sum_{j=1}^{N_B}\mathbb{E}_2^{\xi_j}X_j-\mathbb{E}_1\mathbb{E}_2X> \theta t \right\} \right)\\
      &\leq \mathbb{P}\left\{\frac{1}{N_BN_F}\sum_{j=1}^{N_B}\sum_{s=1}^{N_F}X_{j,s}-\frac{1}{N_B}\sum_{j=1}^{N_B}\mathbb{E}_2^{\xi_j}X_j>\left(1-\theta\right) t \right\}+\mathbb{P}\left\{\frac{1}{N_B}\sum_{j=1}^{N_B}\mathbb{E}_2^{\xi_j}X_j-\mathbb{E}_1\mathbb{E}_2X> \theta t \right\}\\
      &=: \left( \mathrm{I} \right)+\left( \mathrm{II} \right),
    \end{aligned}
  \end{equation*}
where $\theta\in \left( 0,1 \right)$ will be fixed in specific applications. Since
\begin{equation*}
  \left| \mathbb{E}_2^{\xi_j}X_j-\mathbb{E}_1\mathbb{E}_2X \right|\leq M_1+M_2,
\end{equation*}
by Bernstein's Inequality~\cite[\S2.2]{Chung2006},
\begin{align*}
    \left( \mathrm{II} \right)&=\mathbb{P}\left\{ \sum_{j=1}^{N_B}\left( \mathbb{E}_2^{\xi_j}X_j -\mathbb{E}_1\mathbb{E}_2X\right)> \theta N_Bt \right\}\\
    &\leq \exp \left\{ -\frac{\displaystyle \frac{1}{2}\theta^2N_B^2t^2}{\displaystyle \sum_{j=1}^{N_B}\mathbb{E}_1\left[ \mathbb{E}_2^{\xi_j}X_j -\mathbb{E}_1\mathbb{E}_2X\right]^2+\frac{1}{3}\left(M_1+M_2\right) \theta N_Bt}\right\}\\
    &=\exp \left\{ -\frac{\displaystyle \frac{1}{2}\theta^2N_B^2t^2}{\displaystyle N_B\left[\mathbb{E}_1 \left(\mathbb{E}_2X\right)^2 -\left(\mathbb{E}_1\mathbb{E}_2X\right)^2\right]+\frac{1}{3}\left(M_1+M_2\right) \theta N_Bt}\right\}\\
    &=\exp \left\{ -\frac{\displaystyle \frac{1}{2}\theta^2N_Bt}{\displaystyle\left[\mathbb{E}_1 \left(\mathbb{E}_2X\right)^2 -\left(\mathbb{E}_1\mathbb{E}_2X\right)^2\right]+\frac{1}{3} \left(M_1+M_2\right) \theta t} \right\}.
\end{align*}
For $\left( \mathrm{I} \right)$, note that
\begin{equation*}
  \begin{aligned}
    \left( \mathrm{I} \right)&=\mathbb{P}\left\{\sum_{j=1}^{N_B}\left( \frac{1}{N_F}\sum_{s=1}^{N_F}X_{j,s}-\mathbb{E}_2^{\xi_j}X_j \right)>\left(1-\theta\right) N_Bt \right\}\leq \sum_{j=1}^{N_B}\mathbb{P}\left\{ \frac{1}{N_F}\sum_{s=1}^{N_F}X_{j,s}-\mathbb{E}_2^{\xi_j}X_j>\left(1-\theta\right) t \right\}\\
    &=\sum_{j=1}^{N_B}\mathbb{P}\left\{ \sum_{s=1}^{N_F}\left( X_{j,s}-\mathbb{E}_2^{\xi_j}X_j \right)>\left(1-\theta\right) N_F t \right\}
  \end{aligned}
\end{equation*}
and applying Bernstein's Inequality to each individual term in the summation yields
\begin{equation*}
  \begin{aligned}
    &\mathbb{P}\left\{ \sum_{s=1}^{N_F}\left( X_{j,s}-\mathbb{E}_2^{\xi_j}X_j \right)>\left(1-\theta\right) N_F t \right\}\leq \exp \left\{ -\frac{\displaystyle \frac{1}{2}\left(1-\theta\right)^2N_F^2t^2}{\displaystyle \sum_{s=1}^{N_F}\mathbb{E}_2^{\xi_j}\left[ X_{j,s}-\mathbb{E}_2^{\xi_j}X_j \right]^2+\frac{1}{3}\left(M_0+M_1\right)\left(1-\theta\right) N_Ft} \right\}\\
    &=\exp \left\{-\frac{\displaystyle \frac{1}{2}\left(1-\theta\right)^2N_F^2t^2}{\displaystyle N_F\left[\mathbb{E}_2^{\xi_j}X_j^2-\left(\mathbb{E}_2^{\xi_j}X_j \right)^2\right]+\frac{1}{3}\left(M_0+M_1\right)\left(1-\theta\right) N_Ft}\right\}\\
    &=\exp \left\{-\frac{\displaystyle \frac{1}{2}\left(1-\theta\right)^2N_F t^2}{\displaystyle \left[\mathbb{E}_2^{\xi_j}X_j^2-\left(\mathbb{E}_2^{\xi_j}X_j \right)^2\right]+\frac{1}{3}\left(M_0+M_1\right)\left(1-\theta\right) t}  \right\},
  \end{aligned}
\end{equation*}
and thus
\begin{equation*}
  \left( \mathrm{I} \right)\leq \sum_{j=1}^{N_B}\exp \left\{-\frac{\displaystyle \frac{1}{2}\left(1-\theta\right)^2N_F t^2}{\displaystyle \left[\mathbb{E}_2^{\xi_j}X_j^2-\left(\mathbb{E}_2^{\xi_j}X_j \right)^2\right]+\frac{1}{3}\left(M_0+M_1\right)\left(1-\theta\right) t}  \right\}.
\end{equation*}
\end{proof}
\begin{remark}
\label{rem:insight_large_deviation_bound}
  Intuitively, the second term in the bound stems from the sampling error on the base manifold, and is thus independent of $\delta$ and $N_F$; the first term in the bound comes from accumulating fibrewise sampling error across all $N_B$ fibres.
\end{remark}

\begin{proof}[Proof of Theorem{\rm ~\ref{thm:utm_finite_sampling_noiseless}}]
  We shall first establish the result for $\alpha=0$. In this case, $\hat{K}^0_{\epsilon,\delta}\left( \cdot,\cdot \right)=\hat{K}_{\epsilon,\delta}\left( \cdot,\cdot \right)$, and
\begin{align*}
    \hat{H}_{\epsilon,\delta}^0f \left( x_{i,r} \right)&=\frac{\displaystyle\sum_{j=1}^{N_B}\sum_{s=1}^{N_F}\hat{K}_{\epsilon,\delta}\left( x_{i,r},x_{j,s} \right)f \left( x_{j,s} \right)}{\displaystyle\sum_{j=1}^{N_B}\sum_{s=1}^{N_F}\hat{K}_{\epsilon,\delta}\left( x_{i,r},x_{j,s} \right)}\\
    &=\frac{\displaystyle \frac{1}{N_BN_F} \sum_{j=1}^{N_B}\sum_{s=1}^{N_F}K \left( \frac{\|\xi_i-\xi_j\|^2}{\epsilon}, \frac{\|P_{\xi_j,\xi_i}x_{i,r}-x_{j,s}\|^2}{\delta} \right)f \left( x_{j,s} \right)}{\displaystyle\frac{1}{N_BN_F}\sum_{j=1}^{N_B}\sum_{s=1}^{N_F}K \left( \frac{\|\xi_i-\xi_j\|^2}{\epsilon}, \frac{\|P_{\xi_j,\xi_i}x_{i,r}-x_{j,s}\|^2}{\delta} \right)}.
\end{align*}
Since $\left\{x_{j,s}\right\}_{s=1}^{N_F}$ are i.i.d. with respect to $p \left( \cdot\mid \xi_j \right)$, by the law of large numbers, for each fixed $j=1,\cdots,N_B$, as $N_F\rightarrow\infty$,
\begin{equation*}
  \begin{aligned}
    \lim_{N_F\rightarrow\infty}&\frac{1}{N_F}\sum_{s=1}^{N_F}K \left( \frac{\|\xi_i-\xi_j\|^2}{\epsilon}, \frac{\|P_{\xi_j,\xi_i}x_{i,r}-x_{j,s}\|^2}{\delta} \right)f \left( x_{j,s} \right)\\
    &= \int_{S_{\xi_j}}K \left( \frac{\|\xi_i-\xi_j\|^2}{\epsilon}, \frac{\|P_{\xi_j,\xi_i}x_{i,r}-w\|^2}{\delta} \right)f \left( \xi_j,w \right) p \left( w\mid\xi_j \right) d\sigma_{\xi_j} \left( w \right),
  \end{aligned}
\end{equation*}
Note that $\left\{\xi_j\right\}_{j=1}^{N_B}$ are i.i.d. with respect to $\overline{p}$, again by the law of large numbers, as $N_B\rightarrow\infty$,
\begin{equation*}
  \begin{aligned}
    &\lim_{N_B\rightarrow\infty}\frac{1}{N_B}\sum_{j=1}^{N_B}\lim_{N_F\rightarrow\infty}\frac{1}{N_F}\sum_{s=1}^{N_F} K \left( \frac{\|\xi_i-\xi_j\|^2}{\epsilon}, \frac{\|P_{\xi_j,\xi_i}x_{i,r}-x_{j,s}\|^2}{\delta} \right)f \left( x_{j,s} \right)\\
    &=\int_M\overline{p}\left( y \right)\cdot\frac{1}{\omega_{d-1}}\int_{S_{\xi_j}}\!\!\!K \left( \frac{\|\xi_i-y\|^2}{\epsilon}, \frac{\|P_{y,\xi_i}x_{i,r}-w\|^2}{\delta} \right)f \left( y,w \right)d\sigma_{\xi_j} \left( w \right)p \left( w\mid y \right)d\mathrm{vol}_M\left( y \right)\\
    &=\int_{UTM}K \left( \frac{\|\xi_i-y\|^2}{\epsilon}, \frac{\|P_{y,\xi_i}x_{i,r}-w\|^2}{\delta} \right)f \left( y,w \right)p \left( y,w \right) d\Theta \left( y,w \right),
  \end{aligned}
\end{equation*}
where we used $p \left( y,w \right)=\overline{p}\left( y \right)p \left( w\mid y \right)$. Setting $f\equiv 1$,
\begin{equation*}
  \begin{aligned}
    &\lim_{N_B\rightarrow\infty}\frac{1}{N_B}\sum_{j=1}^{N_B}\lim_{N_F\rightarrow\infty}\frac{1}{N_F}\sum_{s=1}^{N_F} K \left( \frac{\|\xi_i-\xi_j\|^2}{\epsilon}, \frac{\|P_{\xi_j,\xi_i}x_{i,r}-x_{j,s}\|^2}{\delta} \right)\\
    &=\int_{UTM}K \left( \frac{\|\xi_i-y\|^2}{\epsilon}, \frac{\|P_{y,\xi_i}x_{i,r}-w\|^2}{\delta} \right)p \left( y,w \right) d\Theta \left( y,w \right).
  \end{aligned}
\end{equation*}
Therefore,
\begin{align*}
    \lim_{N_B\rightarrow\infty}&\lim_{N_F\rightarrow\infty}\hat{H}_{\epsilon,\delta}^0f \left( x_{i,r} \right)\\
    &=\frac{\displaystyle\int_{UTM}K \left( \frac{\|\xi_i-y\|^2}{\epsilon}, \frac{\|P_{y,\xi_i}x_{i,r}-w\|^2}{\delta} \right)f \left( y,w \right)p \left( y,w \right)d\Theta \left( y,w \right)}{\displaystyle\int_{UTM}K \left( \frac{\|\xi_i-y\|^2}{\epsilon}, \frac{\|P_{y,\xi_i}x_{i,r}-w\|^2}{\delta} \right)p \left( y,w \right)d\Theta \left( y,w \right)}\\
    &=\tilde{H}_{\epsilon,\delta}^0f \left( x_{i,r} \right)\\
    &=f \left( x_{i,r} \right)+\epsilon \frac{m_{21}}{2m_0}\left[\frac{\Delta_S^H\left[fp\right]\left( x_{i,r} \right)}{p\left( x_{i,r} \right)}-f \left( x_{i,r} \right) \frac{\Delta_S^Hp\left( x_{i,r} \right)}{p\left( x_{i,r} \right)}  \right]\\
    &+\delta \frac{m_{22}}{2m_0}\left[\frac{\Delta_S^V\left[fp\right]\left( x_{i,r} \right)}{p\left( x_{i,r} \right)}-f \left( x_{i,r} \right) \frac{\Delta_S^Vp\left( x_{i,r} \right)}{p\left( x_{i,r} \right)}  \right]+O \left( \epsilon^2+\delta^2 \right),
\end{align*}
where in the last equality we used the assumption $\delta=O \left( \epsilon \right)$ as $\epsilon\rightarrow0$, as well as the observation in Remark~\ref{rem:lazy_remark} that Theorem~\ref{thm:main_utm} holds true when $K_{\epsilon,\delta}^{\alpha}$ is replaced with $\tilde{K}_{\epsilon,\delta}^{\alpha}$. The bias error is thus $O \left( \epsilon^2+\delta^2 \right)$.

It remains to estimate the variance error for the special case $\alpha=0$. Write for any fixed $x_{i,r}\in UTM$
\begin{equation*}
  \begin{aligned}
    F_{j,s} = \hat{K}_{\epsilon,\delta}\left( x_{i,r},x_{j,s} \right)f \left( x_{j,s} \right),\qquad G_{j,s} = \hat{K}_{\epsilon,\delta}\left( x_{i,r},x_{j,s} \right).
  \end{aligned}
\end{equation*}
Note that $F_{i,s}=0$, $G_{i,s}=0$ for all $s=1,\cdots,N_F$, by Definition~\ref{defn:utm_finite_sampling_noiseless}~\eqref{item:34}. Also, by the compactness of $UTM$ we have some trivial bounds uniform in $j,s$:
\begin{equation*}
  \left| F_{j,s} \right|\leq \left\|K\right\|_{\infty}\left\|f\right\|_{\infty},\quad \left| G_{j,s} \right|\leq \left\|K\right\|_{\infty}.
\end{equation*}
In these notations,
\begin{equation*}
  \lim_{N_B\rightarrow\infty}\lim_{N_F\rightarrow\infty}\hat{H}_{\epsilon,\delta}^0f \left( x_{i,r} \right)=\frac{\mathbb{E}_1\mathbb{E}_2 F}{\mathbb{E}_1\mathbb{E}_2 G},
\end{equation*}
and we would like to estimate
\begin{equation*}
  p \left( N_B,N_F,\beta \right):=\mathbb{P} \left\{ \frac{\sum_j\sum_sF_{j,s}}{\sum_j\sum_sG_{j,s}}-\frac{\mathbb{E}_1\mathbb{E}_2 F}{\mathbb{E}_1\mathbb{E}_2 G}>\beta \right\}
\end{equation*}
for sufficiently small $\beta>0$. An upper bound for
\begin{equation*}
  \mathbb{P} \left\{ \frac{\sum_j\sum_sF_{j,s}}{\sum_j\sum_sG_{j,s}}-\frac{\mathbb{E}_1\mathbb{E}_2 F}{\mathbb{E}_1\mathbb{E}_2 G}<-\beta \right\}
\end{equation*}
can be obtained in a similar manner.

Since $G_{j,s}$ are all positive,
\begin{equation*}
  \begin{aligned}
    &p \left( N_B,N_F,\beta \right) = \mathbb{P} \left\{ \frac{\left(\sum_j\sum_sF_{j,s}\right)\mathbb{E}_1\mathbb{E}_2 G-\left( \sum_j\sum_sG_{j,s} \right)\mathbb{E}_1\mathbb{E}_2 F}{\left(\sum_j\sum_sG_{j,s}\right)\mathbb{E}_1\mathbb{E}_2 G} > \beta \right\}\\
    &=\mathbb{P} \left\{ \left(\sum_j\sum_sF_{j,s}\right)\mathbb{E}_1\mathbb{E}_2 G-\left( \sum_j\sum_sG_{j,s} \right)\mathbb{E}_1\mathbb{E}_2 F>\beta\left(\sum_j\sum_sG_{j,s}\right)\mathbb{E}_1\mathbb{E}_2 G \right\}.
  \end{aligned}
\end{equation*}
Denote
\begin{equation*}
  Y_{j,s}:=F_{j,s}\mathbb{E}_1\mathbb{E}_2G-G_{j,s}\mathbb{E}_1\mathbb{E}_2F+\beta \left( \mathbb{E}_1\mathbb{E}_2G-G_{j,s} \right)\mathbb{E}_1\mathbb{E}_2G,
\end{equation*}
then it is easily verifiable that $\mathbb{E}_1\mathbb{E}_2Y_{j,s}=0$ for all $1\leq j\leq N_B$, $1\leq s\leq N_F$, and
\begin{equation*}
  p \left( N_B,N_F,\beta \right)=\mathbb{P}\left\{ \frac{1}{N_BN_F}\sum_j\sum_s Y_{j,s}> \beta\left( \mathbb{E}_1\mathbb{E}_2G \right)^2 \right\}.
\end{equation*}
By Lemma~\ref{lem:large_deviation_bound}, bounding this quantity reduces to computing various moments. 
Define
\begin{equation*}
  X_j:=\mathbb{E}_2Y_j,
\end{equation*}
then $X_1,\cdots,X_{N_B}$ are i.i.d. on $M$ with respect to $\overline{p}$, and $\mathbb{E}_1X_j=0$ for $1\leq j\leq N_B$. Furthermore, $X_1,\cdots,X_{N_B}$ are uniformly bounded. To find this bound explicitly, note that
\begin{equation*}
  \begin{aligned}
    \left| X_j \right|&=\left| \mathbb{E}_2Y_j \right|=\left| \left(\mathbb{E}_2F_j\right)\mathbb{E}_1\mathbb{E}_2G-\left(\mathbb{E}_2G_j\right)\mathbb{E}_1\mathbb{E}_2F+\beta \left( \mathbb{E}_1\mathbb{E}_2G-\mathbb{E}_2G_j \right)\mathbb{E}_1\mathbb{E}_2G \right|\\
    &\leq \left|\left(\mathbb{E}_2F_j\right)\mathbb{E}_1\mathbb{E}_2G \right|+ \left|\left(\mathbb{E}_2G_j\right)\mathbb{E}_1\mathbb{E}_2F\right|+\beta \left( \mathbb{E}_1\mathbb{E}_2G \right)^2+\beta \left| \mathbb{E}_2G_j \right|\left| \mathbb{E}_1\mathbb{E}_2G \right|,
  \end{aligned}
\end{equation*}
and recall from Lemma~\ref{lem:euclidean_kernel_integral} and Remark~\ref{rem:lazy_remark} that
\begin{equation*}
  \begin{aligned}
    \mathbb{E}_1\mathbb{E}_2F &= O \left( \epsilon^{\frac{d}{2}}\delta^{\frac{d-1}{2}} \right),\quad \mathbb{E}_1\mathbb{E}_2G = O \left( \epsilon^{\frac{d}{2}}\delta^{\frac{d-1}{2}} \right),\\
    \mathbb{E}_2F_j&=O \left( \delta^{\frac{d-1}{2}} \right),\quad \mathbb{E}_2G_j=O \left( \delta^{\frac{d-1}{2}} \right),
  \end{aligned}
\end{equation*}
thus
\begin{equation*}
  \left| X_j \right|\leq  \tilde{C}\epsilon^{\frac{d}{2}}\delta^{d-1}+\beta \left( \epsilon^d\delta^{d-1}+\epsilon^{\frac{d}{2}}\delta^{d-1} \right)
\end{equation*}
where $\tilde{C}$ is some positive constant depending on the pointwise bounds of $K$, $p$, and $f$. Since we will be mostly interested in small $\beta>0$, let us pick $\beta = O \left( \epsilon^2+\delta^2 \right)$ and write the upper bound on $\left| X_j \right|$ as
\begin{equation}
\label{eq:constant_C}
  \left| X_j \right|\leq C\epsilon^{\frac{d}{2}}\delta^{d-1}, \quad C=C \left(\left\|K\right\|_{\infty}, \left\|f\right\|_{\infty}, p_m, p_M\right)>0.
\end{equation}
We then need to bound $\mathbb{E}_1X_j^2$. Note that
\begin{align*}
    \mathbb{E}_1X_j^2&=\mathbb{E}_1\left[\left(\mathbb{E}_2F_j\right)^2\right]\left(\mathbb{E}_1\mathbb{E}_2G\right)^2+\mathbb{E}_1\left[\left(\mathbb{E}_2G_j\right)^2\right]\left(\mathbb{E}_1\mathbb{E}_2F\right)^2\\
    &\quad-2 \mathbb{E}_1\left[\left( \mathbb{E}_2F_j \right)\left( \mathbb{E}_2G_j \right)\right]\left( \mathbb{E}_1\mathbb{E}_2F \right)\left( \mathbb{E}_1\mathbb{E}_2G \right)+\beta^2  \mathbb{E}_1 \left[\left( \mathbb{E}_1\mathbb{E}_2G-\mathbb{E}_2G_j \right)^2\right]\left(\mathbb{E}_1\mathbb{E}_2G\right)^2\\
    &\quad+2\beta \left(\mathbb{E}_1\mathbb{E}_2G\right)\mathbb{E}_1\big\{\left(\mathbb{E}_1\mathbb{E}_2G-\mathbb{E}_2G_j \right) \left[\left(\mathbb{E}_2F_j\right)\mathbb{E}_1\mathbb{E}_2G-\left(\mathbb{E}_2G_j\right)\mathbb{E}_1\mathbb{E}_2F \right]\big\}\\
    &=\left[\mathbb{E}_1\left(\mathbb{E}_2F_j\right)^2\right]\left(\mathbb{E}_1\mathbb{E}_2G\right)^2+\left[\mathbb{E}_1\left(\mathbb{E}_2G_j\right)^2\right]\left(\mathbb{E}_1\mathbb{E}_2F\right)^2\\
    &\quad-2 \mathbb{E}_1\left[\left( \mathbb{E}_2F_j \right)\left( \mathbb{E}_2G_j \right)\right]\left( \mathbb{E}_1\mathbb{E}_2F \right)\left( \mathbb{E}_1\mathbb{E}_2G \right)+\beta^2 \left( \mathbb{E}_1\mathbb{E}_2G \right)^2 \left[ \mathbb{E}_1 \left(\mathbb{E}_2G\right)^2-\left( \mathbb{E}_1\mathbb{E}_2G \right)^2 \right]\\
    &\quad+2\beta\left(\mathbb{E}_1\mathbb{E}_2G\right) \left[ \mathbb{E}_1 \left( \mathbb{E}_2G \right)^2\mathbb{E}_1\mathbb{E}_2F-\left(\mathbb{E}_1\mathbb{E}_2G\right)\mathbb{E}_1 \left( \mathbb{E}_2F_j\mathbb{E}_2G_j \right) \right],
\end{align*}
it suffices to compute the first and second moments of $\mathbb{E}_2F_j$, $\mathbb{E}_2G_j$ for $1\leq j\leq N_B$. By~\eqref{eq:numerator_taylor_expansion_utm_hat},
\begin{align*}
  \mathbb{E}_1\mathbb{E}_2F&=\epsilon^{\frac{d}{2}}\delta^{\frac{d-1}{2}}\bigg\{m_0\left[fp\right] \left( x_{i,r} \right)+\epsilon \frac{m_{21}}{2}\left( \Delta_S^H\left[fp\right]\left( x_{i,r} \right)+E_1 \left( \xi_i \right)\left[ fp\right]\left( x_{i,r} \right) \right)\\
   &+\delta \frac{m_{22}}{2}\left( \Delta_S^V \left[ fp\right]\left( x_{i,r} \right)+E_2\cdot\left[fp\right]\left( x_{i,r} \right)\right) +O \left( \epsilon^2+\delta^2 \right) \bigg\},
\end{align*}
and if we set $f\equiv 1$, then
\begin{align*}
  \mathbb{E}_1\mathbb{E}_2G&=\epsilon^{\frac{d}{2}}\delta^{\frac{d-1}{2}}\bigg\{m_0p \left( x_{i,r} \right)+\epsilon \frac{m_{21}}{2}\left( \Delta_S^Hp\left( x_{i,r} \right)+E_1 \left( \xi_i \right)p\left( x_{i,r} \right) \right)\\
   &+\delta \frac{m_{22}}{2}\left( \Delta_S^Vp\left( x_{i,r} \right)+E_2 p\left( x_{i,r} \right)\right) +O \left( \epsilon^2+\delta^2 \right) \bigg\}.
\end{align*}
Before we turn to computation of second moments, let us introduce another notation: write the conditional probability density function on fibre $S_x$ as
\begin{equation*}
  p \left( v\mid x \right)=\frac{p \left( x,v \right)}{\overline{p}\left( x \right)}=\frac{p \left( x,v \right)}{\overline{p}\circ\pi \left( x,v \right)}=\left[ \frac{p}{\overline{p}\circ\pi} \right]\left( x,v \right).
\end{equation*}
where $\pi:UTM\rightarrow M$ is the canonical projection from $UTM$ to $M$. This will help us avoid creating new notations for the base and fibre components of the coordinates of $x_{i,r}$, as needed in $p \left( v\mid x \right)$. Applying Lemma~\ref{lem:euclidean_kernel_integral} once, we have
\begin{align*}
  &\mathbb{E}_2F_j=\delta^{\frac{d-1}{2}}\bigg\{ M_0 \left( \frac{\left\|\xi_i-\xi_j\right\|^2}{\epsilon} \right)\left[ \frac{fp}{\overline{p}\circ\pi} \right]\left(P_{\xi_j,\xi_i}x_{i,r} \right)\\
  &+\frac{\delta}{2}M_2 \left( \frac{\left\|\xi_i-\xi_j\right\|^2}{\epsilon} \right)\left[ \Delta_S^V \left[ \frac{fp}{\overline{p}\circ\pi} \right]\left(P_{\xi_j,\xi_i}x_{i,r} \right)+E_2\cdot\left[ \frac{fp}{\overline{p}\circ\pi} \right]\left(P_{\xi_j,\xi_i}x_{i,r} \right) \right]+O \left( \delta^2 \right) \bigg\},
\end{align*}
where $M_0 \left( \cdot \right)$, $M_2 \left( \cdot \right)$ are functions depending only on the kernel $K$, as in the proof of Lemma~\ref{lem:asymp_sym_kernels_utm}. For simplicity of notation, let us write them as $M_0$,$M_2$ for short. Now we square both sides of the equality above
\begin{align*}
  \left( \mathbb{E}_2F_j \right)^2&=\delta^{d-1} \bigg\{M^2_0 \left[ \frac{fp}{\overline{p}\circ\pi} \right]^2\left(P_{\xi_j,\xi_i}x_{i,r} \right)\\
  &+\delta M_0M_2\left[ \frac{fp}{\overline{p}\circ\pi} \right]\left(P_{\xi_j,\xi_i}x_{i,r} \right)\left( \Delta_S^V \left[ \frac{fp}{\overline{p}\circ\pi} \right]\left(P_{\xi_j,\xi_i}x_{i,r} \right)+E_2\cdot\left[ \frac{fp}{\overline{p}\circ\pi} \right]\left(P_{\xi_j,\xi_i}x_{i,r} \right) \right)+O \left( \delta^2 \right) \bigg\},
\end{align*}
and apply Lemma~\ref{lem:euclidean_kernel_integral} to get
\begin{align*}
  &\mathbb{E}_1\left( \mathbb{E}_2F_j \right)^2=\epsilon^{\frac{d}{2}}\delta^{d-1}\bigg\{m'_0\left[ \frac{\left(fp\right)^2}{\overline{p}\circ\pi} \right]\left( x_{i,r} \right)+\epsilon \frac{m_{21}'}{2} \left( \Delta_S^H\left[ \frac{\left(fp\right)^2}{\overline{p}\circ\pi} \right]\left( x_{i,r} \right)+E_1 \left( \xi_i \right)\left[ \frac{\left(fp\right)^2}{\overline{p}\circ\pi} \right]\left( x_{i,r} \right) \right)\\
  &+\delta m_{22}' \left[ fp \right] \left( x_{i,r} \right)\left( \Delta_S^V \left[ \frac{fp}{\overline{p}\circ\pi} \right]\left( x_{i,r} \right)+E_2\cdot\left[ \frac{fp}{\overline{p}\circ\pi} \right]\left( x_{i,r} \right) \right)+O \left( \epsilon^2+\delta^2 \right)\bigg\},
\end{align*}
where $m_0'$, $m_{21}'$, $m_{22}'$ are positive constants determined by the kernel function $K$ and dimension $d$:
\begin{align*}
    m_0'&=\int_{B_1^d \left( 0 \right)}M_0^2 \left( r^2 \right)ds^1\cdots ds^d,\quad r^2=\sum_{k=1}^d\left(s^k\right)^2,\\
    m_{21}'&=\int_{B_1^d \left( 0 \right)}M_0 \left( r^2 \right)\left( s^1 \right)^2ds^1\cdots ds^d,\\
    m_{22}'&=\int_{B_1^d \left( 0 \right)}M_0 \left( r^2 \right)M_2 \left( r^2 \right)ds^1\cdots ds^d.
\end{align*}
Setting $f\equiv 1$, we obtain a similar expansion for the variance of $\mathbb{E}_2G_j$
\begin{align*}
  &\mathbb{E}_1\left( \mathbb{E}_2G_j \right)^2=\epsilon^{\frac{d}{2}}\delta^{d-1}\bigg\{m'_0\left[ \frac{p^2}{\overline{p}\circ\pi} \right]\left( x_{i,r} \right)+\epsilon \frac{m_{21}'}{2} \left( \Delta_S^H\left[ \frac{p^2}{\overline{p}\circ\pi} \right]\left( x_{i,r} \right)+E_1 \left( \xi_i \right)\left[ \frac{p^2}{\overline{p}\circ\pi} \right]\left( x_{i,r} \right) \right)\\
  &+\delta m_{22}' p \left( x_{i,r} \right) \left(\Delta_S^V \left[ \frac{p}{\overline{p}\circ\pi} \right]\left( x_{i,r} \right)+E_2\cdot\left[ \frac{p}{\overline{p}\circ\pi} \right]\left( x_{i,r} \right) \right)+O \left( \epsilon^2+\delta^2 \right)\bigg\}.
\end{align*}
Similarly,
\begin{align*}
  &\mathbb{E}_1 \left[ \left(\mathbb{E}_2F_j\right)\left(\mathbb{E}_2G_j\right) \right]=\epsilon^{\frac{d}{2}}\delta^{d-1}\bigg\{m'_0 \left[ \frac{fp^2}{\overline{p}\circ\pi}\right]\left( x_{i,r} \right)\\
  &+\epsilon \frac{m_{21}'}{2} \left( \Delta_S^H\left[ \frac{fp^2}{\overline{p}\circ\pi} \right]\left( x_{i,r} \right)+E_1 \left( \xi_i \right)\left[ \frac{fp^2}{\overline{p}\circ\pi} \right]\right)\\
  &+\delta \frac{m_{22}'}{2} \bigg( p\left( x_{i,r} \right) \Delta_S^V \left[ \frac{fp}{\overline{p}\circ\pi} \right]\left( x_{i,r} \right)+\left[fp\right]\left( x_{i,r} \right) \Delta_S^V \left[ \frac{p}{\overline{p}\circ\pi} \right]\left( x_{i,r} \right)\\
  &+2E_2\cdot \left[ \frac{fp^2}{\overline{p}\circ\pi} \right]\left( x_{i,r} \right) \bigg)+O \left( \epsilon^2+\delta^2 \right)\bigg\}.
\end{align*}
It remains to plug all these moment expansions back into $\mathbb{E}_1X_j^2$. Clearly, we are only interested in scenarios in which $\beta$ is sufficiently small, say $\beta=O \left( \epsilon^2+\delta^2 \right)$, thus the $O \left( \beta \right)$ and $O \left( \beta^2 \right)$ terms in $\mathbb{E}_1X_j^2$ can be thrown into $O \left[\epsilon^{\frac{3d}{2}}\delta^{2 \left( d-1 \right)}\left( \epsilon^2+\delta^2 \right)\right]$. Direct computation yields
\begin{align*}
  \mathbb{E}_1X_j^2&=\left[\mathbb{E}_1\left(\mathbb{E}_2F_j\right)^2\right]\left(\mathbb{E}_1\mathbb{E}_2G\right)^2+\left[\mathbb{E}_1\left(\mathbb{E}_2G_j\right)^2\right]\left(\mathbb{E}_1\mathbb{E}_2F\right)^2\\
   &\qquad-2 \mathbb{E}_1\left[\left( \mathbb{E}_2F_j \right)\left( \mathbb{E}_2G_j \right)\right]\left( \mathbb{E}_1\mathbb{E}_2F \right)\left( \mathbb{E}_1\mathbb{E}_2G \right)+O \left( \epsilon^2+\delta^2 \right)\\
   &=\epsilon^{\frac{3d}{2}}\delta^{2 \left( d-1 \right)}\bigg\{\epsilon \frac{m_0^2m_{21}'}{2}p^2\left( x_{i,r} \right)\bigg(\Delta_S^H \left[ \frac{\left(fp\right)^2}{\overline{p}\circ\pi} \right] \left( x_{i,r} \right)+f^2 \left( x_{i,r} \right)\Delta_S^H\left[ \frac{p^2}{\overline{p}\circ\pi} \right]\left( x_{i,r} \right)\\
   &\qquad-2f \left( x_{i,r} \right)\Delta_S^H \left[\frac{fp^2}{\overline{p}\circ\pi} \right]\left( x_{i,r} \right) \bigg)+O \left( \epsilon^2+\delta^2 \right)\bigg\}\\
   &\overset{\left(*\right)}{=}\epsilon^{\frac{3d}{2}}\delta^{2 \left( d-1 \right)}\bigg\{\epsilon \frac{m_0^2m_{21}'}{2}p^2\left( x_{i,r} \right)\cdot 2\left[ \frac{p^2}{\overline{p}\circ\pi} \right]\left( x_{i,r} \right)\left\|\nabla_S^H f \right\|^2\left( x_{i,r} \right)\\
   &\qquad\qquad\qquad +O \left( \epsilon^2+\delta^2 \right)\bigg\}\\
   &=\epsilon^{\frac{3d}{2}}\delta^{2 \left( d-1 \right)}\bigg\{\epsilon m_0^2m_{21}'\left[ \frac{p^4}{\overline{p}\circ\pi} \right]\left( x_{i,r} \right)\left\|\nabla_S^H f \right\|^2 \left( x_{i,r} \right)+O \left( \epsilon^2+\delta^2 \right)\bigg\}.
\end{align*}
Note that at $\left( * \right)$ we used (denote $g=p^2/\left(\overline{p}\circ\pi\right)$ for short)
\begin{align*}
  \Delta_S^H \left( f^2g \right)&+f^2\Delta_S^Hg-2f\Delta_S^H \left( fg \right)\\
  &=\left( \Delta_S^Hf^2 \right)g+ f^2\Delta_S^Hg+2\left\langle\nabla_S^Hf^2,\nabla_S^Hg\right\rangle+f^2\Delta_S^Hg\\
  &\qquad-2f^2\Delta_S^Hg-2fg\Delta_S^Hf-4f\left\langle\nabla_S^Hf,\nabla_S^Hg\right\rangle\\
  &=\left[ \Delta_S^Hf^2-2f\Delta_S^Hf \right]g+2\left\langle\nabla_S^Hf^2-2f\nabla_S^Hf, \nabla_S^Hg  \right\rangle\\
  &=2\left\|\nabla^H_Sf\right\|^2g.
\end{align*}
Therefore, we can bound $\mathbb{E}_1 X_j^2$ uniformly in $j$ as
\begin{equation*}
  \mathbb{E}_1X_j^2\leq \epsilon^{\frac{3d}{2}}\delta^{2 \left( d-1 \right)}\left( C'\epsilon+O \left( \epsilon^2+\delta^2 \right)\right)
\end{equation*}
where
\begin{equation*}
  C'=\frac{m_0^2m_{21}'p_M^4\left\|\nabla_S^Hf\right\|_{\infty}^2}{\omega_{d-1}^4p_m}
\end{equation*}
is a positive constant. Interestingly, $O \left( \delta \right)$ terms do not show up in this bound. In hindsight, this makes sense because $X_j=\mathbb{E}_2Y_{j,s}$ is already the expectation along the fibre direction, which intuitively ``froze'' the variability controlled by the fibrewise bandwidth $\delta$.

It remains to bound
\begin{equation*}
  \mathbb{E}_2^{\xi_j}Y_j^2-\left(\mathbb{E}_2^{\xi_j}Y_j\right)^2
\end{equation*}
for each $1\leq j\leq N_B$. For $\beta= O \left( \epsilon^2+\delta^2 \right)$, a bound for $\left| Y_{j,s} \right|$ can be found as
\begin{equation*}
  \begin{aligned}
    \left| Y_{j,s} \right| &= \left| F_{j,s}\mathbb{E}_1\mathbb{E}_2G-G_{j,s}\mathbb{E}_1\mathbb{E}_2F+\beta \left( \mathbb{E}_1\mathbb{E}_2G-G_{j,s} \right)\mathbb{E}_1\mathbb{E}_2G \right|\\
    &=O \left( \epsilon^{\frac{d}{2}}\delta^{\frac{d-1}{2}} \left\|K\right\|_{\infty}\left\|f\right\|_{\infty} \right)+O \left( \epsilon^{\frac{d}{2}}\delta^{\frac{d-1}{2}}\left\|K\right\|_{\infty}\left\|f\right\|_{\infty} \right)\\
    &\quad+\beta \left( O \left( \epsilon^{\frac{d}{2}}\delta^{\frac{d-1}{2}}\left\|K\right\|_{\infty} \right)+O \left( \epsilon^d\delta^{d-1}\left\|K\right\|_{\infty} \right) \right)\\
    &\leq C\epsilon^{\frac{d}{2}}\delta^{\frac{d-1}{2}}
  \end{aligned}
\end{equation*}
where
\begin{equation*}
  C = C \left( \left\|K\right\|_{\infty},\left\|f\right\|_{\infty},p_m,p_M \right)
\end{equation*}
is a positive constant independent of $j$. Again taking advantage of $\beta=O \left( \epsilon^2+\delta^2 \right)$, we have
\begin{align*}
  \mathbb{E}_2Y_j^2&=\mathbb{E}_2 F_{j,s}^2\left(\mathbb{E}_1\mathbb{E}_2G\right)^2+\mathbb{E}_2 G_{j,s}^2\left(\mathbb{E}_1\mathbb{E}_2F\right)^2-2\mathbb{E}_2 \left( F_{j,s}G_{j,s} \right)\left(\mathbb{E}_1\mathbb{E}_2F\right)\left(\mathbb{E}_1\mathbb{E}_2G\right)\\
  &+O \left[\epsilon^d\delta^{2 \left( d-1 \right)} \left( \epsilon^2+\delta^2 \right)\right],\\
  \left(\mathbb{E}_2Y_j\right)^2&=\left[\left(\mathbb{E}_2F_j\right)\mathbb{E}_1\mathbb{E}_2G-\left(\mathbb{E}_2G_j\right)\mathbb{E}_1\mathbb{E}_2F+\beta \left( \mathbb{E}_1\mathbb{E}_2G-\mathbb{E}_2G_j \right)\mathbb{E}_1\mathbb{E}_2G\right]^2\\
  &=\left(\mathbb{E}_2F_j\right)^2\left(\mathbb{E}_1\mathbb{E}_2G\right)^2+\left(\mathbb{E}_2G_j\right)^2\left(\mathbb{E}_1\mathbb{E}_2F\right)^2\\
  &\qquad-2\left(\mathbb{E}_2F_j\right)\left(\mathbb{E}_1\mathbb{E}_2G\right)\left(\mathbb{E}_2G_j\right)\left(\mathbb{E}_1\mathbb{E}_2F\right)+O \left[\epsilon^d\delta^{2 \left( d-1 \right)} \left( \epsilon^2+\delta^2 \right)\right],
\end{align*}
thus
\begin{equation*}
  \begin{aligned}
    \mathbb{E}_2Y_j^2-\left(\mathbb{E}_2Y_j\right)^2&=\left[ \mathbb{E}_2F_{j,s}^2- \left( \mathbb{E}_2F_j \right)^2 \right]\left(\mathbb{E}_1\mathbb{E}_2G\right)^2+\left[ \mathbb{E}_2G_{j,s}^2- \left( \mathbb{E}_2G_j \right)^2 \right]\left(\mathbb{E}_1\mathbb{E}_2F\right)^2\\
    &\quad+2 \left[ \left( \mathbb{E}_2F_j \right)\left( \mathbb{E}_2G_j \right)-\mathbb{E}_2 \left( F_{j,s}G_{j,s} \right) \right])\left(\mathbb{E}_1\mathbb{E}_2F\right)\left(\mathbb{E}_1\mathbb{E}_2G\right)\\
    &\quad+O \left[\epsilon^d\delta^{2 \left( d-1 \right)} \left( \epsilon^2+\delta^2 \right)\right].
  \end{aligned}
\end{equation*}
Observe that
\begin{equation*}
  \mathbb{E}_2F_{j,s}^2=O \left( \delta^{\frac{d-1}{2}} \right),\quad \mathbb{E}_2G_{j,s}^2=O \left( \delta^{\frac{d-1}{2}} \right),\quad \mathbb{E}_2 \left[ F_{j,s}G_{j,s} \right]=O \left( \delta^{\frac{d-1}{2}} \right),
\end{equation*}
while
\begin{equation*}
  \left(\mathbb{E}_2F_{j,s}\right)^2=O \left( \delta^{d-1} \right),\quad\left(\mathbb{E}_2G_{j,s}\right)^2=O \left( \delta^{d-1} \right),\quad\left( \mathbb{E}_2F_j \right)\left( \mathbb{E}_2G_j \right)=O \left( \delta^{d-1} \right),
\end{equation*}
thus the leading order error in $\mathbb{E}_2Y_j^2-\left(\mathbb{E}_2Y_j\right)^2$ are determined by
\begin{equation*}
  \left(\mathbb{E}_2F_{j,s}^2\right)\left(\mathbb{E}_1\mathbb{E}_2G\right)^2+\left(\mathbb{E}_2G_{j,s}^2\right)\left(\mathbb{E}_1\mathbb{E}_2F\right)^2-2\mathbb{E}_2 \left( F_{j,s}G_{j,s} \right)\left(\mathbb{E}_1\mathbb{E}_2F\right)\left(\mathbb{E}_1\mathbb{E}_2G\right).
\end{equation*}
Note that by Lemma~\ref{lem:euclidean_kernel_integral}
\begin{equation*}
  \begin{aligned}
    \mathbb{E}_2F_{j,s}^2&=\int_{S_{\xi_j}}K^2 \left( \frac{\left\|\xi_i-\xi_j\right\|^2}{\epsilon}, \frac{\left\|P_{\xi_j,\xi_i}x_{i,r}-\left( \xi_j,w \right)\right\|^2}{\delta} \right)f^2 \left( \xi_j,w \right)p \left( w\mid\xi_j \right)d\sigma_{\xi_j}\left( w \right)\\
    &=\delta^{\frac{d-1}{2}}\bigg\{\widetilde{M}_0 \left( \frac{\left\|\xi_i-\xi_j\right\|^2}{\epsilon} \right)f^2 \left( P_{\xi_j,\xi_i}x_{i,r} \right)\left[\frac{p}{\overline{p}\circ\pi}\right]\left( P_{\xi_j,\xi_i}x_{i,r} \right)\\
    &+\frac{\delta}{2}\widetilde{M}_2\left( \frac{\left\|\xi_i-\xi_j\right\|^2}{\epsilon} \right)\left[ \Delta_S^V \left[ \frac{f^2p}{\overline{p}\circ\pi} \right]\left( P_{\xi_j,\xi_i}x_{i,r} \right)+E_2\cdot\left[ \frac{f^2p}{\overline{p}\circ\pi} \right]\left( P_{\xi_j,\xi_i}x_{i,r} \right) \right]\\
    &+O \left( \delta^2 \right)  \bigg\},
  \end{aligned}
\end{equation*}
where $\widetilde{M}_0$, $\widetilde{M}_2$ are constants depending on $\xi_j$ but uniformly bounded over $M$:
\begin{align*}
  \left|\widetilde{M}_0\left( \frac{\left\|\xi_i-\xi_j\right\|^2}{\epsilon} \right)\right|&=\left|\int_{B_1^{d-1} \left( 0 \right)}K^2 \left(\frac{\left\|\xi_i-\xi_j\right\|^2}{\epsilon}, \rho^2 \right)d\theta^1\cdots d\theta^{d-1} \right|\leq \left\|K\right\|_{\infty}^2\mathrm{Vol}\left( B_1^{d-1}\left( 0 \right) \right),\\
  \left|\widetilde{M}_2\left( \frac{\left\|\xi_i-\xi_j\right\|^2}{\epsilon} \right)\right|&=\left|\int_{B_1^{d-1} \left( 0 \right)}\left( \theta^1 \right)^2K^2 \left(\frac{\left\|\xi_i-\xi_j\right\|^2}{\epsilon}, \rho^2 \right)d\theta^1\cdots d\theta^{d-1} \right|\leq \left\|K\right\|_{\infty}^2\mathrm{Vol}\left( B_1^{d-1}\left( 0 \right) \right).
\end{align*}
For simplicity of notation, we shall denote these constants merely as $\widetilde{M}_0$, $\widetilde{M}_2$, dropping the dependency on $\xi_j$. The expansion for $\mathbb{E}_2F_{j,s}^2$ is thus
\begin{equation*}
  \begin{aligned}
    \mathbb{E}_2F_{j,s}^2&=\delta^{\frac{d-1}{2}}\bigg\{\widetilde{M}_0\left[\frac{f^2p}{\overline{p}\circ\pi}\right]\left( P_{\xi_j,\xi_i}x_{i,r} \right)+\frac{\delta}{2}\widetilde{M}_2\bigg[ \Delta_S^V \left[ \frac{f^2p}{\overline{p}\circ\pi} \right]\left( P_{\xi_j,\xi_i}x_{i,r} \right)\\
    &\qquad\qquad\qquad\qquad\qquad\qquad+E_2\cdot\left[ \frac{f^2p}{\overline{p}\circ\pi} \right]\left( P_{\xi_j,\xi_i}x_{i,r} \right) \bigg]+O \left( \delta^2 \right)  \bigg\}.
  \end{aligned}
\end{equation*}
Similarly,
\begin{equation*}
  \begin{aligned}
   \mathbb{E}_2G_{j,s}^2&=\delta^{\frac{d-1}{2}}\bigg\{\widetilde{M}_0\left[\frac{p}{\overline{p}\circ\pi}\right]\left( P_{\xi_j,\xi_i}x_{i,r} \right)+\frac{\delta}{2}\widetilde{M}_2\bigg[ \Delta_S^V \left[ \frac{p}{\overline{p}\circ\pi} \right]\left( P_{\xi_j,\xi_i}x_{i,r} \right)\\
    &\qquad\qquad\qquad\qquad\qquad\qquad+E_2\cdot\left[ \frac{p}{\overline{p}\circ\pi} \right]\left( P_{\xi_j,\xi_i}x_{i,r} \right) \bigg]+O \left( \delta^2 \right)  \bigg\},\\
    \mathbb{E}_2\left(F_{j,s}G_{j,s}\right)&=\delta^{\frac{d-1}{2}}\bigg\{\widetilde{M}_0\left[\frac{fp}{\overline{p}\circ\pi}\right]\left( P_{\xi_j,\xi_i}x_{i,r} \right)+\frac{\delta}{2}\widetilde{M}_2\bigg[ \Delta_S^V \left[ \frac{fp}{\overline{p}\circ\pi} \right]\left( P_{\xi_j,\xi_i}x_{i,r} \right)\\
    &\qquad\qquad\qquad\qquad\qquad\qquad+E_2\cdot\left[ \frac{fp}{\overline{p}\circ\pi} \right]\left( P_{\xi_j,\xi_i}x_{i,r} \right) \bigg]+O \left( \delta^2 \right)  \bigg\}.\\
  \end{aligned}
\end{equation*}
A direct computation yields
\begin{align*}
  &\left(\mathbb{E}_2F_{j,s}^2\right)\left(\mathbb{E}_1\mathbb{E}_2G\right)^2+\left(\mathbb{E}_2G_{j,s}^2\right)\left(\mathbb{E}_1\mathbb{E}_2F\right)^2-2\mathbb{E}_2 \left( F_{j,s}G_{j,s} \right)\left(\mathbb{E}_1\mathbb{E}_2F\right)\left(\mathbb{E}_1\mathbb{E}_2G\right)\\
  &=\epsilon^d\delta^{\frac{3 \left( d-1 \right)}{2}}\bigg\{m_0^2\widetilde{M}_0 p^2\left( x_{i,r} \right)\left[ \frac{p}{\overline{p}\circ\pi} \right]\left( P_{\xi_j,\xi_i}x_{i,r} \right)\left[ f \left( P_{\xi_j,\xi_i}x_{i,r} \right)-f \left( x_{i,r} \right) \right]^2\\
  &+\epsilon m_0m_{21}\widetilde{M}_0p \left( x_{i,r} \right) \left[ \frac{p}{\overline{p}\circ\pi} \right]\left( P_{\xi_j,\xi_i}x_{i,r} \right) \bigg[\left( \Delta_S^Hp\left( x_{i,r} \right)+E_1 \left( \xi \right)p\left( x_{i,r} \right) \right)\left[ f \left( P_{\xi_j,\xi_i}x_{i,r} \right)-f \left( x_{i,r} \right) \right]^2\\
  &-\left(\Delta_S^H\left[fp\right]\left( x_{i,r} \right)-f \left( x_{i,r} \right)\Delta_S^H p\left( x_{i,r} \right) \right)\left[ f \left( P_{\xi_j,\xi_i}x_{i,r} \right)-f \left( x_{i,r} \right) \right]  \bigg]\\
  &+\delta m_0^2\widetilde{M}_2p^2\left( x_{i,r} \right)\left[ \frac{p}{\overline{p}\circ\pi} \right]\left( P_{\xi_j,\xi_i}x_{i,r} \right)\left\|\nabla_S^Hf \left( P_{\xi_j,\xi_i}x_{i,r} \right)\right\|^2\\
  &+\delta m_0m_{22}\widetilde{M}_0p \left( x_{i,r} \right) \left[ \frac{p}{\overline{p}\circ\pi} \right]\left( P_{\xi_j,\xi_i}x_{i,r} \right)\bigg[\left( \Delta_S^Vp\left( x_{i,r} \right)+E_2\cdot p\left( x_{i,r} \right) \right)\left[ f \left( P_{\xi_j,\xi_i}x_{i,r} \right)-f \left( x_{i,r} \right) \right]^2\\
  &-\left(\Delta_S^V\left[fp\right]\left( x_{i,r} \right)-f \left( x_{i,r} \right)\Delta_S^V p\left( x_{i,r} \right) \right)\left[ f \left( P_{\xi_j,\xi_i}x_{i,r} \right)-f \left( x_{i,r} \right) \right]  \bigg]+O \left( \epsilon^2+\delta^2 \right) \bigg\}.
\end{align*}
Recall from Lemma~\ref{lem:taylor_parallel_transport} that the difference between $P_{\xi_j,\xi_i}x_{i,r}$ and $x_{i,r}$ along the fibre direction is $O \left( d^2_M \left( \xi_j,\xi_i \right) \right)=O \left( \epsilon \right)$. Therefore, the distance (under the Sasaki metric) between $P_{\xi_j,\xi_i}x_{i,r}$ and $x_{i,r}$ in $UTM$ is bounded by the square root of the square sum of $d_M \left( \xi_j,\xi_i \right)$ and the difference between $P_{\xi_j,\xi_i}x_{i,r}$ and $x_{i,r}$ along the fibre direction, which is of order $O \left( \epsilon^{\frac{1}{2}} \right)$. As a result, all terms involving\
\begin{equation*}
  \left[ f \left( P_{\xi_j,\xi_i}x_{i,r} \right)-f \left( x_{i,r} \right) \right]
\end{equation*}
are of order $O \left( \epsilon^{\frac{1}{2}} \right)$. Thus
\begin{equation*}
  \begin{aligned}
    &\left| \left(\mathbb{E}_2F_{j,s}^2\right)\left(\mathbb{E}_1\mathbb{E}_2G\right)^2+\left(\mathbb{E}_2G_{j,s}^2\right)\left(\mathbb{E}_1\mathbb{E}_2F\right)^2-2\mathbb{E}_2 \left( F_{j,s}G_{j,s} \right)\left(\mathbb{E}_1\mathbb{E}_2F\right)\left(\mathbb{E}_1\mathbb{E}_2G\right) \right|\\
    &\leq \epsilon^d\delta^{\frac{3 \left( d-1 \right)}{2}} \left( C'\epsilon+C''\delta \right),\quad C'>0,C''>0.
  \end{aligned}
\end{equation*}
As a result,
\begin{equation*}
  \mathbb{E}_2^{\xi_j}Y_j^2-\left(\mathbb{E}_2^{\xi_j}Y_j\right)^2=O \left( \epsilon^d\delta^{\frac{3 \left( d-1 \right)}{2}}\left( \epsilon+\delta \right) \right).
\end{equation*}
We are now ready for applying Lemma~\ref{lem:large_deviation_bound} to $Y_{j,s}$. Since
\begin{equation*}
  \left|\mathbb{E}_1\mathbb{E}_2G\right|=O \left( \epsilon^{\frac{d}{2}}\delta^{\frac{d-1}{2}} \right),
\end{equation*}
we have constants $C''_1,C''_2$ such that
\begin{equation*}
  C_1''\epsilon^{\frac{d}{2}}\delta^{\frac{d-1}{2}}\leq \left|\mathbb{E}_1\mathbb{E}_2G\right| \leq C_2''\epsilon^{\frac{d}{2}}\delta^{\frac{d-1}{2}}.
\end{equation*}
For any $\theta\in\left( 0,1 \right)$ to be fixed later,
\begin{align*}
  &p \left( N_B,N_F,\beta \right)=\mathbb{P}\left\{ \frac{1}{N_BN_F}\sum_j\sum_s Y_{j,s}> \beta\left( \mathbb{E}_1\mathbb{E}_2G \right)^2 \right\}\\
  &\leq \sum_{j=1}^{N_B}\exp \left\{ -\frac{\displaystyle \frac{1}{2}\left(1-\theta\right)^2N_F\beta^2 \left( \mathbb{E}_1\mathbb{E}_2G \right)^4}{\displaystyle \left[\mathbb{E}_2^{\xi_j}Y_j^2-\left(\mathbb{E}_2^{\xi_j}Y_j\right)^2 \right]+\frac{2}{3}\cdot 2C\epsilon^{\frac{d}{2}}\delta^{\frac{d-1}{2}}\left(1-\theta\right)\beta\left( \mathbb{E}_1\mathbb{E}_2G \right)^2} \right\}\\
  &+\exp \left\{ -\frac{\displaystyle \frac{1}{2}\theta^2N_B\beta^2 \left( \mathbb{E}_1\mathbb{E}_2G \right)^4}{\displaystyle \mathbb{E}_1 X_j^2+\frac{2}{3}C\epsilon^{\frac{d}{2}}\delta^{d-1}\theta\beta\left( \mathbb{E}_1\mathbb{E}_2G \right)^2} \right\}\\
  &\leq N_B\cdot \exp \left\{ -\frac{\displaystyle \frac{1}{2}\left(1-\theta\right)^2N_F\beta^2 \left( C_1'' \right)^4\epsilon^{2d}\delta^{2\left(d-1\right)} }{\displaystyle \tilde{C}\epsilon^d\delta^{\frac{3 \left( d-1 \right)}{2}}\left( \epsilon+\delta \right)+\frac{4}{3}C\epsilon^{\frac{d}{2}}\delta^{\frac{d-1}{2}}\left(1-\theta\right)\beta\cdot\left( C_2'' \right)^2\epsilon^d\delta^{d-1}} \right\}\\
  &+\exp \left\{ -\frac{\displaystyle \frac{1}{2}\theta^2N_B\beta^2\left( C_1'' \right)^4\epsilon^{2d}\delta^{2\left(d-1\right)}}{\displaystyle \left(C'\epsilon+O \left( \epsilon^2+\delta^2 \right)  \right)\epsilon^{\frac{3d}{2}}\delta^{2 \left( d-1 \right)}+\frac{2}{3}C\epsilon^{\frac{d}{2}}\delta^{d-1} \theta\beta\cdot\left( C_2'' \right)^2\epsilon^d\delta^{d-1}} \right\}.
\end{align*}
Again by restricting ourselves to $\beta=O \left( \epsilon^2+\delta^2 \right)$, this bound can be rewritten as
\begin{equation}
\label{eq:special_case_bound}
  p \left( N_B,N_F,\beta \right)\leq N_B\exp \left\{ -\frac{\left(1-\theta\right)^2N_F\epsilon^d\delta^{\frac{d-1}{2}}\beta^2}{C_1 \left( \epsilon+\delta \right)+O\left(\epsilon^{\frac{d}{2}}\left( \epsilon^2+\delta^2 \right)\right)} \right\}+\exp \left\{ -\frac{\theta^2N_B\epsilon^{\frac{d}{2}}\beta^2}{C_2\epsilon+O \left( \epsilon^2+\delta^2 \right)} \right\}.
\end{equation}
As pointed out in Remark~\ref{rem:insight_large_deviation_bound}, the second term in this bound is the sampling error on the base manifold; the noise error resulted from this term is of the order
\begin{equation*}
  O\left[\left( N_B\epsilon^{\frac{d}{2}-1} \right)^{-\frac{1}{2}}\right]=O \left(N^{-\frac{1}{2}}_B\epsilon^{\frac{1}{2}-\frac{d}{4}} \right),
\end{equation*}
which is in accordance with the convergence rate obtained in \cite{Singer2006ConvergenceRate}. The first term in the bound reflects the accumulated fibrewise sampling error and grows linearly with respect to the number of fibres sampled, but can be reduced as one increases $N_F$ accordingly (which has an effect of reducing fibrewise sampling errors). The choice of $\theta$ is important: as $\theta$ increases from $0$ to $1$, the first term in the bound decreases but the second term increases. One may wish to pick an ``optimal'' $\theta\in \left( 0,1 \right)$, but this does not make sense unless one chooses $\epsilon,\delta,N_F$ appropriately so as to make the sum of the two terms smaller than $1$. Let us consider $\theta_{*}\in \left( 0,1 \right)$ satisfying
\begin{equation}
\label{eq:seraching_for_theta_star}
  \left(1-\theta_{*}\right)^2N_F\epsilon^d\delta^{\frac{d-1}{2}}=\theta_{*}^2N_B\epsilon^{\frac{d}{2}},
\end{equation}
or equivalently
\begin{equation}
\label{eq:theta_star}
  \epsilon^{\frac{d}{4}}\delta^{\frac{d-1}{4}}\sqrt{\frac{N_F}{N_B}}=\frac{\theta_{*}}{1-\theta_{*}}\Leftrightarrow\theta_{*}=\frac{\displaystyle \epsilon^{\frac{d}{4}}\delta^{\frac{d-1}{4}}\sqrt{\frac{N_F}{N_B}}}{\displaystyle 1+\epsilon^{\frac{d}{4}}\delta^{\frac{d-1}{4}}\sqrt{\frac{N_F}{N_B}}}.
\end{equation}
Setting $\theta=\theta_{*}$ in \eqref{eq:special_case_bound}, we have
\begin{equation}
\label{eq:bound_with_theta_star}
  \begin{aligned}
    p \left( N_B,N_F,\beta \right)& \leq \left( N_B+1 \right)\exp \left\{ -\frac{\theta_{*}^2N_B\epsilon^{\frac{d}{2}}\beta^2}{C \left( \epsilon+\delta \right)} \right\}\\
    &=\exp \left(-\frac{\theta_{*}^2N_B\epsilon^{\frac{d}{2}}\beta^2}{C \left( \epsilon+\delta \right)}+\log \left( N_B+1 \right) \right),
  \end{aligned}
\end{equation}
where $C$ is some positive constant. Since
\begin{equation*}
  \lim_{N_B\rightarrow\infty}\frac{N_B}{\log N_B}=\infty,
\end{equation*}
for any fixed $\epsilon,\delta$ we have $p \left( N_B,N_F,\beta \right)\rightarrow0$ as $N_B\rightarrow\infty$, as long as one increases $N_F$ accordingly so as to prevent $\theta_{*}$ from approaching $0$ or $1$; for instance, this can be achieved by requiring
\begin{equation}
\label{eq:balance_sampling_number}
  \lim_{N_B\rightarrow\infty\atop N_F\rightarrow\infty}\frac{N_F}{N_B}=\rho\in \left( 0,\infty \right).
\end{equation}
Under this condition, we have the pointwise convergence in probability of $\hat{H}_{\epsilon,\delta}^0f$.

We now turn to the general case $\alpha\neq 0$. Recall that
\begin{equation*}
  \begin{aligned}
    \hat{H}_{\epsilon,\delta}^\alpha f \left( x_{i,r} \right)&=\frac{\displaystyle\sum_{j=1}^{N_B}\sum_{s=1}^{N_F}\hat{K}^{\alpha}_{\epsilon,\delta}\left( x_{i,r},x_{j,s} \right)f \left( x_{j,s} \right)}{\displaystyle\sum_{j=1}^{N_B}\sum_{s=1}^{N_F}\hat{K}^{\alpha}_{\epsilon,\delta}\left( x_{i,r},x_{j,s} \right)}=\frac{\displaystyle\sum_{j=1}^{N_B}\sum_{s=1}^{N_F}\frac{\hat{K}_{\epsilon,\delta}\left( x_{i,r},x_{j,s} \right)f \left( x_{j,s} \right)}{\hat{p}^{\alpha}_{\epsilon,\delta}\left( x_{i,r} \right)\hat{p}^{\alpha}_{\epsilon,\delta}\left( x_{j,s} \right)}}{\displaystyle\sum_{j=1}^{N_B}\sum_{s=1}^{N_F}\frac{\hat{K}_{\epsilon,\delta}\left( x_{i,r},x_{j,s} \right)}{\hat{p}^{\alpha}_{\epsilon,\delta}\left( x_{i,r} \right)\hat{p}^{\alpha}_{\epsilon,\delta}\left( x_{j,s} \right)}}
  \end{aligned}
\end{equation*}
where
\begin{equation*}
  \begin{aligned}
    \hat{p}\left( x_{j,s} \right)=\sum_{k=1}^{N_B}\sum_{t=1}^{N_F}\hat{K}_{\epsilon,\delta}\left( x_{j,s},x_{k,t} \right).
  \end{aligned}
\end{equation*}
As $N_B\rightarrow\infty$, $N_F\rightarrow\infty$, by the law of large numbers,
\begin{equation*}
  \begin{aligned}
    \lim_{N_B\rightarrow\infty}\frac{1}{N_B}\lim_{N_F\rightarrow\infty}\frac{1}{N_F}\hat{p}\left( x_{j,s} \right)&=\int_{UTM}\tilde{K}_{\epsilon,\delta}\left( x_{i,r},\eta \right)p \left( \eta \right)d\Theta \left( \eta \right)\\
     &=\tilde{p} \left( x_{i,r} \right)=\mathbb{E}_1\mathbb{E}_2\left[\tilde{K}_{\epsilon,\delta}\left( x_{j,s},\cdot \right)\right].
  \end{aligned}
\end{equation*}
Therefore, as $N_B\rightarrow\infty,N_F\rightarrow\infty$, we expect $\hat{H}_{\epsilon,\delta}^\alpha f \left( x_{i,r} \right)$ to converge to
\begin{equation*}
  \begin{aligned}
    &\frac{\displaystyle\int_{UTM}\tilde{K}^{\alpha}_{\epsilon,\delta} \left( x_{i,r},\eta \right)f \left( \eta \right)p \left( \eta \right)d\Theta \left( y,w \right)}{\displaystyle\int_{UTM}\tilde{K}^{\alpha}_{\epsilon,\delta} \left( x_{i,r},\eta \right)p \left( \eta \right)d\Theta \left( y,w \right)}=\tilde{H}_{\epsilon,\delta}^{\alpha}f \left( x_{i,r} \right)\\
    =&f \left( x_{i,r} \right)+\epsilon \frac{m_{21}}{2m_0}\left[\frac{\Delta_S^H\left[fp^{1-\alpha}\right]\left( x_{i,r} \right)}{p^{1-\alpha}\left( x_{i,r} \right)}-f \left( x_{i,r} \right) \frac{\Delta_S^Hp^{1-\alpha}\left( x_{i,r} \right)}{p^{1-\alpha}\left( x_{i,r} \right)}  \right]\\
    &+\delta \frac{m_{22}}{2m_0}\left[\frac{\Delta_S^V\left[fp^{1-\alpha}\right]\left( x_{i,r} \right)}{p^{1-\alpha}\left( x_{i,r} \right)}-f \left( x_{i,r} \right) \frac{\Delta_S^Vp^{1-\alpha}\left( x_{i,r} \right)}{p^{1-\alpha}\left( x_{i,r} \right)}  \right]+O \left( \epsilon^2+\delta^2 \right),
  \end{aligned}
\end{equation*}
which gives the same bias error $O \left( \epsilon^2+\delta^2 \right)$ as for the $\alpha=0$ case.

Now it remains to estimate the variance error. Since our notation $\hat{p}_{\epsilon,\delta}$ differs from the standard kernel density estimator by a factor $\epsilon^{\frac{d}{2}}\delta^{\frac{d-1}{2}}$, we shall compensate for it in the following computation.
\begin{equation*}
  \begin{aligned}
    \hat{H}_{\epsilon,\delta}^\alpha f \left( x_{i,r} \right)=\frac{\displaystyle\sum_{j=1}^{N_B}\sum_{s=1}^{N_F}\frac{\hat{K}_{\epsilon,\delta}\left( x_{i,r},x_{j,s} \right)f \left( x_{j,s} \right)}{\hat{p}^{\alpha}_{\epsilon,\delta}\left( x_{i,r} \right)\hat{p}^{\alpha}_{\epsilon,\delta}\left( x_{j,s} \right)}}{\displaystyle\sum_{j=1}^{N_B}\sum_{s=1}^{N_F}\frac{\hat{K}_{\epsilon,\delta}\left( x_{i,r},x_{j,s} \right)}{\hat{p}^{\alpha}_{\epsilon,\delta}\left( x_{i,r} \right)\hat{p}^{\alpha}_{\epsilon,\delta}\left( x_{j,s} \right)}}=\frac{\displaystyle\sum_{j=1}^{N_B}\sum_{s=1}^{N_F}\hat{K}_{\epsilon,\delta}\left( x_{i,r},x_{j,s} \right)\hat{p}^{-\alpha}_{\epsilon,\delta}\left( x_{j,s} \right)f \left( x_{j,s} \right)}{\displaystyle\sum_{j=1}^{N_B}\sum_{s=1}^{N_F}\hat{K}_{\epsilon,\delta}\left( x_{i,r},x_{j,s} \right)\hat{p}^{-\alpha}_{\epsilon,\delta}\left( x_{j,s} \right)},
  \end{aligned}
\end{equation*}
and
\begin{align*}
  &\frac{\displaystyle\sum_{j=1}^{N_B}\sum_{s=1}^{N_F}\hat{K}_{\epsilon,\delta}\left( x_{i,r},x_{j,s} \right)\hat{p}^{-\alpha}_{\epsilon,\delta}\left( x_{j,s} \right)f \left( x_{j,s} \right)}{\displaystyle\sum_{j=1}^{N_B}\sum_{s=1}^{N_F}\hat{K}_{\epsilon,\delta}\left( x_{i,r},x_{j,s} \right)\hat{p}^{-\alpha}_{\epsilon,\delta}\left( x_{j,s} \right)}-\frac{\displaystyle\sum_{j=1}^{N_B}\sum_{s=1}^{N_F}\hat{K}_{\epsilon,\delta}\left( x_{i,r},x_{j,s} \right)\tilde{p}^{-\alpha}_{\epsilon,\delta}\left( x_{j,s} \right)f \left( x_{j,s} \right)}{\displaystyle\sum_{j=1}^{N_B}\sum_{s=1}^{N_F}\hat{K}_{\epsilon,\delta}\left( x_{i,r},x_{j,s} \right)\tilde{p}^{-\alpha}_{\epsilon,\delta}\left( x_{j,s} \right)}\\
  =&\frac{\displaystyle\sum_{j=1}^{N_B}\sum_{s=1}^{N_F}\hat{K}_{\epsilon,\delta}\left( x_{i,r},x_{j,s} \right)\left[N_B^{\alpha}N^{\alpha}_F\hat{p}^{-\alpha}_{\epsilon,\delta}\left( x_{j,s} \right)-\tilde{p}^{-\alpha}_{\epsilon,\delta}\left( x_{j,s} \right)\right]f \left( x_{j,s} \right)}{\displaystyle\sum_{j=1}^{N_B}\sum_{s=1}^{N_F}\hat{K}_{\epsilon,\delta}\left( x_{i,r},x_{j,s} \right)N_B^{\alpha}N_F^{\alpha}\hat{p}^{-\alpha}_{\epsilon,\delta}\left( x_{j,s} \right)}\\
  &+\sum_{j=1}^{N_B}\sum_{s=1}^{N_F}\hat{K}_{\epsilon,\delta}\left( x_{i,r},x_{j,s} \right)\tilde{p}^{-\alpha}_{\epsilon,\delta}\left( x_{j,s} \right)f \left( x_{j,s} \right)\times\\
  &\quad\left[\frac{-\displaystyle \sum_{j=1}^{N_B}\sum_{s=1}^{N_F}\hat{K}_{\epsilon,\delta}\left( x_{i,r},x_{j,s} \right)\left[N_B^{\alpha}N_F^{\alpha}\hat{p}^{-\alpha}_{\epsilon,\delta}\left( x_{j,s} \right)-\tilde{p}^{-\alpha}_{\epsilon,\delta}\left( x_{j,s} \right)\right]}{\displaystyle\left(\sum_{j=1}^{N_B}\sum_{s=1}^{N_F}\hat{K}_{\epsilon,\delta}\left( x_{i,r},x_{j,s} \right)N_B^{\alpha}N_F^{\alpha}\hat{p}^{-\alpha}_{\epsilon,\delta}\left( x_{j,s} \right)\right)\left(\sum_{j=1}^{N_B}\sum_{s=1}^{N_F}\hat{K}_{\epsilon,\delta}\left( x_{i,r},x_{j,s} \right)\tilde{p}^{-\alpha}_{\epsilon,\delta}\left( x_{j,s} \right) \right)} \right]\\
  &=: \left( A \right)+ \left( B \right),
\end{align*}
thus if we can estimate $\left( A \right)$, $\left( B \right)$ by controlling the error
\begin{equation*}
\left[N_B^{\alpha}N^{\alpha}_F\hat{p}^{-\alpha}_{\epsilon,\delta}\left( x_{j,s} \right)-\tilde{p}^{-\alpha}_{\epsilon,\delta}\left( x_{j,s} \right)\right]
\end{equation*}
then it suffices to estimate the variance error caused by
\begin{equation}
\label{eq:better_variance_error}
  \frac{\displaystyle\sum_{j=1}^{N_B}\sum_{s=1}^{N_F}\frac{\hat{K}_{\epsilon,\delta}\left( x_{i,r},x_{j,s} \right)f \left( x_{j,s} \right)}{\tilde{p}^{\alpha}_{\epsilon,\delta}\left( x_{i,r} \right)\tilde{p}^{\alpha}_{\epsilon,\delta}\left( x_{j,s} \right)}}{\displaystyle\sum_{j=1}^{N_B}\sum_{s=1}^{N_F}\frac{\hat{K}_{\epsilon,\delta}\left( x_{i,r},x_{j,s} \right)}{\tilde{p}^{\alpha}_{\epsilon,\delta}\left( x_{i,r} \right)\tilde{p}^{\alpha}_{\epsilon,\delta}\left( x_{j,s} \right)}}.
\end{equation}
Our previous proof for the special case $\alpha=0$ can then be applied to \eqref{eq:better_variance_error}: the only adjustment is to replace the kernel $\hat{K}_{\epsilon,\delta}\left( x,y \right)$ in that proof with the $\alpha$-normalized kernel
\begin{equation*}
  \frac{\tilde{K}_{\epsilon,\delta}\left( x,y \right)}{\tilde{p}_{\epsilon,\delta}^{\alpha}\left( x \right)\tilde{p}_{\epsilon,\delta}^{\alpha}\left( y \right)}.
\end{equation*}

We would like to estimate the tail probability
\begin{equation*}
  \mathbb{P}\left\{ \frac{1}{N_BN_F}\hat{p}_{\epsilon,\delta}\left( x_{j,s} \right)-\tilde{p}_{\epsilon,\delta}\left( x_{j,s} \right)>\beta \right\},
\end{equation*}
but since $\tilde{p}_{\epsilon,\delta}\left( x_{j,s} \right)=O \left( \epsilon^{\frac{d}{2}}\delta^{\frac{d-1}{2}} \right)$, it is not lower bounded away from $0$ as $\epsilon,\delta\rightarrow0$. For this reason, and noting that $\left( A \right)$ and $\left( B \right)$ are invariant if we replace $\hat{p}_{\epsilon,\delta}$, $\tilde{p}_{\epsilon,\delta}$ with $\epsilon^{-\frac{d}{2}}\delta^{-\frac{d-1}{2}}\hat{p}_{\epsilon,\delta}$, $\epsilon^{-\frac{d}{2}}\delta^{-\frac{d-1}{2}}\tilde{p}_{\epsilon,\delta}$, we estimate instead
\begin{equation*}
  \begin{aligned}
    q \left( N_B, N_F, \beta \right)&:=\mathbb{P}\left\{ \frac{1}{N_BN_F}\epsilon^{-\frac{d}{2}}\delta^{-\frac{d-1}{2}}\tilde{p}_{\epsilon,\delta}\left( x_{j,s} \right)-\epsilon^{-\frac{d}{2}}\delta^{-\frac{d-1}{2}}\tilde{p}_{\epsilon,\delta}\left( x_{j,s} \right)>\beta \right\}\\
    &=\mathbb{P}\left\{ \frac{1}{N_BN_F}\tilde{p}_{\epsilon,\delta}\left( x_{j,s} \right)-\tilde{p}_{\epsilon,\delta}\left( x_{j,s} \right)>\epsilon^{\frac{d}{2}}\delta^{\frac{d-1}{2}}\beta \right\}
  \end{aligned}
\end{equation*}
where
\begin{equation*}
  \begin{aligned}
    \hat{p}_{\epsilon,\delta}\left( x_{j,s} \right)&=\sum_{k=1}^{N_B}\sum_{t=1}^{N_F}\hat{K}_{\epsilon,\delta}\left( x_{j,s},x_{k,t} \right),\\
    \tilde{p}_{\epsilon,\delta}\left( x_{j,s} \right)&=\mathbb{E}_1\mathbb{E}_2\left[\tilde{K}_{\epsilon,\delta}\left( x_{j,s},\cdot \right)\right].
  \end{aligned}
\end{equation*}
We would like to apply Lemma~\ref{lem:large_deviation_bound} again. For this purpose, first note that
\begin{equation*}
  \begin{aligned}
    \left|\tilde{K}_{\epsilon,\delta}\left( x_{i,r},x_{j,s} \right)\right| &\leq \left\|K\right\|_{\infty},\quad \left|\mathbb{E}_2\left[\tilde{K}_{\epsilon,\delta}\left( x_{i,r},\cdot \right)\right]\right|\leq C\delta^{\frac{d-1}{2}},\\
    &\left|\mathbb{E}_1\mathbb{E}_2 \left[ \tilde{K}_{\epsilon,\delta}\left( x_{i,r},\cdot \right)\right]\right|\leq C\epsilon^{\frac{d}{2}}\delta^{\frac{d-1}{2}},
  \end{aligned}
\end{equation*}
where
\begin{equation*}
  C = C \left( \left\|K\right\|_{\infty},p_M,p_m,d \right)
\end{equation*}
is some positive constant.
Moreover, direct computation yields
\begin{equation*}
  \mathbb{E}_2^{\xi_j}\left[\tilde{K}_{\epsilon,\delta}\left( x_{i,r},\cdot \right)\right]^2=O \left( \delta^{\frac{d-1}{2}} \right),\quad \mathbb{E}_1\left[\mathbb{E}_2\tilde{K}_{\epsilon,\delta}\left( x_{i,r},\cdot \right)\right]^2=O \left( \epsilon^{\frac{d}{2}}\delta^{d-1} \right).
\end{equation*}
Therefore, by Lemma~\ref{lem:large_deviation_bound}, for $\beta=O \left( \epsilon^2+\delta^2 \right)$,
\begin{align*}
  \begin{aligned}
    q \left( N_B,N_F,\beta \right)&\leq N_B\exp \left\{ -\frac{\left(1-\theta\right)^2N_F\epsilon^d\delta^{d-1}\beta^2}{2C_1\delta^{\frac{d-1}{2}}} \right\}+\exp \left\{ -\frac{\theta^2N_B\epsilon^d\delta^{d-1}\beta^2}{2C_1\epsilon^{\frac{d}{2}}\delta^{d-1}} \right\}\\
    &=N_B\exp \left\{ -\frac{\left(1-\theta\right)^2N_F\epsilon^d\delta^{\frac{d-1}{2}}\beta^2}{2C_1} \right\}+\exp \left\{ -\frac{\theta^2N_B\epsilon^{\frac{d}{2}}\beta^2}{2C_1} \right\}
  \end{aligned}
\end{align*}
where $C_1>0$ is some constant. A simple union bound gives
\begin{equation*}
  \begin{aligned}
    &\mathbb{P}\left(\bigcup_{j,s}\left\{\left|\frac{1}{N_BN_F}\hat{p}_{\epsilon,\delta}\left( x_{j,s} \right)-\tilde{p}_{\epsilon,\delta}\left( x_{j,s} \right) \right|>\epsilon^{\frac{d}{2}}\delta^{\frac{d-1}{2}}\beta\right\}\right)\\
    &\leq N_BN_F \left[ N_B\exp \left\{ -\frac{\left(1-\theta\right)^2N_F\epsilon^d\delta^{\frac{d-1}{2}}\beta^2}{2C_1} \right\}+\exp \left\{ -\frac{\theta^2N_B\epsilon^{\frac{d}{2}}\beta^2}{2C_1} \right\} \right].
  \end{aligned}
\end{equation*}
If we let $\theta=\theta_{*}$, where $\theta_{*}$ is defined in \eqref{eq:seraching_for_theta_star},
\begin{align*}
  \left(1-\theta_{*}\right)^2N_F=\frac{\theta_{*}^2N_B}{\epsilon^{\frac{d}{2}}\delta^{\frac{d-1}{2}}},
\end{align*}
and hence
\begin{equation}
\label{eq:density_uniform_bound}
\begin{aligned}
  &\mathbb{P}\left(\bigcup_{j,s}\left\{\left|\frac{1}{N_BN_F}\hat{p}_{\epsilon,\delta}\left( x_{j,s} \right)-\tilde{p}_{\epsilon,\delta}\left( x_{j,s} \right) \right|>\epsilon^{\frac{d}{2}}\delta^{\frac{d-1}{2}}\beta\right\}\right)\\
  &\leq N_B \left( N_B+1 \right)N_F\exp \left\{ -\frac{\theta_{*}^2N_B\epsilon^{\frac{d}{2}}\beta^2}{2C_1} \right\}.
\end{aligned}
\end{equation}
We are interested in seeing how this bound compares with the bound in \eqref{eq:bound_with_theta_star}. As $N_B,N_F\rightarrow\infty$, as long as \eqref{eq:balance_sampling_number} holds,
\begin{align*}
  &\frac{\displaystyle N_B \left( N_B+1 \right)N_F \exp \left\{ -\frac{\theta_{*}^2N_B\epsilon^{\frac{d}{2}}\beta^2}{2C_1} \right\}}{\displaystyle \left( N_B+1 \right)\exp \left\{ -\frac{\theta_{*}^2N_B\epsilon^{\frac{d}{2}}\beta^2}{C \left( \epsilon+\delta \right)} \right\}}\\
  =&N_B N_F\exp \left\{ -\theta_{*}^2N_B\epsilon^{\frac{d}{2}}\beta^2\left[\frac{1}{2C_1}-\frac{1}{C \left( \epsilon+\delta \right)}\right] \right\}\longrightarrow\infty\quad\textrm{for small $\epsilon,\delta$},
\end{align*}
thus the bound in \eqref{eq:bound_with_theta_star} is asymptotically negligible compared to the bound in \eqref{eq:density_uniform_bound}. This means that when $\alpha\neq 0$ the density estimation in general slows down the convergence rate by a factor $\left( \epsilon+\delta \right)^{\frac{1}{2}}$, which is consistent with the conclusion for standard diffusion maps on manifolds~\cite{SingerWu2012VDM,HeinAudibertVonLuxburg2007}, since $\hat{H}_{\epsilon,\delta}^{\alpha}$ is essentially the heat kernel of the diffusion process (instead of the graph hypoelliptic Laplacian itself). As a consequence of this observation, we know that for probability at least
\begin{equation*}
  1-N_B \left( N_B+1 \right)N_F\exp \left\{ -\frac{\theta_{*}^2N_B\epsilon^{\frac{d}{2}}\beta^2}{2C_1} \right\}
\end{equation*}
we have
\begin{equation*}
  \left|\frac{\displaystyle\sum_{j=1}^{N_B}\sum_{s=1}^{N_F}\frac{\displaystyle\hat{K}_{\epsilon,\delta}\left( x_{i,r},x_{j,s} \right)f \left( x_{j,s} \right)}{\tilde{p}^{\alpha}_{\epsilon,\delta}\left( x_{i,r} \right)\tilde{p}^{\alpha}_{\epsilon,\delta}\left( x_{j,s} \right)}}{\displaystyle \sum_{j=1}^{N_B}\sum_{s=1}^{N_F}\frac{\hat{K}_{\epsilon,\delta}\left( x_{i,r},x_{j,s} \right)\tilde{p}^{-\alpha}_{\epsilon,\delta}\left( x_{j,s} \right)}{\tilde{p}^{\alpha}_{\epsilon,\delta}\left( x_{i,r} \right)\tilde{p}^{\alpha}_{\epsilon,\delta}\left( x_{j,s} \right)}}-\tilde{H}^{\alpha}_{\epsilon,\delta}f \left( x_{i,r} \right)\right|\leq \beta
\end{equation*}
as well as
\begin{equation*}
  \left|\frac{1}{N_BN_F}\hat{p}_{\epsilon,\delta}\left( x_{j,s} \right)-\tilde{p}_{\epsilon,\delta}\left( x_{j,s} \right) \right|\leq \epsilon^{\frac{d}{2}}\delta^{\frac{d-1}{2}}\beta\quad\textrm{for all }1\leq j\leq N_B, 1\leq s\leq N_F.
\end{equation*}
Note that by our assumption
\begin{equation*}
  0<p_m\leq p \left( x,v \right)\leq p_M<\infty\quad\textrm{for all }\left( x,v \right)\in UTM
\end{equation*}
there exists constants $C_1,C_2$ such that
\begin{equation*}
  0<C_1 < \epsilon^{-\frac{d}{2}}\delta^{-\frac{d-1}{2}}\tilde{p}_{\epsilon,\delta}\left( x_{j,s} \right) < C_2<\infty.
\end{equation*}
For sufficiently small $\beta$, these bounds also apply to $N_B^{-1}N_F^{-1}\hat{p}_{\epsilon,\delta}\left( x_{j,s} \right)$ with high probability:
\begin{equation*}
  0<C_1 < \frac{1}{N_BN_F}\epsilon^{-\frac{d}{2}}\delta^{-\frac{d-1}{2}}\hat{p}_{\epsilon,\delta}\left( x_{j,s} \right) < C_2<\infty.
\end{equation*}
More specifically, we have
\begin{align*}
  &0<\frac{1}{C_2\epsilon^{\frac{d}{2}}\delta^{\frac{d-1}{2}}}<N_BN_F\hat{p}_{\epsilon,\delta}^{-1}\left( x_{j,s} \right)<\frac{1}{C_1\epsilon^{\frac{d}{2}}\delta^{\frac{d-1}{2}}}<\infty,\\
  &0<\frac{1}{C_2\epsilon^{\frac{d}{2}}\delta^{\frac{d-1}{2}}}<\tilde{p}_{\epsilon,\delta}^{-1}\left( x_{j,s} \right)<\frac{1}{C_1\epsilon^{\frac{d}{2}}\delta^{\frac{d-1}{2}}}<\infty,
\end{align*}
and
\begin{equation*}
  \left| N_BN_F\hat{p}_{\epsilon,\delta}^{-1}\left( x_{j,s} \right)-\tilde{p}_{\epsilon,\delta}^{-1}\left( x_{j,s} \right) \right|\leq \epsilon^{\frac{d}{2}}\delta^{\frac{d-1}{2}}\beta\cdot \frac{1}{C_1^2\epsilon^d\delta^{d-1}}=\frac{\beta}{C_1^2\epsilon^{\frac{d}{2}}\delta^{\frac{d-1}{2}}}.
\end{equation*}
The errors $\left( A \right)$, $\left( B \right)$ can thus be bounded as
\begin{align*}
  &\left|\left( A \right)\right|\leq C_2^{\alpha}\epsilon^{\frac{\alpha d}{2}}\delta^{\frac{\alpha \left( d-1 \right)}{2}}\left\|f\right\|_{\infty}\cdot \alpha \left(\frac{2}{C_2\epsilon^{\frac{d}{2}}\delta^{\frac{d-1}{2}}}\right)^{\alpha-1}\frac{\beta}{C_1^2\epsilon^{\frac{d}{2}}\delta^{\frac{d-1}{2}}}=\frac{2^{\alpha-1}\alpha C_2 \left\|f\right\|_{\infty}}{C_1^2}\beta,\\
  &\left|\left( B \right)\right|\leq \frac{C_2^{2\alpha}\epsilon^{\alpha d}\delta^{\alpha \left( d-1 \right)}}{C_1^{\alpha}\epsilon^{\frac{\alpha d}{2}}\delta^{\frac{\alpha \left( d-1 \right)}{2}}}\left\|f\right\|_{\infty}\cdot \alpha \left(\frac{2}{C_2\epsilon^{\frac{d}{2}}\delta^{\frac{d-1}{2}}}\right)^{\alpha-1}\frac{\beta}{C_1^2\epsilon^{\frac{d}{2}}\delta^{\frac{d-1}{2}}}=\frac{2^{\alpha-1}\alpha C_2^{\alpha+1} \left\|f\right\|_{\infty}}{C_1^{\alpha+2}}\beta.
\end{align*}
Since $C_1,C_2$ only depends on the kernel function $K$, the dimension $d$, and $p_m,p_M$, these bounds ensures that
\begin{equation*}
  \left| \hat{H}_{\epsilon,\delta}^{\alpha}f \left( x_{i,r} \right)-\tilde{H}_{\epsilon,\delta}^{\alpha}f \left( x_{i,r} \right) \right| < C\beta
\end{equation*}
with probability at least
\begin{equation*}
  1-N_B \left( N_B+1 \right)N_F\exp \left\{ -\frac{\theta_{*}^2N_B\epsilon^{\frac{d}{2}}\beta^2}{2C_1} \right\},
\end{equation*}
where constants $C,C_1$ only depends on the kernel function $K$, the dimension $d$, and $p_m,p_M$. This establishes the conclusion for all $\alpha\in \left[ 0,1 \right]$.
\end{proof}

\subsubsection{Sampling from Empirical Tangent Spaces}
\label{sec:proor-theor-noise}

The following two lemmas from \cite{SingerWu2012VDM} provide estimates for the approximation error in the local PCA step. We adapted these lemmas to our notation; note that the statements are more compact than their original form since we assume $M$ is closed.

\begin{lemma}
\label{lem:local_PCA}
  Suppose $K_{\mathrm{PCA}}\in C^2 \left( \left[ 0,1 \right] \right)$. If $\epsilon_{\mathrm{PCA}}=O \left( N_B^{-\frac{2}{d+2}} \right)$, then, with high probability, the columns of the $D\times d$ matrix $O_i$ determined by local PCA form an orthonormal basis to a $d$-dimensional subspace of $\mathbb{R}^D$ that deviates from $\iota_{*}T_{x_i}M$ by $O \left( \epsilon_{\mathrm{PCA}}^{\frac{3}{2}} \right)$, in the following sense:
  \begin{equation}
    \label{eq:localPCAerror}
    \min_{O\in O \left( d \right)}\|O_i^{\top}\Theta_i-O\|_{\mathrm{HS}}=O \left( \epsilon_{\mathrm{PCA}}^{\frac{3}{2}} \right)=O \left( N_B^{-\frac{3}{d+2}} \right),
  \end{equation}
where $\Theta_i$ is a $D\times d$ matrix whose columns form an orthonormal basis to $\iota_{*}T_{x_i}M$. Let the minimizer if \eqref{eq:localPCAerror} be
\begin{equation}
  \label{eq:localPCAminimizer}
  \hat{O}_i=\argmin_{O\in O \left( d \right)}\|O_i^{\top}\Theta_i-O\|_{\mathrm{F}},
\end{equation}
and denote by $Q_i$ the $D\times d$ matrix
\begin{equation}
  \label{eq:localPCAbasis}
  Q_i:=\Theta_i\hat{O}_i^{\top},
\end{equation}
The columns of $Q_i$ form an orthonormal basis to $\iota_{*}T_{x_i}M$, and
\begin{equation}
  \label{eq:localPCAbasiserror}
  \|O_i-Q_i\|_{\mathrm{F}}=O \left( \epsilon_{\mathrm{PCA}} \right),
\end{equation}
where $\left\|\cdot\right\|_{\mathrm{F}}$ is the matrix Frobenius norm.
\end{lemma}
\begin{proof}
  See \cite[Lemma B.1]{SingerWu2012VDM}.
\end{proof}

\begin{lemma}
\label{lem:approximate_parallel_transport}
  Consider points $x_i,x_j\in M$ such that the geodesic distance between them is $O \left( \epsilon^{\frac{1}{2}} \right)$. For $\epsilon_{\mathrm{PCA}}=O \left( N_B^{-\frac{2}{d+2}} \right)$, with high probability, $O_{ij}$ approximates $P_{x_i,x_j}$ in the following sense:
  \begin{equation}
    \label{eq:approximate_parallel_transport}
    O_{ij}\bar{X}_j=\left( \langle\iota_{*}P_{x_i,x_j}X \left( x_j \right),u_l \left( x_i \right)\rangle \right)_{l=1}^d+O \left( \epsilon_{\mathrm{PCA}}^{\frac{1}{2}}+\epsilon^{\frac{3}{2}} \right),\quad\textrm{for all }X\in \Gamma \left( M,TM \right),
  \end{equation}
where $\left\{ u_l \left( x_i \right) \right\}_{l=1}^d$ is an orthonormal set determined by local PCA, and
\begin{equation*}
  \bar{X}_i\equiv \left( \langle\iota_{*}X \left( x_i \right),u_l \left( x_i \right)\rangle \right)_{l=1}^d\in\mathbb{R}^d.
\end{equation*}
\end{lemma}
\begin{proof}
  See \cite[Theorem B.2]{SingerWu2012VDM}.
\end{proof}

\begin{proof}[Proof of Theorem{\rm ~\ref{thm:utm_finite_sampling_noise}}]
By Definition~\ref{defn:utm_finite_sampling_noise}~\eqref{item:40},
\begin{equation*}
  O_{ji}c_{i,r}=\frac{O_{ji}B_i^{\top}\overline{\tau}_{i,r}}{\left\|B_i^{\top}\overline{\tau}_{i,r}\right\|}.
\end{equation*}
By Lemma~\ref{lem:approximate_parallel_transport},
\begin{equation*}
  O_{ji}B_i^{\top}\overline{\tau}_{i,r}=B_j^{\top}\left( P_{\xi_j,\xi_i}\overline{\tau}_{i,r} \right)+O \left( \epsilon_{\mathrm{PCA}}^{\frac{1}{2}}+\epsilon^{\frac{3}{2}} \right),
\end{equation*}
thus
\begin{align*}
  \frac{O_{ji}B_i^{\top}\overline{\tau}_{i,r}}{\left\|B_i^{\top}\overline{\tau}_{i,r}\right\|}=\frac{B_j^{\top}\left( P_{\xi_j,\xi_i}\overline{\tau}_{i,r} \right)}{\left\|B_i^{\top}\overline{\tau}_{i,r}\right\|}+O \left( \epsilon_{\mathrm{PCA}}^{\frac{1}{2}}+\epsilon^{\frac{3}{2}} \right),
\end{align*}
where we used
\begin{align*}
  \left| \left\|B_i^{\top}\overline{\tau}_{i,r}\right\|_{\mathrm{F}}-1 \right|&=\left| \left\|B_i^{\top}\overline{\tau}_{i,r}\right\|_{\mathrm{F}}-\left\|Q_i^{\top}\overline{\tau}_{i,r}\right\|_{\mathrm{F}}\right|\leq \left\|B_i^{\top}\overline{\tau}_{i,r}-Q_i^{\top}\overline{\tau}_{i,r}\right\|_{\mathrm{F}}\\
  &\leq \left\|B_i^{\top}-Q_i^{\top}\right\|_{\mathrm{F}}=O \left( \epsilon_{\mathrm{PCA}} \right)
\end{align*}
and
\begin{align*}
  \left\|B_j^{\top}\left( P_{\xi_j,\xi_i}\overline{\tau}_{i,r} \right)\right\|\leq \left\|P_{\xi_j,\xi_i}\overline{\tau}_{i,r}\right\|=1.
\end{align*}
Therefore,
\begin{align*}
  O_{ji}c_{i,r}-c_{j,s}&=\frac{O_{ji}B_i^{\top}\overline{\tau}_{i,r}}{\left\|B_i^{\top}\overline{\tau}_{i,r}\right\|}-\frac{B_j^{\top}\overline{\tau}_{j,s}}{\left\|B_j^{\top}\overline{\tau}_{j,s}\right\|}\\
  &=\frac{B_j^{\top}\left( P_{\xi_j,\xi_i}\overline{\tau}_{i,r} \right)}{\left\|B_i^{\top}\overline{\tau}_{i,r}\right\|}+O \left( \epsilon_{\mathrm{PCA}}^{\frac{1}{2}}+\epsilon^{\frac{3}{2}} \right)-\frac{B_j^{\top}\overline{\tau}_{j,s}}{\left\|B_j^{\top}\overline{\tau}_{j,s}\right\|}\\
  &=B_j^{\top}\left( P_{\xi_j,\xi_i}\overline{\tau}_{i,r} \right)-B_j^{\top}\overline{\tau}_{j,s}+O \left( \epsilon_{\mathrm{PCA}}^{\frac{1}{2}}+\epsilon^{\frac{3}{2}} \right)\\
  &=P_{\xi_j,\xi_i}\overline{\tau}_{i,r}-\overline{\tau}_{j,s}+O \left( \epsilon_{\mathrm{PCA}}^{\frac{1}{2}}+\epsilon^{\frac{3}{2}} \right),
\end{align*}
and
\begin{align*}
  \left|\left\|O_{ji}c_{i,r}-c_{j,s}\right\|^2-\left\|P_{\xi_j,\xi_i}\overline{\tau}_{i,r}-\overline{\tau}_{j,s}\right\|^2\right|&\leq 4\left\|\left(O_{ji}c_{i,r}-c_{j,s}\right)-\left(P_{\xi_j,\xi_i}\overline{\tau}_{i,r}-\overline{\tau}_{j,s} \right)\right\|\\
  &=O \left( \epsilon_{\mathrm{PCA}}^{\frac{1}{2}}+\epsilon^{\frac{3}{2}} \right).
\end{align*}
Thus a Taylor expansion for $K$ at point
\begin{align*}
  \left( \frac{\left\|\xi_i-\xi_j\right\|^2}{\epsilon},\frac{\left\|O_{ji}c_{i,r}-c_{j,s}\right\|^2}{\delta} \right)
\end{align*}
gives
\begin{align*}
  &K\left( \frac{\left\|\xi_i-\xi_j\right\|^2}{\epsilon},\frac{\left\|O_{ji}c_{i,r}-c_{j,s}\right\|^2}{\delta} \right)\\
  &=K\left( \frac{\left\|\xi_i-\xi_j\right\|^2}{\epsilon},\frac{\left\|P_{\xi_j,\xi_i}\overline{\tau}_{i,r}-\overline{\tau}_{j,s}+O \left( \epsilon_{\mathrm{PCA}}^{\frac{1}{2}}+\epsilon^{\frac{3}{2}} \right)\right\|^2}{\delta} \right)\\
  &=K \left(\frac{\left\|\xi_i-\xi_j\right\|^2}{\epsilon},\frac{\left\|P_{\xi_j,\xi_i}\overline{\tau}_{i,r}-\overline{\tau}_{j,s}\right\|^2}{\delta} \right)\\
  &\quad+\partial_2K\left(\frac{\left\|\xi_i-\xi_j\right\|^2}{\epsilon},\frac{\left\|P_{\xi_j,\xi_i}\overline{\tau}_{i,r}-\overline{\tau}_{j,s}\right\|^2}{\delta} \right)\cdot \frac{O \left( \epsilon_{\mathrm{PCA}}^{\frac{1}{2}}+\epsilon^{\frac{3}{2}} \right)}{\delta}.
\end{align*}
For any function $g\in C^{\infty}\left( UTM \right)$, this leads to
\begin{align*}
  &\int_{UTM}\mathscr{K}_{\epsilon,\delta}\left( \overline{\tau}_{i,r},\eta \right)g \left( \eta \right)d\Theta \left( \eta \right)\\
  &=\int_{UTM}\tilde{K}_{\epsilon,\delta}\left( \overline{\tau}_{i,r},\eta \right)g \left( \eta \right)d\Theta \left( \eta \right)\\
  &+\frac{O \left( \epsilon_{\mathrm{PCA}}^{\frac{1}{2}}+\epsilon^{\frac{3}{2}} \right)}{\delta}\int_M\!\int_{S_y}\partial_2K\left(\frac{\left\|\xi_i-y\right\|^2}{\epsilon},\frac{\left\|P_{\xi_j,\xi_i}\overline{\tau}_{i,r}-w\right\|^2}{\delta} \right)g \left( y,w \right)d\sigma_y \left( w \right)d\mathrm{vol}_M \left( y \right)\\
  &=\int_{UTM}\tilde{K}_{\epsilon,\delta}\left( \overline{\tau}_{i,r},\eta \right)g \left( \eta \right)d\Theta \left( \eta \right)+\epsilon^{\frac{d}{2}}\delta^{\frac{d-1}{2}-1}O \left( \epsilon_{\mathrm{PCA}}^{\frac{1}{2}}+\epsilon^{\frac{3}{2}} \right).
\end{align*}
Following the notation used in the proof of Theorem~\ref{thm:utm_finite_sampling_noiseless}, by the law of large numbers
\begin{align*}
  &\lim_{N_B\rightarrow\infty}\lim_{N_F\rightarrow\infty}\frac{1}{N_BN_F}\hat{q}_{\epsilon,\delta}\left( \overline{\tau}_{i,r} \right)=\mathbb{E}_1\mathbb{E}_2\left[\mathscr{K}_{\epsilon,\delta}\left(\overline{\tau}_{i,r},\cdot \right)\right]\\
  &=\mathbb{E}_1\mathbb{E}_2\left[\tilde{K}_{\epsilon,\delta}\left(\overline{\tau}_{i,r},\cdot \right)\right]+\epsilon^{\frac{d}{2}}\delta^{\frac{d-1}{2}-1}O \left( \epsilon_{\mathrm{PCA}}^{\frac{1}{2}}+\epsilon^{\frac{3}{2}} \right),
\end{align*}
and hence we expect $\mathscr{H}_{\epsilon,\delta}^{\alpha}f \left( \overline{\tau}_{i,r} \right)$ to converge to
\begin{align*}
    &\tilde{H}_{\epsilon,\delta}^{\alpha}f \left( \overline{\tau}_{i,r} \right)+O \left( \delta^{-1}\left( \epsilon_{\mathrm{PCA}}^{\frac{1}{2}}+\epsilon^{\frac{3}{2}} \right) \right)\\
    &= f \left( \overline{\tau}_{i,r} \right)+\epsilon \frac{m_{21}}{2m_0}\left[\frac{\Delta_S^H\left[fp^{1-\alpha}\right]\left( \overline{\tau}_{i,r} \right)}{p^{1-\alpha}\left( \overline{\tau}_{i,r} \right)}-f \left( \overline{\tau}_{i,r} \right) \frac{\Delta_S^Hp^{1-\alpha}\left( \overline{\tau}_{i,r} \right)}{p^{1-\alpha}\left( \overline{\tau}_{i,r} \right)}  \right]\\
    &\quad+\delta \frac{m_{22}}{2m_0}\left[\frac{\Delta_S^V\left[fp^{1-\alpha}\right]\left( \overline{\tau}_{i,r} \right)}{p^{1-\alpha}\left( \overline{\tau}_{i,r} \right)}-f \left( \overline{\tau}_{i,r} \right) \frac{\Delta_S^Vp^{1-\alpha}\left( \overline{\tau}_{i,r} \right)}{p^{1-\alpha}\left( \overline{\tau}_{i,r} \right)}  \right]\\
    &\quad+O \left( \epsilon^2+\delta^2\right)+O \left( \delta^{-1}\left( \epsilon_{\mathrm{PCA}}^{\frac{1}{2}}+\epsilon^{\frac{3}{2}} \right) \right).
\end{align*}
In fact, noting that
\begin{align*}
  \frac{1}{\epsilon^{\frac{d}{2}}\delta^{\frac{d-1}{2}}N_BN_F}\hat{q}_{\epsilon,\delta}\left( \overline{\tau}_{i,r} \right)&=\frac{1}{\epsilon^{\frac{d}{2}}\delta^{\frac{d-1}{2}}N_BN_F}\sum_{j=1}^{N_F}\sum_{s=1}^{N_B}\mathscr{K}_{\epsilon,\delta} \left( \overline{\tau}_{i,r},\overline{\tau}_{j,s} \right)\\
   &=\frac{1}{\epsilon^{\frac{d}{2}}\delta^{\frac{d-1}{2}}N_BN_F}K_{\epsilon,\delta}\left( \overline{\tau}_{i,r},\overline{\tau}_{j,s} \right)+\frac{O \left( \delta^{-1}\left( \epsilon_{\mathrm{PCA}}^{\frac{1}{2}}+\epsilon^{\frac{3}{2}} \right) \right)}{\epsilon^{\frac{d}{2}}\delta^{\frac{d-1}{2}}}\\
   &=\frac{1}{\epsilon^{\frac{d}{2}}\delta^{\frac{d-1}{2}}N_BN_F}\hat{p}_{\epsilon,\delta}\left( \overline{\tau}_{i,r},\overline{\tau}_{j,s} \right)+\frac{O \left( \delta^{-1}\left( \epsilon_{\mathrm{PCA}}^{\frac{1}{2}}+\epsilon^{\frac{3}{2}} \right) \right)}{\epsilon^{\frac{d}{2}}\delta^{\frac{d-1}{2}}},
\end{align*}
we have
\begin{align*}
  &\epsilon^{\alpha d}\delta^{\alpha\left(d-1\right)}N_B^{2\alpha}N_F^{2\alpha}\mathscr{K}_{\epsilon,\delta}^{\alpha} \left( \overline{\tau}_{i,r}, \overline{\tau}_{j,s} \right)\\
  & = \frac{\mathscr{K}_{\epsilon,\delta}\left( \overline{\tau}_{i,r}, \overline{\tau}_{j,s} \right)}{\displaystyle \left(\frac{1}{\epsilon^{\frac{d}{2}}\delta^{\frac{d-1}{2}}N_BN_F}\hat{q}_{\epsilon,\delta}\left( \overline{\tau}_{i,r}\right)\right)^{\alpha}\left(\frac{1}{\epsilon^{\frac{d}{2}}\delta^{\frac{d-1}{2}}N_BN_F}\hat{q}_{\epsilon,\delta}\left( \overline{\tau}_{j,s}\right)\right)^{\alpha}}\\
  &=\frac{\displaystyle K_{\epsilon,\delta}\left( \overline{\tau}_{i,r},\overline{\tau}_{j,s} \right)+O \left( \delta^{-1}\left( \epsilon_{\mathrm{PCA}}^{\frac{1}{2}}+\epsilon^{\frac{3}{2}} \right) \right)}{\displaystyle \left(\frac{1}{\epsilon^{\frac{d}{2}}\delta^{\frac{d-1}{2}}N_BN_F}\hat{p}_{\epsilon,\delta}\left( \overline{\tau}_{i,r}\right)\right)^{\alpha}\left(\frac{1}{\epsilon^{\frac{d}{2}}\delta^{\frac{d-1}{2}}N_BN_F}\hat{p}_{\epsilon,\delta}\left( \overline{\tau}_{j,s}\right)\right)^{\alpha}+O \left( \delta^{-1}\left( \epsilon_{\mathrm{PCA}}^{\frac{1}{2}}+\epsilon^{\frac{3}{2}} \right) \right)}\\
  &=\epsilon^{\alpha d}\delta^{\alpha\left(d-1\right)}N_B^{2\alpha}N_F^{2\alpha}K_{\epsilon,\delta}^{\alpha}\left( \overline{\tau}_{i,r},\overline{\tau}_{j,s} \right)+O \left( \delta^{-1}\left( \epsilon_{\mathrm{PCA}}^{\frac{1}{2}}+\epsilon^{\frac{3}{2}} \right) \right).
\end{align*}
Consequently, $\mathscr{H}_{\epsilon,\delta}^{\alpha}f \left( \overline{\tau}_{i,r} \right)$ is approximately $\hat{H}_{\epsilon,\delta}^{\alpha}f \left( \overline{\tau}_{i,r} \right)$:
\begin{align*}
  &\mathscr{H}_{\epsilon,\delta}^{\alpha}f \left( \overline{\tau}_{i,r} \right)=\frac{\displaystyle \sum_{j=1}^{N_B}\sum_{s=1}^{N_F}\mathscr{K}_{\epsilon,\delta}^{\alpha} \left( \overline{\tau}_{i,r}, \overline{\tau}_{j,s}\right)f \left( \overline{\tau}_{j,s} \right)}{\displaystyle \sum_{j=1}^{N_B}\sum_{s=1}^{N_F}\mathscr{K}_{\epsilon,\delta}^{\alpha} \left( \overline{\tau}_{i,r}, \overline{\tau}_{j,s}\right)}\\
  &=\frac{\displaystyle \frac{1}{N_BN_F}\sum_{j=1}^{N_B}\sum_{s=1}^{N_F}\epsilon^{\alpha d}\delta^{\alpha\left(d-1\right)}N_B^{2\alpha}N_F^{2\alpha}\mathscr{K}_{\epsilon,\delta}^{\alpha} \left( \overline{\tau}_{i,r}, \overline{\tau}_{j,s}\right)f \left( \overline{\tau}_{j,s} \right)}{\displaystyle \frac{1}{N_BN_F}\sum_{j=1}^{N_B}\sum_{s=1}^{N_F}\epsilon^{\alpha d}\delta^{\alpha\left(d-1\right)}N_B^{2\alpha}N_F^{2\alpha}\mathscr{K}_{\epsilon,\delta}^{\alpha} \left( \overline{\tau}_{i,r}, \overline{\tau}_{j,s}\right)}\\
  &=\frac{\displaystyle \frac{1}{N_BN_F}\sum_{j=1}^{N_B}\sum_{s=1}^{N_F}\epsilon^{\alpha d}\delta^{\alpha\left(d-1\right)}N_B^{2\alpha}N_F^{2\alpha}\tilde{K}_{\epsilon,\delta}^{\alpha} \left( \overline{\tau}_{i,r}, \overline{\tau}_{j,s}\right)f \left( \overline{\tau}_{j,s} \right)+O \left( \delta^{-1}\left( \epsilon_{\mathrm{PCA}}^{\frac{1}{2}}+\epsilon^{\frac{3}{2}} \right) \right)}{\displaystyle \frac{1}{N_BN_F}\sum_{j=1}^{N_B}\sum_{s=1}^{N_F}\epsilon^{\alpha d}\delta^{\alpha\left(d-1\right)}N_B^{2\alpha}N_F^{2\alpha}\tilde{K}_{\epsilon,\delta}^{\alpha} \left( \overline{\tau}_{i,r}, \overline{\tau}_{j,s}\right)+O \left( \delta^{-1}\left( \epsilon_{\mathrm{PCA}}^{\frac{1}{2}}+\epsilon^{\frac{3}{2}} \right) \right)}\\
  &=\frac{\displaystyle \frac{1}{N_BN_F}\sum_{j=1}^{N_B}\sum_{s=1}^{N_F}\epsilon^{\alpha d}\delta^{\alpha\left(d-1\right)}N_B^{2\alpha}N_F^{2\alpha}\tilde{K}_{\epsilon,\delta}^{\alpha} \left( \overline{\tau}_{i,r}, \overline{\tau}_{j,s}\right)f \left( \overline{\tau}_{j,s} \right)}{\displaystyle \frac{1}{N_BN_F}\sum_{j=1}^{N_B}\sum_{s=1}^{N_F}\epsilon^{\alpha d}\delta^{\alpha\left(d-1\right)}N_B^{2\alpha}N_F^{2\alpha}\tilde{K}_{\epsilon,\delta}^{\alpha} \left( \overline{\tau}_{i,r}, \overline{\tau}_{j,s}\right)}+O \left( \delta^{-1}\left( \epsilon_{\mathrm{PCA}}^{\frac{1}{2}}+\epsilon^{\frac{3}{2}} \right) \right)\\
  &=\frac{\displaystyle \sum_{j=1}^{N_B}\sum_{s=1}^{N_F}\tilde{K}_{\epsilon,\delta}^{\alpha} \left( \overline{\tau}_{i,r}, \overline{\tau}_{j,s}\right)f \left( \overline{\tau}_{j,s} \right)}{\displaystyle \sum_{j=1}^{N_B}\sum_{s=1}^{N_F}\tilde{K}_{\epsilon,\delta}^{\alpha} \left( \overline{\tau}_{i,r}, \overline{\tau}_{j,s}\right)}+O \left( \delta^{-1}\left( \epsilon_{\mathrm{PCA}}^{\frac{1}{2}}+\epsilon^{\frac{3}{2}} \right) \right)\\
  &=\hat{H}_{\epsilon,\delta}^{\alpha}f \left( \overline{\tau}_{i,r} \right)+O \left( \delta^{-1}\left( \epsilon_{\mathrm{PCA}}^{\frac{1}{2}}+\epsilon^{\frac{3}{2}} \right) \right).
\end{align*}
Therefore, under the assumption that
\begin{align*}
  \delta^{-1}\left( \epsilon_{\mathrm{PCA}}^{\frac{1}{2}}+\epsilon^{\frac{3}{2}} \right)\longrightarrow 0\quad \textrm{as $\epsilon\rightarrow0$,}
\end{align*}
we can apply Theorem~\ref{thm:utm_finite_sampling_noiseless}. This completes the whole proof.
\end{proof}
